\documentclass[a4paper]{article}
\usepackage[margin=3cm]{geometry}
\usepackage{amsmath,amssymb,amsthm, amsfonts}
\usepackage[colorlinks=true,linkcolor=blue,citecolor=blue]{hyperref}
\usepackage{graphicx} 
\usepackage[subpreambles=true]{standalone} 
\usepackage{import} 
\usepackage{tikz}
\usepackage{tikz-cd}
\usepackage{mathrsfs}
\usepackage{relsize}

\newtheorem{counter}{Counter}
\newtheorem{theorem}[counter]{Theorem}
\newtheorem{lemma}[counter]{Lemma}

\newtheorem{proposition}[counter]{Proposition}
\theoremstyle{definition}
\newtheorem{definition}[counter]{Definition}
\newtheorem{remark}[counter]{Remark}
\newtheorem{example}[counter]{Example}

\numberwithin{counter}{section}
\numberwithin{equation}{section}

\newcommand{\unit}{{\mathrm{1}\mkern-4mu{\mathchoice{}{}{\mskip-0.5mu}{\mskip-1mu}}\mathrm{l}}}

\newcommand{\chr}{{\mathsf{c}}}
\newcommand{\gm}{{\mathsf{g}}}
\newcommand\mt{\operatorname{t}}
\newcommand\oT{\overline{T}}
\newcommand\D{\mathbb{D}}

\newcommand\Aa{{\mathscr{A}}}
\renewcommand\AA{\mathcal{A}}
\renewcommand\SS{\mathcal{S}}
\newcommand\CC{\mathcal{C}}
\newcommand\DD{\mathcal{D}}
\newcommand\EE{\mathcal{E}}

\newcommand\II{\mathcal{I}}

\newcommand\RR{\mathcal{R}}

\newcommand\ZZ{\mathcal{Z}}
\newcommand\inj{\hookrightarrow}
\newcommand{\Proj}{\operatorname{\textsc{\textbf{Proj}}}}
\newcommand{\Cob}{\operatorname{\textsc{\textbf{Cob}}}}
\DeclareMathOperator{\WRT}{WRT}
\DeclareMathOperator{\DGGPR}{DGGPR}
\DeclareMathOperator{\cob}{Cob}
\DeclareMathOperator{\Vect}{Vect}
\DeclareMathOperator{\Sk}{Sk}

\DeclareMathOperator{\SkFun}{\underline{Sk}}

\DeclareMathOperator{\SkCat}{SkCat}

\DeclareMathOperator{\Bimod}{Bimod}
\DeclareMathOperator{\Cat}{Cat}

\DeclareMathOperator{\End}{End}
\DeclareMathOperator{\Hom}{Hom}
\DeclareMathOperator{\id}{id}

\DeclareMathOperator{\Fun}{Fun}

\DeclareMathOperator{\Triv}{Triv}

\title{Non-semisimple WRT at the boundary of Crane--Yetter}
\author{Benjamin Ha\"ioun }
\date{\today}

\begin{document}

\maketitle

\begin{abstract}
    We prove the slogan, promoted by Walker and Freed--Teleman twenty years ago, that \begin{center}
        ``The Witten–Reshetikhin–Turaev 3-TQFT is a boundary condition\\ for the Crane–Yetter 4-TQFT"
    \end{center}  and generalize it to the non-semisimple case following ideas of Jordan, Reutter and Walker.
    
    To achieve this, we prove that the Crane–Yetter 4-TQFT and its non-semisimple version \cite{CGHP} are once-extended TQFTs, using the main result of \cite{HaiounHandle}.
We define a boundary condition, partially defined in the non-semisimple case, for this 4D theory. When the ribbon category used is modular, possibly non-semisimple, we check that the composition of this boundary condition with the values of the 4-TQFT on bounding manifolds reconstructs the Witten–Reshetikhin–Turaev 3-TQFTs and their non-semisimple versions \cite{DGGPR}, in a sense that we make precise.
\end{abstract}
\tableofcontents

\section{Introduction} 
The 3-dimensional Topological Quantum Field Theories predicted by Witten \cite{WittenJonesPol} and constructed mathematically by Reshetikhin and Turaev \cite{ReshetikhinTuraev, TuraevBook} are not TQFTs in the usual sense: they have an anomaly. This can be observed for example in the fact that they induce projective representations of the Mapping Class Groups of surfaces, instead of linear ones. 

This anomaly has been extensively studied by Walker \cite{WalkerOnWitten, WalkerNotes} and he observed that it can be described by a TQFT of dimension 4, the Crane--Yetter theory \cite{CraneYetter}. Along with Freed and Teleman \cite{FreedSlides}, see also \cite{FreedHopkinsLurieTeleman}, they reformulate this observation in the following slogan:
\begin{quote} \it \centering
   The Witten–Reshetikhin–Turaev 3-TQFT is a boundary condition for the Crane–Yetter 4-TQFT
\end{quote}
This idea was communicated in talks and the unfinished notes \cite{WalkerNotes} and was not formalized at the time. Aspects of this story have since been formalized \cite{BarrettFariaMartinGarciaIslas_ObsTVandCY, ThamThesis} but we are still missing a complete picture obtaining the whole TQFT defined by Witten--Reshetikhin--Turaev (and not simply invariants of 3-manifolds, or state spaces) by a single object: a boundary condition to Crane--Yetter. 

In order to formalize the above slogan, we propose to address the following three highlighted points. First:
\begin{quote} 
  \it\centering What does it mean to be a boundary condition to a given TQFT?
\end{quote}
This has been answered by \cite{FreedTeleman, JFS}, see Definition \ref{def:BondaryCond}. From these works we learn that in order for this notion to be interesting, we need the given TQFT to be (at least) once-extended. Hence, our first step is to prove that Crane--Yetter is a once-extended TQFT. This fact is a well-known folklore result \cite{WalkerNotes, ThamThesis, KirillovTham4DTQFTskein} but has never been formalized in the framework used to define boundary conditions, i.e. as a symmetric monoidal 2-functor out of a bicategory of cobordisms. This is done in Section \ref{Sec_ExtendedCY}.
A boundary condition is then an oplax natural transformation from the trivial TQFT to the given one. We are then prepared to:
\begin{quote} 
  \it\centering  Construct a boundary condition to the Crane--Yetter 4-TQFT.
\end{quote}
Several key ingredients are given in \cite{WalkerNotes, ThamThesis} but such a construction has not been made formal in the above framework. This is done in Section \ref{Sec_BoundaryCond}. Finally we must answer:
\begin{quote} 
  \it\centering  In what sense does this boundary condition recover the WRT 3-TQFT?
\end{quote}
The WRT theory as defined in \cite{ReshetikhinTuraev, TuraevBook} is not the same kind of object as a boundary condition. As argued by Freed \cite{FreedAnomaly}, the boundary condition is the ``right object" describing the TQFT with anomaly, at least from a physical perspective, but we must explain in what sense they model the same thing. This is achieved in two different ways in Sections \ref{Sec_WRT} and \ref{Sec_projectiveTQFT}.

In this paper, we treat directly a non-semisimple generalization of this story and prove that the \cite{DGGPR} 3-TQFT is a boundary condition to the \cite{CGHP} 4-TQFT. This generalization, based on the non-semisimple skein theory of \cite{GKPgeneralizedTraceModdim, CGP, DGGPR}, has been developed in not yet published work by Reutter--Walker (see the slides \cite{WalkerSlides} and the forthcoming paper \cite{ReutterWalkerPaper}) and Jordan--Reutter--Walker (see the slides \cite{ReutterSlides}). 

\subsection{Context}
\paragraph{TQFTs}
Topological Quantum Field Theories are a particular kind of physical theories whose physical states are described by some quantum fields living in the ambient space, and whose time evolution does not depend on the metric, but only on the topology of the space-time. They were modeled mathematically by Atiyah \cite{Atiyah} and Segal \cite{SegalCFT}. In modern language, an $(n+1)$-TQFT is a symmetric monoidal functor $$\ZZ:\cob_{n+1} \to \Vect\ .$$
It assigns a vector space $\ZZ(M)$, called state space, to an oriented closed $n$-manifold $M$ and a linear map $\ZZ(W)$, called correlation function, to an $(n+1)$-cobordism $W$ which one may think of as time evolution along $W$.

\paragraph{Anomalies}
The definition of Atiyah and Segal does not model one aspect of the physical theory, namely that a vector $\psi \in \ZZ(M)$ represents the same physical state as $\lambda\psi$ for any scalar $\lambda\in \mathbb \Bbbk^\times$. In general, one should expect that the time evolution operators $\ZZ(W)$ are only defined up to global multiplication by a scalar. One can always fix a choice of normalization for $\ZZ(W)$, but these choices may not respect composition. Hence, from a physical standpoint, instead of a TQFT in the sense above, one should expect some \emph{projective TQFT} or \emph{anomalous TQFT} which is an assignment $\ZZ:\cob_{n+1} \to \Vect$ that only respects composition up to a scalar. We will call these scalars the \emph{anomaly} of $\ZZ$. In general, one expects that the anomaly of $\ZZ$ can be described by an invertible TQFT in dimension $n+2$.
See \cite{FreedAnomalyInvertibleFT, FreedAnomaly, JacksonVanDykeProjectiveTQFT} for more details and physical motivations on these notions. 

\paragraph{Projective theories and anomalous theories} There are two ways to encode the notion of a TQFT with anomaly presented above. One is a functor $\ZZ:\cob_{n+1} \to \mathbb P \Vect$ to some category of projective vector spaces, see \cite{FreedAnomaly, JacksonVanDykeProjectiveTQFT} and Section \ref{Sec_projectiveTQFT}, which is called a \emph{projective theory}. Such a theory contains the information of the \emph{anomaly} $\alpha$ of $\ZZ$ which is an invertible once-categorified $(n+1)$-TQFT (in particular, it assigns one-dimensional vector spaces to closed $(n+1)$-manifolds). The projective theory $\ZZ$ on an $(n+1)$-manifold $M$ gives a vector $\ZZ(M)$ in the one-dimensional vector space $\alpha(M)$. This is not quite an invariant of $(n+1)$-manifolds as we expect from a usual TQFT, as we need to know how this vector space $\alpha(M)$ identifies with $\Bbbk$. The extra data of such an identification is called a trivialization of the anomaly on $M$. 

An other approach to model anomalous theories is that of a functor $\widetilde\ZZ:\widetilde\cob_{n+1} \to \Vect$ where $\widetilde\cob_{n+1}$ is a category of cobordisms with extra data, which we think of as the data needed to trivialize the anomaly and hence promote the projective theory $\ZZ$ to an honest linear theory. We will sometimes call $\widetilde\ZZ$ the resolution of $\ZZ$. See \cite{FreedAnomaly, JacksonVanDykeProjectiveTQFT} and Section \ref{Sec_projectiveTQFT} for details. 

\paragraph{WRT theories}
Witten described in \cite{WittenJonesPol} an example of a Topological Quantum Field Theory in dimension 3, obtained from the Chern--Simons action on $G$-bundles with connections. As is often the case with quantum field theories, the time evolution is described physically via an ill-defined path integral on a infinite-dimensional space of fields. Despite this, Witten managed to obtain a much more concrete description which was made mathematically precise by Reshetikhin and Turaev and the development of \emph{skein theory} \cite{ReshetikhinTuraev, TuraevBook, BHMV95}.

Skein theory is a theory of diagrammatics of tangles in 3D. In the modern formulation, it takes as input a ribbon category $\AA$, which is a well-behaved rigid balanced braided monoidal category, and outputs a way of evaluating tangles or graphs whose strands and vertices are colored by objects and morphisms of $\AA$ \cite[Thm. 2.5]{TuraevBook}. When the category $\AA$ is moreover semisimple modular, Reshetikhin--Turaev's construction produces a (2+1)-TQFT which we will denote $\WRT_\AA$. In the case where $\AA$ is obtained from representations of a quantum group $\mathcal{U}_q(G)$ at a root of unity, it should model Witten's theories. As one should expect for a theory coming from physics, this TQFT has an anomaly.

\paragraph{WRT at the boundary of Crane--Yetter}
Walker explains in his unfinished notes \cite{WalkerNotes} how to describe the anomaly of WRT using skein theory. It turns out, skein theory most naturally fits into a 4-dimensional TQFT $\SS_\AA$. The state space of a 3-manifold $M$ is given by its \emph{skein module} $\Sk_\AA(M)$, spanned by $\AA$-colored graphs in $M$ modulo isotopy and some local relations. The linear map associated to a 4-cobordism $W$ is defined, under some conditions on $\AA$, using a handle decomposition of $W$. These restrictions are met when $\AA$ is semisimple modular and the 4-TQFT obtained coincides \cite{WalkerNotes, ThamThesis} with Crane--Yetter theory \cite{CraneYetter}. 

Walker explains moreover how to reobtain WRT from this 4-TQFT (and, we will see, some extra structure). The state space of WRT on a surface $\Sigma$ is given by the skein module of a bounding handlebody $H$. The invariant of 3-manifold $\WRT_\AA(M)$ is given, up to a renormalization scalar, by taking $\SS_\AA$ on a bounding 4-manifold $M\overset{W}{\to}\emptyset$, which gives a linear map $\SS_\AA(W):\Sk_\AA(M) \to \mathbb C$, and evaluating this map on the empty skein $\emptyset \in \Sk_\AA(M)$. For general 3-cobordisms there is a similar procedure, see Section \ref{Sec_WRT} and Figure \ref{fig:AaofM}. The idea here is that the linear maps $\WRT_\AA(M)$ are only defined up to a scalar, as one expects from an anomalous theory, but choosing a bounding 4-manifold fixes the scalar anomaly. 
In this way, the anomaly of WRT is entirely described by the values of Crane--Yetter on 4-manifolds.

Here we have implicitly used two facts about the skein-theoretic 4-TQFT $\SS_\AA$. 
First, when we used skein modules of 3-manifolds with boundary, we relied on the fact that $\SS_\AA$ is a \emph{once-extended TQFT} and also gives values to surfaces and 3-manifolds with boundary. Second, when we used the empty skein in the skein modules, we relied on the presence of a \emph{boundary condition} to $\SS_\AA$.

\paragraph{Once-extended TQFTs} Locality in physics, which asks that physics at one part of space does not depend on "distant" parts of space, can be modeled by asking that our TQFTs can be cut in the space directions too. Let us write $n+2$ for the space-time dimension, and consider an $(n+1)$-manifold $M$ with a decomposition $M=M_1\underset{\Sigma}{\cup}M_2$ into two $(n+1)$-manifolds with boundary glued along an $n$-manifold $\Sigma$. In a ``local" TQFT $\ZZ$, one should be able to compute the state space $\ZZ(M)$ from this decomposition, i.e. as some kind of composition of ``states on $M_1$" $\ZZ(M_1)$ and ``states on $M_2$" $\ZZ(M_2)$, composed, or glued, over some ``states on $\Sigma$" $\ZZ(\Sigma)$. The precise nature of these objects is still to be determined at this point (below, $\ZZ(\Sigma)$ will be a linear category and $\ZZ(M_i)$ a pro-functor). Similarly, if a space-time $W: M\to M'$ can be decomposed as $W = W_1 \underset{\Sigma\times I}{\cup} W_2$ into two cobordisms with corners glued along their side boundary $\Sigma\times I$, then the time evolution $\ZZ(W)$ should be obtained as a composition of time evolution $\ZZ(W_1)$ on $W_1$ and time evolution $\ZZ(W_2)$ on $W_2$.

This idea is formalized mathematically by the notion of a once-extended $(n+2)$-TQFT, also called an $(n+1+1)$-TQFT or an $(n,n+1,n+2)$-TQFT, namely a symmetric monoidal 2-functor 
$$\ZZ: \Cob_{n+1+1}^\sqcup \to 2\hskip-2pt\Vect^\otimes$$
where $\Cob_{n+1+1}$ is a bicategory of $n$-manifolds, $(n+1)$-cobordisms and $(n+2)$-cobordisms with corners, and $2\hskip-2pt\Vect$ is a symmetric monoidal 2-category with $\Omega2\hskip-2pt\Vect := \End_{2\hskip-2pt\Vect}(\unit) = \Vect$. 

See \cite{FreedHighAlgQuantiz92, LawrenceETFT, WalkerOnWitten} for some early formulations and \cite{SPPhD,MullerPhD, HaiounHandle} for definitions of $\Cob_{n+1+1}$, which we recall in Definition \ref{def:cobBicat}. There are various models for $2\hskip-2pt\Vect$ and we will use $2\hskip-2pt\Vect := \Bimod$, which we recall in Definition \ref{def:Bimod}.

\paragraph{Fully-extended TQFTs and the cobordism hypothesis}
As the name suggests, once-extended TQFTs are only the first step towards a notion of fully extended TQFTs where one can cut spaces in more than one direction (and in particular chunk them into disks). The notion of a fully extended $n$-TQFT can be formalized as a symmetric monoidal $(\infty, n)$-functor 
$$\ZZ: \operatorname{\textbf{Bord}}_n \to n\hskip-2pt\Vect$$
from some $(\infty,n)$-category $\operatorname{\textbf{Bord}}_n$ of 0-manifolds, 1-cobordisms, 2-cobordisms with corners etc... up to n-cobordisms, defined in \cite{LurieCob, CalaqueSchembauerCob}, to some symmetric monoidal $(\infty,n)$-category $n\hskip-2pt\Vect$ to be specified (see e.g. \cite{JFS}).

In this setting, and precisely because one can cut spaces and space-times into disks, there is a good classification of TQFTs by the Baez--Dolan--Hopkins--Lurie Cobordism Hypothesis \cite{BaezDolan, LurieCob}. It states that a framed $n$-TQFT is determined by its value on the point $\ZZ(\bullet)$ which has to be a fully dualizable object (has a dual, and the evaluation and coevaluation morphisms have adjoints, and so on) in $n\hskip-2pt\Vect$. An oriented $n$-TQFT is classified by such a fully dualizable object $\ZZ(\bullet)\in n\hskip-2pt\Vect$ together with an $SO(n)$-homotopy-fixed-point structure. 

\paragraph{Boundary conditions} 
Motivated again by constructions from physics, a boundary condition to a TQFT $\ZZ$ is some extra structure that allows one to ``close" a part of the boundary. The data of $\ZZ$ with a boundary condition is sometimes called an open-closed TQFT. In some contexts they are called relative TQFTs, twisted TQFTs, or domain walls. 

There are geometrically \cite{LurieCob} and algebraically \cite{FreedTeleman, JFS} flavored definitions for boundary conditions. We will focus on the latter. See \cite{StewartPhD} for a comparison. In general, a boundary condition to an $n$-TQFT $\ZZ$ is some kind of natural transformation 
$$\RR: \Triv \Rightarrow \ZZ\vert_{n-1}$$
where $\Triv$ is the trivial theory with constant value the monoidal unit, and $\ZZ\vert_{n-1}$ is the restriction of $\ZZ$ to spaces of dimension at most $n-1$ (i.e. we have forgotten that $\ZZ$ is also defined on space-times). The precise definition of a boundary condition depends on what kind of TQFT $\ZZ$ is. 

If $\ZZ$ is a non-extended TQFT, then the notion of a boundary condition is not very rich, and is simply a $\operatorname{Diff}(M)$-invariant vector $\RR(M) \in \ZZ(M)$ for every closed space $M$.

If $\ZZ$ is a once-extended $(n+2)$-TQFT, then a boundary condition to $\ZZ$ is a symmetric monoidal oplax-natural transformation
$$\RR: \Triv \Rightarrow \ZZ^\varepsilon$$
where $\ZZ^\varepsilon$ is the restriction of $\ZZ$ to the bicategory $\Cob_{n+1+\varepsilon}$ of $n$-manifolds, $(n+1)$-cobordisms and diffeomorphisms. A more explicit description is given in Definition \ref{def:BondaryCond}.

If $\ZZ$ is a fully-extended $n$-TQFT, then a boundary condition to $\ZZ$ is a symmetric monoidal oplax-natural transformation
$$\RR: \Triv \Rightarrow \ZZ\vert_{\operatorname{\textbf{Bord}}_{n-1}}$$
in the sense of \cite{JFS}, namely $\RR$ is a symmetric monoidal functor $$\RR: \operatorname{\textbf{Bord}}_{n-1}\to n\hskip-2pt\Vect^{\unit \to}$$ into the arrow category of \cite{JFS}. Similarly, boundary conditions to a once-extended TQFT can be repackaged into a functor from $\Cob_{n+1+\varepsilon}$ to some arrow category.

This captures the notion of a projective TQFT of \cite{FreedAnomaly, JacksonVanDykeProjectiveTQFT} when the theory $\ZZ$ is invertible.

\paragraph{Non-semisimple TQFTs}
When the ribbon category $\AA$ is non-semisimple, its skein theory does not behave well. One problem is that the link invariants obtained from projective objects of $\AA$ are zero. One solution that has had great results is to precisely restrict to these ``bad" objects in $\AA$ (more generally to some tensor ideal $\II\subseteq\AA$) and to renormalize the zero that appears everywhere in order to get something non-trivial. This is the approach behind modified traces developed in \cite{GeerPatureauTuraevModifiedqdim, GKPgeneralizedTraceModdim}. Another problem was that the Reshetikhin--Turaev procedure to obtain 3-manifold invariants doesn't work. This was solved by Hennings \cite{Hennings3mfldInvHopf} and Lyubashenko \cite{Lyub3mfldProjMCG}. 

These techniques have been exploited to define ``non-semisimple" invariants of 3-manifolds \cite{CGP} and 3-TQFTs \cite{BCGP, DGGPR}. Importantly, these 3-TQFTs are only partially defined, they need some admissibility condition on the 3-cobordism, and are an instance of the notion of \emph{non-compact TQFTs} \cite{LurieCob}. A crucial object in these construction is the \emph{admissible skein module} of a 3-manifold with coefficient in the tensor ideal $\II\subseteq\AA$, studied in \cite{CGPAdmissbleskein}. The topological content of these non-semisimple TQFTs have proven to be very strong, the 3-manifold invariants of \cite{CGP} distinguish Lens spaces which no semisimple TQFT could do, and the Mapping Class Groups representations of \cite{DGGPR} have no known kernel. 

A non-semisimple 4-TQFT generalizing Crane--Yetter appeared in \cite{CGHP}. Its constructions closely mirrors the skein-theoretic construction of \cite{WalkerNotes} with the non-semisimple adaptations above. We see in Sections \ref{Sec_WRT} and \ref{Sec_projectiveTQFT} that it indeed models the anomaly of \cite{DGGPR}.

\paragraph{Crane--Yetter as a once-extended TQFT}
The Crane--Yetter theories, as any state sum theory (see \cite{BalsamKirillovExtendedTV} for the Turaev--Viro case), have long been known to be extended \cite{WalkerNotes, ThamThesis, KirillovTham4DTQFTskein}, though, to our knowledge, they have never been written down as a symmetric monoidal 2-functor in the literature. We do this in Section \ref{Sec_ExtendedCY}. 

In the description of \cite{WalkerNotes}, the Crane--Yetter theories associate \emph{skein categories} to surfaces and \emph{skein bimodule functors} to 3-cobordisms. The non-semisimple adaptation of these skein categories and skein bimodule functors is studied in \cite{BHskcat}. See Section \ref{Sec_Skein}.

On 4-cobordisms, Walker describes the 4-TQFT on handle attachments, which is generalized in \cite{CGHP} in the non-semisimple case. One then has to check that the assignments defined this way do not depend on the given handle decomposition and assemble into a 2-functor.
The main result of the author's \cite{HaiounHandle} reduces this to checking a finite number of relations, by extending \cite{Juhasz} to once-extended TQFTs. See Section \ref{Sec_ExtendedCY}.

\paragraph{Crane--Yetter as a fully extended TQFT} 
The description of Walker actually extends all the way down to the point, and one expects that Crane--Yetter theories are fully extended. In particular, they should fall under the classification of the cobordism hypothesis and be associated to a 4-dualizable object (equipped with an $SO(4)$-homotopy-fixed-point structure) in some 4-category $4\hskip-2pt\Vect$. This object should be the modular category $\AA$ itself.

An appropriate model for the 4-category $4\hskip-2pt\Vect$ is the even higher Morita category $\operatorname{Alg}_2(\Pr)$ defined in \cite{JFS}. The modular category $\AA$ (formally, its Ind-completion) indeed gives an object of this 4-category. It is shown to be 4-dualizable in \cite{BJS}. In the non-semisimple, it is shown that (Ind-completions of) non-semisimple modular categories are 4-dualizable in \cite{BJSS}. The $SO(4)$-structure is still poorly understood (see \cite{SchomPriesDualLowdimHighCat} for some ideas of what it represents in lower dimensions) but the author has previously conjectured in \cite{HaiounUnit} that it is induced by the ribbon structure of $\AA$ and a modified trace on $\II$.

\paragraph{A fully extended boundary condition to Crane--Yetter}
Similarly, Walker's description of the boundary condition as the empty skein extends all the way down to the point. Again, it should be classified by some version of the cobordism hypothesis for boundary conditions, and correspond to some fully dualizable 1-morphism $\eta: \unit \to \AA$ in $4\hskip-2pt\Vect$. More precisely $\eta$ should be 3-dualizable in the arrow category $4\hskip-2pt\Vect^\to$ \cite{JFS}.

It is expected by Walker and Freed \cite{FreedSlides} that this 1-morphism is given by the regular bimodule $ {}_{\unit}\AA_\AA$. It is shown to be a 3-dualizable (resp. almost 3-dualizable in an appropriate sense in the non-semisimple case) object in $\operatorname{Alg}_2(\Pr)^\to$ in  \cite{HaiounUnit}. This partial dualizability in the non-semisimple case reflects the fact the DGGPR theories are not defined on every 3-cobordism, i.e. are non-compact TQFTs.

From this description, one expects to be able to give a fully extended description of WRT (resp. \cite{DGGPR}) as a fully extended boundary condition to the fully extended Crane--Yetter (resp. \cite{CGHP}) theory, all of which are obtained through the cobordism hypothesis. These expectations have been formalized as a conjecture in \cite{HaiounUnit}. 

In this paper, we prove a version of these conjectures in a once-extended setting, which does not involve the cobordism hypothesis. 
Our approach naturally gives an oriented theory and does not run into the orientability difficulties that the fully extended approach encounters.

\subsection{Results}
We construct Crane--Yetter and its non-semisimple variant as a once-extended 4-TQFT, define a boundary condition to it, and provide a procedure to recover WRT from this 4-TQFT and its boundary condition.

\paragraph{Once-extended non-semisimple Crane--Yetter}
Even though the description of Crane--Yetter on surfaces, 3-cobordisms and 4-cobordisms with corners is well understood via skein theory, see \cite{WalkerNotes, ThamThesis} for the semisimple case and \cite{CGHP, BHskcat} for the non-semisimple case, it is non-trivial to show that this assignment assembles into a symmetric monoidal 2-functor. One has to check all the coherences asked of such a functor while dealing with choices of handle decomposition for 4-cobordisms. 

This difficulty is overcome by the main result of \cite{HaiounHandle} which allows us to build a once-extended $4$-TQFT $\ZZ: \Cob_{2+1+1}\to \CC$ from: 
\begin{itemize}
    \item A categorified $3$-TQFT $\ZZ^\varepsilon:\Cob_{2+1+\varepsilon}\to \CC$ (i.e. the values on surfaces, 3-cobordisms and diffeomorphisms, but nothing in dimension 4), and
    \item The values on standard handle attachment $Z_k:\ZZ^\varepsilon(S^{k-1}\times \D^{4-k}) \to \ZZ^\varepsilon(\D^k \times S^{3-k}),\ k =0,\dots,4,$ satisfying handle cancellations and invariance under reversal of the attaching spheres.
\end{itemize}
It is well-known that skein theory assembles into a categorified 3-TQFT \cite{WalkerNotes}, see \cite{BHskcat, RunkelSchweigertThamExciAdmSkeins} in the non-semisimple case. However, it was never quite written down as a symmetric monoidal 2-functor $\ZZ^\varepsilon:\Cob_{2+1+\varepsilon}\to \CC$, which we do in Section \ref{Sec_Skein}.
\newtheorem*{thm:SkCategTQFT}{Theorem \ref{thm:SkCategTQFT}}
\begin{thm:SkCategTQFT} \it
    Let $\II\subseteq\AA$ be a tensor ideal in a linear ribbon category. The skein categories of surfaces and skein bimodule functors of 3-cobordisms of \cite{WalkerNotes, BHskcat} assemble into a categorified 3-TQFT, or $(2+1+\varepsilon)$-TQFT
    $$\SkFun_\II:\Cob_{2+1+\varepsilon}\to \Bimod^{hop}$$
\end{thm:SkCategTQFT}

The values on standard handle attachment have already been defined, in a non-extended setting, in \cite{CGHP} and we simply translate the construction in our setting in Section \ref{Sec_ExtendedCY}. The conditions on $\AA$ found in \cite{CGHP} for these handle attachments to be defined is to be ``chromatic compact". This includes in particular the case where $\AA$ is (possibly non-semisimple) modular.
\newtheorem*{thm:ExtendedCGHP}{Theorem \ref{thm:ExtendedCGHP}}
\begin{thm:ExtendedCGHP} \it
    Let $\AA$ be a chromatic compact category, $\II=\operatorname{Proj(\AA)}$ and $\mt$ a modified trace on $\II$. The skein categorified 3-TQFT $\SkFun_\II$ extends to a once-extended 4-TQFT, or $(2+1+1)$-TQFT
    $$\SS_\II:\Cob_{2+1+1}\to \Bimod^{hop}$$
\end{thm:ExtendedCGHP}

\paragraph{A boundary condition to Crane--Yetter}
Having defined Crane--Yetter as a once-extended 4-TQFT, we can define a boundary condition to it. See Section \ref{Sec_BoundaryCond}. In Walker's picture, it is given by the empty skein in skein categories and skein modules. This becomes more subtle in the non-semisimple case as the admissibility conditions precisely discard the empty skein as an element of admissible skein modules and as an object of the $\II$-skein category. 

Nevertheless, the ``empty skein" defines a presheaf on the $\II$-skein category. This is what our boundary condition assigns to surfaces. Similarly, for a 3-cobordism $M$ with incoming boundary in every connected component, we can make sense of something like ``adding the empty skein in $M$". This is what our boundary condition does on 3-cobordisms. When $\II\neq \AA$, i.e. when we do not allow the empty skein, it is not defined on every 3-cobordisms, and is only a non-compact boundary condition.
\newtheorem*{thm:BondCond}{Theorem \ref{thm:BondCond}}
\begin{thm:BondCond} \it
    Let $\II\subseteq\AA$ be a tensor ideal in a linear ribbon category. The empty skein defines a boundary condition, non-compact when $\II\neq \AA$
        \begin{equation*}
            \begin{tikzcd}
            \Cob_{2+1+\varepsilon}^{nc, hop} \ar[dd, bend right=50pt,"\mathlarger\Triv"', ""{name=left, right}]\ar[dd, bend left=50pt,"\mathlarger{\SkFun_\II}", ""{name=right, left}] \\ \ \\\Bimod          
            \ar[from=left,to=right,Rightarrow,"\mathlarger{\RR_\II}"]
        \end{tikzcd}
        \end{equation*}
\end{thm:BondCond}

\paragraph{Reconstructing WRT}
Walker and Freed explained that one should be able to reconstruct WRT by composing the boundary condition with Crane--Yetter on a bounding manifold. 
We define the \emph{anomalous theory} $\Aa$ associated to a once-extended theory equipped with a boundary condition via this construction in Definition \ref{def:anomalous}. It outputs some kind of TQFT defined on a category of cobordisms equipped with bounding manifolds $\cob_{2+1}^{filled}$, non-compact if the boundary condition is non-compact. This category of cobordisms maps to the usual source category $\widetilde\cob_{2+1}$ for WRT by only remembering some part of the data of the bounding manifolds. 

Let $\AA$ be a (possibly non-semisimple) modular category, $\II = \operatorname{Proj}(\AA)$ (so $\II = \AA$ in the semisimple case), and $\mt$ a modified trace on $\II$. We need to normalize $\mt$ correctly  and divide it by a chosen square root $\DD$ of the global dimension $\zeta = \SS_\II(S^4)$. This ensures that with this new modified trace we get $\zeta=1$.

We call $\Aa_\II$ the anomalous theory obtained from $\SS_\II$ and $\RR_\II$ above. From the same data, Reshetikhin--Turaev \cite{ReshetikhinTuraev, TuraevBook} in the semisimple case, and \cite{DGGPR} in the non-semisimple case, construct a 3-TQFT with anomaly 
$$\WRT_\AA : \widetilde\cob_{2+1} \to \Vect$$
and
$$\DGGPR_\AA:\widetilde\cob_{2+1}^{nc} \to \Vect$$
which we compare to the anomalous theory $\Aa_\II$ in Section \ref{Sec_WRT}.
\newtheorem*{thm:WRT}{Theorem \ref{thm:WRT}}
\begin{thm:WRT} \it
If $\AA$ is semisimple modular, then
    \begin{equation*}
        \begin{tikzcd}
            \cob_{2+1}^{filled} \ar[rr, "\Aa_\AA"] && \Vect\\
            & \widetilde \cob_{2+1} &
            \arrow[from = 1-1, to = 2-2, "{\pi}"]
            \arrow[from = 2-2, to = 1-3, "\WRT_\AA"', sloped]
        \end{tikzcd}
    \end{equation*}
    commutes up to symmetric monoidal natural isomorphism.   
\end{thm:WRT}
\newtheorem*{thm:DGGPR}{Theorem \ref{thm:DGGPR}}
\begin{thm:DGGPR} \it
If $\AA$ is non-semisimple modular, then
    \begin{equation*}
        \begin{tikzcd}
            \cob_{2+1}^{filled, nc} \ar[rr, "\Aa_\II"] && \Vect\\
            & \widetilde \cob_{2+1}^{nc} &
            \arrow[from = 1-1, to = 2-2, "{\pi}"]
            \arrow[from = 2-2, to = 1-3, "\DGGPR_\AA"', sloped]
        \end{tikzcd}
    \end{equation*}
    commutes up to symmetric monoidal natural isomorphism.   
\end{thm:DGGPR}

\paragraph{WRT as a projective theory} In the setting developed by Freed and Van Dyke, it is really the boundary condition that describes WRT as a projective theory. The construction above, evaluating the $\SS_\II$ on a bounding manifold, is only a way of trivializing the anomaly and obtaining a linear theory. In other words, the anomalous theory $\Aa_\II: \cob_{2+1}^{filled, nc}\to \Vect$ and the historical construction $\WRT_\AA, \DGGPR_\AA: \widetilde\cob_{2+1}^{nc}\to \Vect$ are resolutions of the projective theory $\RR_\II$. 

The boundary condition $\RR_\II:\Triv \Rightarrow \SkFun_\II$ is the same data as a symmetric monoidal functor $\RR_\II: \Cob_{n+1+\varepsilon}^{nc, hop}\to \Bimod^{\unit\to}$ to the oplax arrow category of \cite{JFS}. It lands in the subcategory $\mathbb P\Vect$ defined by \cite{FreedAnomaly, JacksonVanDykeProjectiveTQFT} as $\SkFun_\II$ is invertible. 

We prove the following in Section \ref{Sec_projectiveTQFT}. 
\newtheorem*{thm:WRTprojective}{Theorem \ref{thm:WRTprojective}}
\begin{thm:WRTprojective} \it
Let $\AA$ be a semisimple modular tensor category with a chosen square root of its global dimension. Then $\RR_\AA: \cob_{2+1} \to \mathbb P\Vect$ is a projective TQFT, and the WRT theory $\WRT_\AA: \widetilde\cob_{2+1}\to \Vect$ is a resolution of $\RR_\AA$ in the sense that  \begin{equation*}
    \begin{tikzcd}
        \widetilde\cob_{2+1} \ar[d, "\pi"] \ar[r, "\WRT_\AA"] & \Vect \ar[d, "\mathbb P"] \\
        \cob_{2+1} \ar[r, "\RR_\AA"] & \mathbb P\Vect
    \end{tikzcd}
    \end{equation*} commutes up to symmetric monoidal pseudo-natural isomorphism.

Let $\AA$ be a non-semisimple modular category with a chosen modified trace with global dimension equal to 1 and $\II=\operatorname{Proj}(\AA)$. Then $\RR_\II: \cob_{2+1}^{nc} \to \mathbb P\Vect$ is a non-compact projective TQFT, and the DGGPR theory $\DGGPR_\AA: \widetilde\cob_{2+1}^{nc}\to \Vect$ is a resolution of $\RR_\II$ in the sense that  \begin{equation*}
    \begin{tikzcd}
        \widetilde\cob_{2+1}^{nc} \ar[d, "\pi"] \ar[r, "\DGGPR_\AA"] & \Vect \ar[d, "\mathbb P"] \\
        \cob_{2+1}^{nc} \ar[r, "\RR_\II"] & \mathbb P\Vect
    \end{tikzcd}
    \end{equation*} commutes up to symmetric monoidal pseudo-natural isomorphism.
\end{thm:WRTprojective}
Note that in the result above we choose a modified trace and a square root of the global dimension, but these are only used in the definition of $\WRT_\AA$ and $\DGGPR_\AA$, not in the definition of $\RR_\II$. In other words, this data is only needed to trivialize the anomaly, not to define the projective TQFT.

\subsection{Description of WRT as an anomalous theory}
From the theorems above, we learn that we may avoid the usual construction of WRT and DGGPR \cite{TuraevBook, DGGPR} using the universal construction of \cite{BHMV95}. Instead, we can describe them in terms of the 4-dimensional skein theory $\SS_\II$ and its boundary condition $\RR_\II$. Let us give an informal account of this description, which is made precise in Section \ref{ssec:DescrAnomalous}.
As we are simply trying to convey a general picture let us focus on the semisimple case. It should be noted that this description is certainly not new and appears in Walker's work. Instead, the reader should think that the following paragraphs describe what it is that we are trying to formalize.

\paragraph{State spaces of surfaces:} Let $\Sigma$ be a closed surface. Choose arbitrarily a handle-body $H$ bounding $\Sigma$. The state space $\Aa(\Sigma)$ of a surface $\Sigma$ is given by the skein module of $H$ with empty boundary
\begin{equation*}
  \Aa(\Sigma) :=  \Sk(H) = \left\langle 
\begin{tikzpicture}[xscale = 0.4,yscale = 0.18, baseline = -5pt]
    \fill[gray!10] (0,0) ..controls (0,3) and (1.5,1) .. (2,1) .. controls (2.5,1) and (4,3)..(4,0) ..controls (4,-3) and (2.5,-1).. (2,-1)..controls (1.5,-1) and (0,-3).. (0,0);
    \fill[gray!10] (4,0)..controls (4,-4) and (2.5,-5).. (2,-5)..controls (1.5,-5) and (0,-4).. (0,0)--cycle;
    \draw[gray] (4,1.5)node[right]{$\Sigma$}--(4,0)..controls (4,-4) and (2.5,-5).. (2,-5)..controls (1.5,-5) and (0,-4).. (0,0)--(0,1.5);
    \begin{scope}[yshift = 1.5cm]
    \fill[gray!10] (0,0) ..controls (0,3) and (1.5,1) .. (2,1) .. controls (2.5,1) and (4,3)..(4,0)..controls (4,-3) and (2.5,-1).. (2,-1)..controls (1.5,-1) and (0,-3).. (0,0);
    \draw[gray] (0,0) ..controls (0,3) and (1.5,1) .. (2,1) .. controls (2.5,1) and (4,3)..(4,0)..controls (4,-3) and (2.5,-1).. (2,-1)..controls (1.5,-1) and (0,-3).. (0,0);
    \end{scope}
    \draw (0.5,-0.5).. controls (0.5, 0.5) and (1,0.9).. (1.5,0.9)  .. controls (2,0.9) and (2.5,0.4).. (2.5,-0.3);
    \draw (0.5,-0.5)..controls (0.5,-2) .. (1,-3)node{$\bullet$}  ..controls (2,-2).. (3,-3)node[midway, sloped]{$\scriptstyle >$} node{$\bullet$}  .. controls (4,-2.5) and (2.5,-1).. (2.5,-0.3)node[midway, sloped]{$\scriptstyle <$};
    \draw (1,-3)  ..controls (2,-4).. (3,-3)node[midway, sloped]{$\scriptstyle >$};
    \node[gray] at (4.5,-2){$H$};
\begin{scope}[yshift = 1.5cm]
    \draw[gray!10, line width = 3pt] (4,0)..controls (4,-3) and (2.5,-1).. (2,-1)..controls (1.5,-1) and (0,-3).. (0,0);
    \draw[gray] (4,0)..controls (4,-3) and (2.5,-1).. (2,-1)..controls (1.5,-1) and (0,-3).. (0,0);
\end{scope}
\end{tikzpicture}
  \right\rangle
\end{equation*}
This can be seen as a composition of first $\RR_\II(\Sigma)$ including the empty object in the skein category of $\Sigma$, then applying the skein bimodule functor $\SkFun_\II(H)$ of $H$ on this object. There are some subtleties in the non-semisimple case as the empty object is not literally an object of the skein category, but this composition still holds. 

\paragraph{Correlation functions of 3-cobordisms:} Let $M:\Sigma \to \Sigma'$ be a 3-cobordism. Remember that we have chosen $H$ and $H'$ bounding $\Sigma$ and $\Sigma'$. We would like to construct a linear map from $\Sk(H)$ to $\Sk(H')$. Let $T$ be a skein in $H$. Then we can think of it as a skein $T \subseteq H\underset{\Sigma}{\cup}M$. Note that in the non-semisimple case this skein will only be admissible if $M$ has incoming boundary in every connected component, and we must restrict to this case. 

So we have a skein $T$ in a 3-manifold $H\underset{\Sigma}{\cup}M$ bounding $\Sigma'$, and we are trying to produce a skein $\Aa(M)(T)$ in another 3-manifold bounding $\Sigma'$, namely $H'$. Any two oriented compact 3-manifolds with same boundary are cobordant, i.e. there exists a 4-manifold $W:H\underset{\Sigma}{\cup}M \Rightarrow H'$. 
We may moreover assume that $W$ is constructed only from 0-, 2- and 4-handles. For every 2-handle of $W$, we obtain a skein in the surgered 3-manifold by adding a red circle along the newly created $\D^2\times S^1$, which is turned into a skein using the Kirby color, see \eqref{eq:redToBlue} and Figure \ref{fig:2handle}. For every 0-handle, we add the empty skein in $S^3$ times an appropriate scalar and for every 4-handle, we evaluate the skein in the $S^3$-component of $H\underset{\Sigma}{\cup}M$ as in Figure \ref{fig:4handle}.
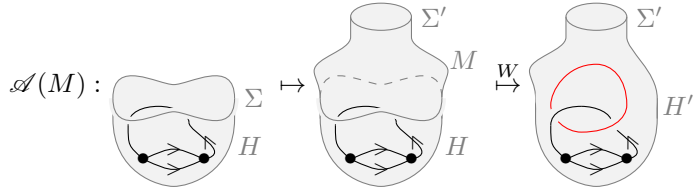
\begin{equation*}
    \Aa(M) : 
    \begin{tikzpicture}[xscale = 0.4,yscale = 0.18, baseline = 10pt]
    \fill[gray!10] (0,0) ..controls (0,3) and (1.5,1) .. (2,1) .. controls (2.5,1) and (4,3)..(4,0) ..controls (4,-3) and (2.5,-1).. (2,-1)..controls (1.5,-1) and (0,-3).. (0,0);
    \fill[gray!10] (4,0)..controls (4,-4) and (2.5,-5).. (2,-5)..controls (1.5,-5) and (0,-4).. (0,0)--cycle;
    \draw[gray] (4,1.5)node[right]{$\Sigma$}--(4,0)..controls (4,-4) and (2.5,-5).. (2,-5)..controls (1.5,-5) and (0,-4).. (0,0)--(0,1.5);
    \begin{scope}[yshift = 1.5cm]
    \fill[gray!10] (0,0) ..controls (0,3) and (1.5,1) .. (2,1) .. controls (2.5,1) and (4,3)..(4,0)..controls (4,-3) and (2.5,-1).. (2,-1)..controls (1.5,-1) and (0,-3).. (0,0);
    \draw[gray] (0,0) ..controls (0,3) and (1.5,1) .. (2,1) .. controls (2.5,1) and (4,3)..(4,0)..controls (4,-3) and (2.5,-1).. (2,-1)..controls (1.5,-1) and (0,-3).. (0,0);
    \end{scope}
    \draw (0.5,-0.5).. controls (0.5, 0.5) and (1,0.9).. (1.5,0.9)  .. controls (2,0.9) and (2.5,0.4).. (2.5,-0.3);
    \draw (0.5,-0.5)..controls (0.5,-2) .. (1,-3)node{$\bullet$}  ..controls (2,-2).. (3,-3)node[midway, sloped]{$\scriptstyle >$} node{$\bullet$}  .. controls (4,-2.5) and (2.5,-1).. (2.5,-0.3)node[midway, sloped]{$\scriptstyle <$};
    \draw (1,-3)  ..controls (2,-4).. (3,-3)node[midway, sloped]{$\scriptstyle >$};
    \node[gray] at (4.5,-2){$H$};
\begin{scope}[yshift = 1.5cm]
    \draw[gray!10, line width = 3pt] (4,0)..controls (4,-3) and (2.5,-1).. (2,-1)..controls (1.5,-1) and (0,-3).. (0,0);
    \draw[gray] (4,0)..controls (4,-3) and (2.5,-1).. (2,-1)..controls (1.5,-1) and (0,-3).. (0,0);
\end{scope}
\end{tikzpicture}
\mapsto
\begin{tikzpicture}[xscale = 0.4,yscale = 0.18, baseline = 10pt]
    \fill[gray!10] (0,0) ..controls (0,3) and (1.5,1) .. (2,1) .. controls (2.5,1) and (4,3)..(4,0)..controls (4,-3) and (2.5,-1).. (2,-1)..controls (1.5,-1) and (0,-3).. (0,0);
    \fill[gray!10] (0,0) -- (0,1.5)..controls (0,2.5) and (-0.2,3).. (-0.2,3.5).. controls (-0.2,4.5) and (1,5).. (1,6) -- (1,7.5) arc (180:0:1) --(3,6) ..controls (3,5) and (4.2,4.5)..(4.2,3.5)..controls (4.2,3) and (4,2.5)..(4,1.5) -- (4,0);
    \draw[gray] (0,0) -- (0,1.5)..controls (0,2.5) and (-0.2,3).. (-0.2,3.5).. controls (-0.2,4.5) and (1,5).. (1,6) -- (1,7.5);
    \draw[gray] (1,7.5) arc(180:540:1);

    \fill[gray!10] (0,0) ..controls (0,3) and (1.5,1) .. (2,1) .. controls (2.5,1) and (4,3)..(4,0)..controls (4,-3) and (2.5,-1).. (2,-1)..controls (1.5,-1) and (0,-3).. (0,0);
    \fill[gray!10] (4,0)..controls (4,-4) and (2.5,-5).. (2,-5)..controls (1.5,-5) and (0,-4).. (0,0)--cycle;
    \draw[gray] (4,1.5)--(4,0)..controls (4,-4) and (2.5,-5).. (2,-5)..controls (1.5,-5) and (0,-4).. (0,0)--(0,1.5);
    \begin{scope}[yshift = 1.5cm]
    \fill[gray!10] (0,0) ..controls (0,3) and (1.5,1) .. (2,1) .. controls (2.5,1) and (4,3)..(4,0)..controls (4,-3) and (2.5,-1).. (2,-1)..controls (1.5,-1) and (0,-3).. (0,0);
    \end{scope}
    \draw (0.5,-0.5).. controls (0.5, 0.5) and (1,0.9).. (1.5,0.9)  .. controls (2,0.9) and (2.5,0.4).. (2.5,-0.3);
    \draw (0.5,-0.5)..controls (0.5,-2) .. (1,-3)node{$\bullet$}  ..controls (2,-2).. (3,-3)node[midway, sloped]{$\scriptstyle >$} node{$\bullet$}  .. controls (4,-2.5) and (2.5,-1).. (2.5,-0.3)node[midway, sloped]{$\scriptstyle <$};
    \draw (1,-3)  ..controls (2,-4).. (3,-3)node[midway, sloped]{$\scriptstyle >$};
    \node[gray] at (4.5,-2){$H$};
    
    \draw[gray] (3,7.5)--(3,6) node[pos =0, right]{$\Sigma'$}..controls (3,5) and (4.2,4.5)..(4.2,3.5) node[pos = 0.7,right]{$M$}..controls (4.2,3) and (4,2.5)..(4,1.5) -- (4,0);
\begin{scope}[yshift = 1.5cm]
    \draw[gray!10, line width = 3pt] (4,0)..controls (4,-3) and (2.5,-1).. (2,-1)..controls (1.5,-1) and (0,-3).. (0,0);
    \draw[gray] (4,0)..controls (4,-3) and (2.5,-1).. (2,-1)..controls (1.5,-1) and (0,-3).. (0,0);
    \draw[gray, dashed] (0,0) ..controls (0,3) and (1.5,1) .. (2,1) .. controls (2.5,1) and (4,3)..(4,0);
\end{scope}
\end{tikzpicture}
\overset W{\mapsto}
\begin{tikzpicture}[xscale = 0.4,yscale = 0.18, baseline = 10pt]
    \fill[gray!10] (0,0) ..controls (0,3) and (1.5,1) .. (2,1) .. controls (2.5,1) and (4,3)..(4,0)..controls (4,-3) and (2.5,-1).. (2,-1)..controls (1.5,-1) and (0,-3).. (0,0);
    \fill[gray!10] (0,0) -- (0,1.5)..controls (0,2.5) and (-0.2,3).. (-0.2,3.5).. controls (-0.2,4.5) and (1,5).. (1,6) -- (1,7.5) arc (180:0:1) --(3,6)..controls (3,5) and (4,3.5)..(4,1.5) -- (4,0);
    \draw[gray] (0,0) -- (0,1.5)..controls (0,2.5) and (-0.2,3).. (-0.2,3.5).. controls (-0.2,4.5) and (1,5).. (1,6) -- (1,7.5);
    \draw[gray] (1,7.5) arc(180:540:1);
    \fill[gray!10] (0,0) ..controls (0,3) and (1.5,1) .. (2,1) .. controls (2.5,1) and (4,3)..(4,0)..controls (4,-3) and (2.5,-1).. (2,-1)..controls (1.5,-1) and (0,-3).. (0,0);
    \fill[gray!10] (4,0)..controls (4,-4) and (2.5,-5).. (2,-5)..controls (1.5,-5) and (0,-4).. (0,0)--cycle;
    \draw[gray] (4,1.5)--(4,0)..controls (4,-4) and (2.5,-5).. (2,-5)..controls (1.5,-5) and (0,-4).. (0,0)--(0,1.5);
    \begin{scope}[yshift = 1.5cm]
    \fill[gray!10] (0,0) ..controls (0,3) and (1.5,1) .. (2,1) .. controls (2.5,1) and (4,3)..(4,0)..controls (4,-3) and (2.5,-1).. (2,-1)..controls (1.5,-1) and (0,-3).. (0,0);
    \end{scope}
    \draw (0.5,-0.5).. controls (0.5, 0.5) and (1,0.9).. (1.5,0.9)  .. controls (2,0.9) and (2.5,0.4).. (2.5,-0.3);
    \draw (0.5,-0.5)..controls (0.5,-2) .. (1,-3)node{$\bullet$}  ..controls (2,-2).. (3,-3)node[midway, sloped]{$\scriptstyle >$} node{$\bullet$}  .. controls (4,-2.5) and (2.5,-1).. (2.5,-0.3)node[midway, sloped]{$\scriptstyle <$};
    \draw (1,-3)  ..controls (2,-4).. (3,-3)node[midway, sloped]{$\scriptstyle >$};
    \node[gray] at (4.7,1){$H'$};
    \draw[gray] (3,7.5)--(3,6) node[pos=0, right]{$\Sigma'$}..controls (3,5) and (4,3.5)..(4,1.5) -- (4,0);
    \draw[gray!10, line width = 5pt] (2,4).. controls (2.5,4) and (3,3).. (3,1)..controls (3,0) and (2.5,-1)..(1.5,-1)..controls (1,-1) and (0.5,0).. (0.5,1) ..controls (0.5,2) and (1,4)..(2,4);
    \draw[red] (2,4).. controls (2.5,4) and (3,3).. (3,1)..controls (3,0) and (2.5,-1)..(1.5,-1)..controls (1,-1) and (0.5,0).. (0.5,1) ..controls (0.5,2) and (1,4)..(2,4);
    \draw[gray!10, line width = 5pt] (0.5,-0.5).. controls (0.5, 0.5) and (1,0.9).. (1.5,0.9);
    \draw (0.5,-0.5).. controls (0.5, 0.5) and (1,0.9).. (1.5,0.9);
\end{tikzpicture}
\end{equation*}
Again, we may think of this as a composition of two maps, the boundary condition $\RR_\II(M)$ adding the empty skein in $M$ to the skein $T$ in $H$, and the skein 4-TQFT $\SS_\II$ on the bounding 4-manifold.

\paragraph{Anomaly}
Of course we made a lot of choices here, but crucially the map $\Aa(M)$ depends on $W$ only up to a global scalar, which is what we expect from a TQFT with anomaly. One can "fix" this anomaly by renormalizing the dependence on $W$. However, this renormalization will not behave well with gluing, and our TQFT will only preserve composition up to scalar. To obtain an actual functor, one has to remember the choices in the source category.

\subsection{Future directions}
\paragraph{Fully extended WRT} As mentioned above, the description of WRT as a boundary condition is expected to extend down to the point. We do know the fully dualizable objects that would induce Crane--Yetter and the boundary condition under the cobordism hypothesis. Their dualizability has been established in \cite{BJS, BJSS, HaiounUnit}. However, we do not know that they can be equipped with orientation structures. Nor do we know that the cobordism hypothesis applied on these objects gives the values we expect, with the notable exception of the anomaly in dimension at most two, that has to agree with factorization homology \cite{AyalaFrancisTanakaFH, Scheimbauer} which itself agrees with skein categories \cite{CookeExcision, BHskcat}.

\paragraph{Anomalous TQFTs} We now know how to define the notion of an anomalous theory and what it might look like. One may hope that a good understanding of the example of WRT will open avenues for generalizations. For example, are there interesting anomalous, possibly non-compact, 4-dimensional TQFTs? 
Their anomaly would be described by an invertible 5-dimensional theory. Do we know how to describe those? In dimension up to 4 they are classified in \cite{SchommerPriesInvertible}.

\subsection*{Acknowledgments}
I would like to thank warmly David Jordan for his advising role at various steps of this project. I would also like to thank Francesco Costantino, Pavel Safronov, David Reutter, Kevin Walker, Marco De Renzi, Will Stewart, Patrick Kinnear and Jackson Van Dyke for guidance, encouragements and enlightening conversations. 


\section{Background}\label{Sec_Background}
\subsection{Cobordism bicategories and once-extended TQFTs}
We begin by recalling the definition of the cobordism bicategory. We adopt the definition of \cite{HaiounHandle} and refer there for details, but see also \cite{SPPhD}. Unless stated otherwise, every manifold below is compact, smooth and oriented.
\begin{definition}\label{def:cobBicat}
    The \textbf{bicategory of $(2+1+1)$-cobordisms} $\Cob_{2+1+1}$ is the symmetric monoidal bicategory with 
    \begin{description}
        \item[objects:] Closed oriented smooth surfaces $\Sigma$
        \item[1-morphisms:] 3-cobordisms $M:\Sigma_-\to \Sigma_+$ equipped with a collar of their boundary $\Sigma_\pm\times [\pm1,\pm\frac{1}{2})\inj M$. Composition is given by gluing the collars, which inherits a natural smooth structure.
        \item[2-morphisms:] 4-cobordisms with corners $W:M_-\to M_+$, equipped with a side collar of their side boundary $\Sigma_\pm\times [-1,1]\times [\pm1,\pm\frac{1}{2})\inj W$ compatible with the collars of $M_\pm$, and considered up to diffeomorphisms preserving $M_\pm$ and preserving side collars up to a reparametrization of the $[-1,1]$-coordinate. Horizontal composition is gluing the collars. Vertical composition is gluing along $M$'s, whose smooth structure is well-defined up to diffeomorphism. 
    \end{description}
    It is symmetric monoidal with disjoint union.

    The \textbf{bicategory of non-compact $(2+1+1)$-cobordisms} $\Cob_{2+1+1}^{nc}$ is the symmetric monoidal sub-bicategory of $\Cob_{2+1+1}$ with the same objects and 1-morphisms but only those 2-morphisms where the source diffeomorphism is surjective on connected components, i.e. every connected component of the $4$-cobordisms have non-empty incoming boundary. 

    The \textbf{bicategory of $(2+1+\varepsilon)$-cobordisms} $\Cob_{2+1+\varepsilon}$ has the same objects and 1-morphisms as $\Cob_{2+1+1}$, but 2-morphisms are isotopy classes of diffeomorphisms preserving the side collars. It comes with a symmetric monoidal strict 2-functor $\Cob_{2+1+\varepsilon}\to \Cob_{2+1+1}^{nc} \subseteq \Cob_{2+1+1}$ which is the identity on objects and 1-morphisms and maps a diffeomorphism to its mapping cylinder.

    The \textbf{bicategory of non-compact $(2+1+\varepsilon)$-cobordisms} $\Cob_{2+1+\varepsilon}^{nc}$ is the locally full symmetric monoidal sub-bicategory of $\Cob_{n+1+\varepsilon}$ with the same objects but only those 1-morphisms where the target diffeomorphism is surjective on connected components, i.e. every connected component of the $3$-cobordisms have non-empty outgoing boundary, and all 2-morphisms between these. Note that we have switched incoming to outgoing here for the purposes of the examples we will study below.
\end{definition}
\begin{definition}
    Let $\CC$ be a symmetric monoidal bicategory. \\
    A \textbf{once-extended $4$-TQFT}, or a \textbf{$(2+1+1)$-TQFT}, or \textbf{$2$-$3$-$4$-TQFT}, with values in $\CC$ is a symmetric monoidal 2-functor 
    \begin{equation*}
        \ZZ: \Cob_{2+1+1}\to \CC\ .
    \end{equation*}
    A \textbf{non-compact once-extended $4$-TQFT} is a symmetric monoidal 2-functor 
    \begin{equation*}
        \ZZ: \Cob_{2+1+1}^{nc}\to \CC\ .
    \end{equation*}
    A \textbf{categorified $3$-TQFT}, or a \textbf{$(2+1+\varepsilon)$-TQFT}, is a symmetric monoidal 2-functor 
    \begin{equation*}
        \ZZ: \Cob_{2+1+\varepsilon}\to \CC\ .
    \end{equation*}
    A \textbf{boundary condition} to a once-extended 4-TQFT $\ZZ$ is a symmetric monoidal oplax natural transformation
    $$\RR: \Triv \Rightarrow \ZZ^\varepsilon$$
    where $\ZZ^\varepsilon: \Cob_{n+1+\varepsilon}\to \Cob_{n+1+1} \overset{\ZZ}{\to} \CC$ is the restriction of $\ZZ$ to $\Cob_{n+1+\varepsilon}$ and $\Triv: \Cob_{2+1+\varepsilon} \to \CC$ is constant equal to the monoidal unit.
\\
    A \textbf{non-compact boundary condition} to a once-extended 4-TQFT $\ZZ$ is a symmetric monoidal oplax natural transformation
    $$\RR: \Triv \Rightarrow \ZZ^{\varepsilon, nc}$$
    where $\ZZ^{\varepsilon, nc}: \Cob_{n+1+\varepsilon}^{nc}\to \Cob_{n+1+1} \overset{\ZZ}{\to} \CC$ is the restriction of $\ZZ$ to $\Cob_{n+1+\varepsilon}^{nc}$.
\end{definition}

\subsection{Categorified linear algebra}
Many of the constructions on this subject date back to Grothendieck, and it is uneasy to find a comprehensive and accessible reference. Details can be found in \cite{AdamekRosicky,DayStreetMonBicat,Kelly2, BCJReflDualPr, BJS, GJS}. We recommend \cite{DuggerSheavesHomotopy} for an introduction to the main ideas at play. 

Let us quickly recall the definition of the target bicategory we will consider for our once-extended TQFTs.
\begin{definition}\label{def:Bimod} Let $\Bbbk$ be a field.

The (strict) bicategory $\Cat_\Bbbk$ has objects small $\Bbbk$-linear categories, 1-morphisms linear functors and 2-morphisms natural transformations. It is symmetric monoidal with tensor product $\otimes$ which is Cartesian product on objects and tensor product on spaces of morphisms.

    The bicategory $\Bimod$ has objects small $\Bbbk$-linear categories, 1-morphisms $\CC\to\DD$ are profunctors, or bimodule functors $F:\CC\otimes\DD^{op}\to\Vect_\Bbbk$, and 2-morphisms natural transformations. Composition of 1-morphisms $F:\CC\otimes\DD^{op}\to\Vect_\Bbbk$ and $G:\DD\otimes\EE^{op}\to\Vect_\Bbbk$ is given by the coend
    \begin{equation*}
        (G\circ F)(C,E) := \int^{D\in\DD} F(C,D)\otimes G(D,E)
    \end{equation*} 
    It is symmetric monoidal with the usual tensor product $\otimes$ on linear categories.

There is a symmetric monoidal 2-functor $\Cat \to \Bimod$ which is the identity on objects and post-composition with the Yoneda embedding $\DD\to\widehat\DD:=\Fun(\DD^{op},\Vect_\Bbbk)$ on morphisms. 

    The (strict) bicategory $\Pr$ has objects presentable $\Bbbk$-linear categories, 1-morphisms cocontinuous functors and 2-morphisms natural transformations. It is symmetric monoidal with Kelly-Deligne tensor product.

    There is a symmetric monoidal fully faithful embedding $\widehat{(-)}:\Bimod \to\Pr$ which maps a category $\CC$ to its free cocompletion, or presheaf category $\widehat\CC := \Fun(\CC^{op},\Vect_\Bbbk)$. A profunctor $\CC\to\widehat\DD$ extends essentially uniquely to a cocontinuous functor $\widehat\CC\to \widehat\DD$ by the co-Yoneda Lemma, see \cite[Prop. 2.2.4]{DuggerSheavesHomotopy}. The essential image of $\widehat{(-)}$ consists of the presentable categories with enough compact-projective objects.
\end{definition}

\section{Skein theory as a categorified 3-TQFT}\label{Sec_Skein}
Ribbon categories are a class of particularly well-behaved $\mathbb E_2^{or}$-algebra in $\Cat_\Bbbk$ which have a graphical calculus, called skein theory, that makes sense in any 3-manifold.
Skein theory for ribbon categories has been formalized in \cite{TuraevBook}. Skein categories and skein module functors associated to cobordisms have been introduced in \cite{WalkerNotes, JohnsonFreydHeisenbergPicture}.
It is a well-known folklore result that these constructions form a categorified TQFT, though it has never been written down as a symmetric monoidal 2-functor. This is the subject of this section.

In order to adapt to the non-semisimple setting, we will also need to consider a tensor ideal $\II$ in a ribbon category $\AA$ as in \cite{CGPAdmissbleskein, BHskcat}. These can also be thought of as a class of $\mathbb E_2^{or}$-algebras, this time in the bicategory $\Pr$, by setting $\EE := \widehat\II$. They correspond to the following:
\begin{definition}
    A cp-ribbon category $\EE\in\Pr$ is an $\mathbb E_2^{or}$-algebra in $\Pr$, i.e. a presentable braided balanced category, such that:
    \begin{itemize}
        \item $\EE$ has enough compact-projectives, i.e. $\EE\simeq \widehat\II$ where $\II$ is the subcategory of compact-projective objects of $\EE$, 
        \item every object of $\II$ is dualizable, and
        \item the rigid balanced category $\AA$ of dualizable objects\footnote{For smallness issues, we take $\II$ and $\AA$ to be small subcategories of all compact-projective and dualizable objects that contain every isomorphism classes.} of $\EE$ is ribbon.
    \end{itemize}
From a cp-ribbon category $\EE$ we extract an inclusion $(\II\subseteq\AA)$ of a tensor ideal in a ribbon category, and $\EE$ can be reconstructed as $\widehat\II$ with tensor product, braiding and balancing induced by those of $\II$. Note that $\II$ will not in general contain the monoidal unit, but the unit of $\widehat\II$ is unique up to isomorphism, or can be reconstructed as $\Hom_\AA(-,\unit)$ using the inclusion into $\AA$.
\end{definition}
The examples coming from the setting of \cite{WalkerNotes} are precisely those where $\II=\AA$, i.e. where the unit of $\EE$ is compact projective.

Let us recall the basic definitions of skein theory, adapted to the non-semisimple setting. Details can be found in \cite{BHskcat}, see also \cite{CGPAdmissbleskein, RunkelSchweigertThamExciAdmSkeins, TuraevBook, CookeExcision, GJS}.
\begin{definition}
An \textbf{$\II$-labeling} $X$ in a closed surface $\Sigma$ is a collection of $\II$-colored framed oriented points in $\Sigma$. It is called \textbf{admissible} if there is at least one point per connected component of $\Sigma$.

An \textbf{$\II$-colored ribbon graph} $\oT$ compatible with two $\II$-labellings $X\subseteq \Sigma_-$ and $Y\subseteq \Sigma_+$ in a 3-cobordism $M:\Sigma_-\to\Sigma_+$ is the image of an embedding  $\Gamma \hookrightarrow M$ of a finite oriented graph $\Gamma$ equipped with a smooth framing, with edges colored by objects of $\II$, inner vertices colored by appropriate morphisms in $\II$ and boundary vertices matching the colored oriented framed points $X$ and $Y$. Sometimes we will draw coupons instead of framed vertices, see Figure \ref{fig:VertexCoupon}. It is called \textbf{admissible} if $\Gamma \inj M$ is surjective on connected components.

In our setting where cobordisms are equipped with a collar $\Sigma_\pm\times [\pm1,\pm\frac{1}{2})\inj M$, we also require that $\Gamma$ is strictly vertical inside the collar, i.e. $\Gamma \cap (\Sigma_-\times [-1,-\frac{1}{2})) = X \times [-1, -\frac{1}{2})$ and $\Gamma \cap (\Sigma_+\times (\frac{1}{2},1]) = Y \times (\frac{1}{2},1]$. This replaces the weaker transversality requirement of \cite{BHskcat} but does not affect the skein module where these are considered up to isotopy, as any ribbon graph transverse to the boundary is isotopic in an essentially unique way to one that is vertical on the collars. See Figure \ref{fig:RibbonGraph}.

The \textbf{relative admissible skein module} $\Sk_\II(M;X,Y)$ is the vector space freely generated by {isotopy classes} of admissible $\II$-colored ribbon graphs in $M$ compatible with $X$ and $Y$ quotiented by \textbf{admissible skein relation}, which are usual local skein relations happening in a cube $[0,1]^3\inj M$ where we require that the ribbon graphs intersect the boundary of the cube at least once. 

A diffeomorphism $f: M\to M'$ preserving orientation and collars induces an isomorphisms of vector spaces 
    \begin{equation*}
        \begin{array}{rcl}
f_{*}: \Sk_\II(M;X,Y) &\to& \Sk_\II(M';X,Y)\\
        T &\mapsto& f(T)
        \end{array}
    \end{equation*}
    which depends on $f$ only up to isotopy.
\end{definition}

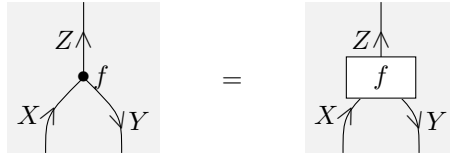
\begin{figure}
    \centering
\begin{tikzpicture}[baseline = 25pt]
\fill[gray!10] (0,0) rectangle (2,2);
\draw (0.5,0)..controls (0.5,0.5)..(1,1) node[midway, sloped]{$>$} node[midway, left]{$X$} node{$\bullet$} node[right]{$f$} -- (1,2)node[midway, sloped]{$>$} node[midway, left]{$Z$};
\draw (1.5,0)..controls (1.5,0.5)..(1,1) node[pos=0.4, sloped]{$>$} node[pos=0.4, right]{$Y$};
\end{tikzpicture} 
$\quad\quad = \quad\quad$
\begin{tikzpicture}[baseline = 25pt]
\fill[gray!10] (0,0) rectangle (2,2);
\draw (0.5,0)..controls (0.5,0.5)..(1,1) node[midway, sloped]{$>$} node[midway, left]{$X$} -- (1,2)node[midway, sloped]{$>$} node[midway, left]{$Z$};
\draw (1.5,0)..controls (1.5,0.5)..(1,1) node[pos=0.4, sloped]{$>$} node[pos=0.4, right]{$Y$};
\node[rectangle, draw, fill=white] at (1,1) {$\ \ f\ \ $};
\end{tikzpicture}
\caption{Left: A framed vertex (with blackboard framing, coming out of the page) colored by a morphism 
    $f \in \Hom_\AA(\unit, X^*\otimes Z \otimes Y) \simeq \Hom_\AA(X, Z \otimes Y) \simeq \Hom_\AA( X\otimes Y^*, Z) \simeq  \cdots$
\\
Right: A coupon representing the same morphism $f \in  \Hom_\AA( X\otimes Y^*, Z)$.}
    \label{fig:VertexCoupon}
\end{figure}
\begin{figure}
    \centering
\begin{tikzpicture}    
    \fill[gray!10] (1,1) -- (4,1)..controls (6,1) and (6,1.3) .. (9,1.5) -- (9,-1.5) ..controls (6,-1.3) and (6,-1).. (4,-1) -- (1,-1) arc (270:90:1);
    \draw[gray] (0,0) arc(180:540:1);
    \draw[gray] (0,0) arc(180:360:1 and 0.5);
    \draw[dashed, gray] (0,0) arc(180:0:1 and 0.5);
    \draw[gray!30] (3,0) arc(180:540:1);
    \draw[gray!30] (3,0) arc(180:360:1 and 0.5);
    \draw[dashed, gray!30] (3,0) arc(180:0:1 and 0.5);
    \draw[gray] (1,1) -- (4,1)..controls (6,1) and (6,1.3) .. (9,1.5);
    \draw[gray] (1,-1) -- (4,-1)..controls (6,-1) and (6,-1.3) .. (9,-1.5);
    \node[gray] at (8.5,1.2){$M$};

    \draw (0.7,0.3) node{$\bullet$} node[above]{$Y,-$} -- (3.7,0.3) node[pos = 0.5, sloped]{$<$} node[pos = 0.5, above]{$Y$} ..controls (4.7, 0.3) and (4.7,-0.7) .. (6,0) node{$\bullet$} node[below =10pt, right = -10pt]{$f:X\otimes Y^* \to Z$} -- (9,0)node[pos = 0.5, sloped]{$>$} node[pos = 0.5, above]{$Z$};
    \draw[gray!10, line width = 5pt] (4,-0.6) ..controls (5,-0.6) and (5,0.5) .. (6,0.3);
    \draw (1,-0.6) node{$\bullet$} node[above]{$X,+$} -- (4,-0.6) node[pos = 0.5, sloped]{$>$} node[pos = 0.5, above]{$X$} ..controls (5,-0.6) and (5,0.5) .. (6,0);

    \node[xscale = 8, rotate = -90, gray] at (2.5,1.3){$\{$};
    \node[gray] at (2.5,1.6){$\Sigma_- \times [-1,-\frac{1}{2})$};
\end{tikzpicture}
\caption{An $\II$-colored ribbon graph in a 3-manifold with collared boundary.}
    \label{fig:RibbonGraph}
\end{figure}
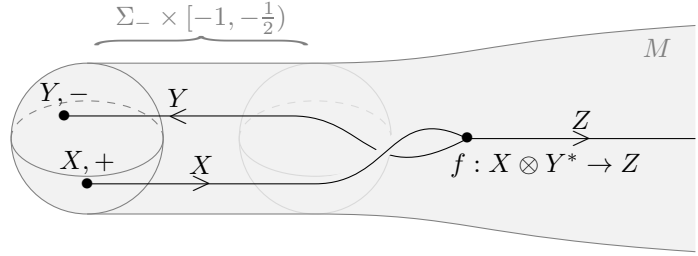
\begin{definition} 
Given composeable cobordisms $\Sigma_1 \overset{M_{12}} \to \Sigma_2 \overset{M_{23}}\to \Sigma_3$ and $\II$-labelings $X_1\subseteq \Sigma_1,\ X_2\subseteq \Sigma_2, X_3\subseteq \Sigma_3$ the \textbf{gluing of skeins} is the linear map 
\begin{equation*}
\begin{array}{rclcl}
\Sk_\II(M_{12};X_1,X_2) &\otimes &\Sk_\II(M_{23};X_2,X_3)& \to &\Sk_\II(M_{23}\circ M_{12};X_1,X_3)\\
T & \otimes & T' &\mapsto & \oT \cup \oT'
\end{array}
    \end{equation*}
where $\oT,\ \oT'$ are ribbon graph representatives of the skeins $T, T'$, and $\oT \cup \oT'\subseteq M_{12} \underset{\Sigma_2\times I}{\cup }M_{23} = M_{23}\circ M_{12}$ is a ribbon graph as $\oT$ and $\oT'$ are both vertical in the collars of $\Sigma_2$ and glue smoothly. See Figure \ref{fig:GluingSkeins}.
\end{definition}
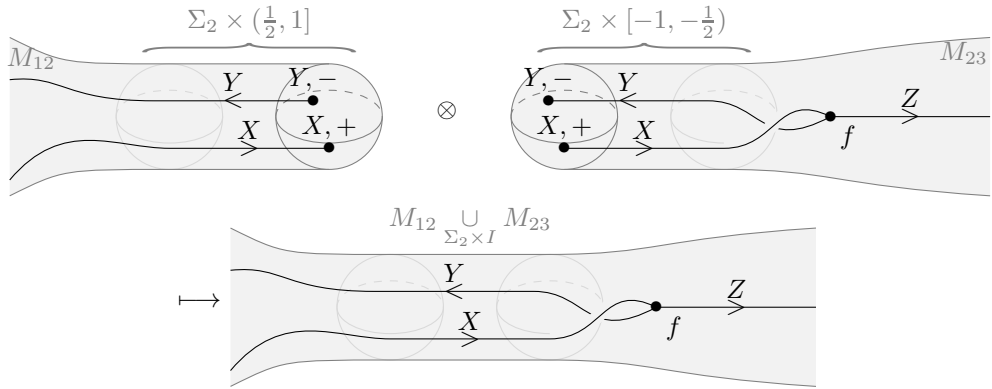
\begin{figure}
    \centering
\begin{tikzpicture} [scale = 0.7, baseline = 0pt]  
    \fill[gray!10] (-2,1.5).. controls (-1,1).. (1,1) -- (4,1) arc (90:-90:1) --  (1,-1).. controls (-1,-1).. (-2,-1.5)--(-2,1.5);
    \draw[gray!30] (0,0) arc(180:540:1);
    \draw[gray!30] (0,0) arc(180:360:1 and 0.5);
    \draw[dashed, gray!30] (0,0) arc(180:0:1 and 0.5);
    \draw[gray] (3,0) arc(180:540:1);
    \draw[gray] (3,0) arc(180:360:1 and 0.5);
    \draw[dashed, gray] (3,0) arc(180:0:1 and 0.5);
    \draw[gray](-2,1.5).. controls (-1,1).. (1,1) -- (4,1);
    \draw[gray](-2,-1.5).. controls (-1,-1)..  (1,-1) -- (4,-1);
    \node[gray] at (-1.6,1.1){$M_{12}$};

    \draw (-2,0.7)..controls (-1,0.8) and (-1,0.3).. (0.7,0.3) -- (3.7,0.3) node[pos = 0.5, sloped]{$<$} node[pos = 0.5, above]{$Y$} node{$\bullet$} node[above]{$Y,-$};
    \draw (-2,-1.2)..controls (-1,0) and (0,-0.6)..  (1,-0.6) -- (4,-0.6) node[pos = 0.5, sloped]{$>$} node[pos = 0.5, above]{$X$} node{$\bullet$} node[above]{$X,+$};

    \node[xscale = 8, rotate = -90, gray] at (2.5,1.3){$\{$};
    \node[gray] at (2.5,1.8){$\Sigma_2 \times (\frac{1}{2},1]$};
\end{tikzpicture} 
$ \otimes  $
\begin{tikzpicture}  [scale = 0.7, baseline = 0pt]  
    \fill[gray!10] (1,1) -- (4,1)..controls (6,1) and (6,1.3) .. (9,1.5) -- (9,-1.5) ..controls (6,-1.3) and (6,-1).. (4,-1) -- (1,-1) arc (270:90:1);
    \draw[gray] (0,0) arc(180:540:1);
    \draw[gray] (0,0) arc(180:360:1 and 0.5);
    \draw[dashed, gray] (0,0) arc(180:0:1 and 0.5);
    \draw[gray!30] (3,0) arc(180:540:1);
    \draw[gray!30] (3,0) arc(180:360:1 and 0.5);
    \draw[dashed, gray!30] (3,0) arc(180:0:1 and 0.5);
    \draw[gray] (1,1) -- (4,1)..controls (6,1) and (6,1.3) .. (9,1.5);
    \draw[gray] (1,-1) -- (4,-1)..controls (6,-1) and (6,-1.3) .. (9,-1.5);
    \node[gray] at (8.5,1.2){$M_{23}$};

    \draw (0.7,0.3) node{$\bullet$} node[above]{$Y,-$} -- (3.7,0.3) node[pos = 0.5, sloped]{$<$} node[pos = 0.5, above]{$Y$} ..controls (4.7, 0.3) and (4.7,-0.7) .. (6,0) node{$\bullet$} node[below right]{$f$} -- (9,0)node[pos = 0.5, sloped]{$>$} node[pos = 0.5, above]{$Z$};
    \draw[gray!10, line width = 5pt] (4,-0.6) ..controls (5,-0.6) and (5,0.5) .. (6,0.3);
    \draw (1,-0.6) node{$\bullet$} node[above]{$X,+$} -- (4,-0.6) node[pos = 0.5, sloped]{$>$} node[pos = 0.5, above]{$X$} ..controls (5,-0.6) and (5,0.5) .. (6,0);

    \node[xscale = 8, rotate = -90, gray] at (2.5,1.3){$\{$};
    \node[gray] at (2.5,1.8){$\Sigma_2 \times [-1,-\frac{1}{2})$};
\end{tikzpicture} \\
$\longmapsto$
\begin{tikzpicture}  [scale = 0.7, baseline = 0pt]      
    \fill[gray!10] (-2,1.5).. controls (-1,1).. (1,1) -- (4,1) arc (90:-90:1) --  (1,-1).. controls (-1,-1).. (-2,-1.5)--(-2,1.5);
    \fill[gray!10] (1,1) -- (4,1)..controls (6,1) and (6,1.3) .. (9,1.5) -- (9,-1.5) ..controls (6,-1.3) and (6,-1).. (4,-1) -- (1,-1) arc (270:90:1);
    \draw[gray!30] (0,0) arc(180:540:1);
    \draw[gray!30] (0,0) arc(180:360:1 and 0.5);
    \draw[dashed, gray!30] (0,0) arc(180:0:1 and 0.5);
    \draw[gray!30] (3,0) arc(180:540:1);
    \draw[gray!30] (3,0) arc(180:360:1 and 0.5);
    \draw[dashed, gray!30] (3,0) arc(180:0:1 and 0.5);
    \draw[gray](-2,1.5).. controls (-1,1)..  (1,1) -- (4,1)..controls (6,1) and (6,1.3) .. (9,1.5);
    \draw[gray] (-2,-1.5).. controls (-1,-1)..  (1,-1) -- (4,-1)..controls (6,-1) and (6,-1.3) .. (9,-1.5);
    \node[gray] at (2.5,1.5){$M_{12}\underset{\Sigma_2\times I}{\cup} M_{23}$};

    \draw  (-2,0.7)..controls (-1,0.8) and (-1,0.3).. (0.7,0.3) -- (3.7,0.3) node[pos = 0.5, sloped]{$<$} node[pos = 0.5, above]{$Y$} ..controls (4.7, 0.3) and (4.7,-0.7) .. (6,0) node{$\bullet$} node[below right]{$f$} -- (9,0)node[pos = 0.5, sloped]{$>$} node[pos = 0.5, above]{$Z$};
    \draw[gray!10, line width = 5pt] (4,-0.6) ..controls (5,-0.6) and (5,0.5) .. (6,0.3);
    \draw (-2,-1.2)..controls (-1,0) and (0,-0.6)..  (1,-0.6) -- (4,-0.6) node[pos = 0.5, sloped]{$>$} node[pos = 0.5, above]{$X$} ..controls (5,-0.6) and (5,0.5) .. (6,0);
\end{tikzpicture}
\caption{The gluing of skeins.}
    \label{fig:GluingSkeins}
\end{figure}
\begin{definition}
The \textbf{skein category} $\SkCat_\II(\Sigma)$ of a surface $\Sigma$ has:
\begin{description}
    \item Objects: Admissible $\II$-labelings in $\Sigma$
    \item Morphisms: The relative admissible skein module $\Sk_\II(\Sigma \times [-1,1];X,Y)$
    \item Composition: Gluing of skeins, i.e. $$\Sk_\II(\Sigma \times [-1,1];X_1,X_2) \otimes \Sk_\II(\Sigma \times [-1,1];X_2,X_3) \to \Sk_\II(\Sigma \times [-1,1];X_1,X_3)$$
    using the unitor diffeomorphism $\Sigma \times [-1,1]\circ \Sigma \times [-1,1] \simeq \Sigma \times [-1,1]$.
\end{description}

The \textbf{admissible skein bimodule functor} of $M:\Sigma_-\to \Sigma_+$ is the functor \begin{equation}
    \begin{aligned}
\SkFun_\II(M): \SkCat_\II(\Sigma_+)\otimes\SkCat_\II(\Sigma_-)^{op} &\to \Vect\\
(Y,X)&\mapsto \Sk_\II(M;X,Y)
    \end{aligned}
        \end{equation}
The action of morphisms in $\SkCat_\II(\Sigma_-)$ and $\SkCat_\II(\Sigma_+)$ is induced by gluing of skeins and unitor diffeomorphisms. For $T_\pm$ morphisms in $\SkCat_\II(\Sigma_\pm)$ and $S$ a skein in $M$ with appropriate endpoints, we will denote $T_-\cdot S \cdot T_+$ the skein in $M$ obtained by acting on $S$.

A diffeomorphism $f: M\to M'$ defines a natural isomorphism $f_*: \SkFun_\II(M) \Rightarrow \SkFun_\II(M')$ whose components has been defined above. It depends on $f$ only up to isotopy.
\end{definition}

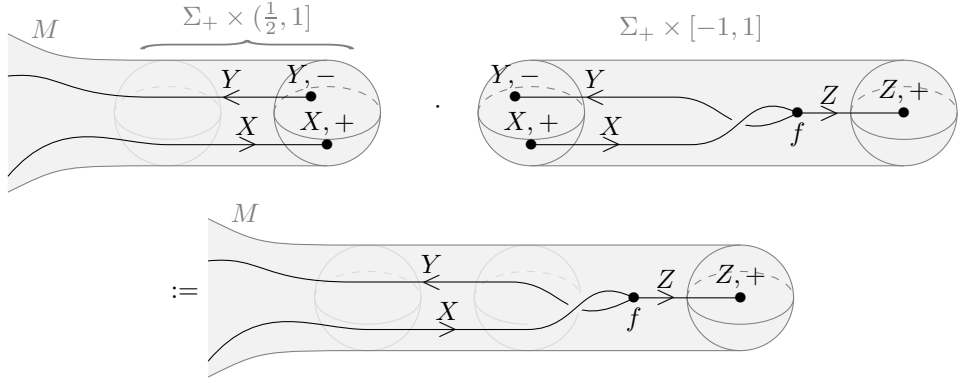
\begin{figure}
    \centering
\begin{tikzpicture} [scale = 0.7, baseline = 0pt]  
    \fill[gray!10] (-2,1.5).. controls (-1,1).. (1,1) -- (4,1) arc (90:-90:1) --  (1,-1).. controls (-1,-1).. (-2,-1.5)--(-2,1.5);
    \draw[gray!30] (0,0) arc(180:540:1);
    \draw[gray!30] (0,0) arc(180:360:1 and 0.5);
    \draw[dashed, gray!30] (0,0) arc(180:0:1 and 0.5);
    \draw[gray] (3,0) arc(180:540:1);
    \draw[gray] (3,0) arc(180:360:1 and 0.5);
    \draw[dashed, gray] (3,0) arc(180:0:1 and 0.5);
    \draw[gray](-2,1.5).. controls (-1,1).. (1,1) -- (4,1);
    \draw[gray](-2,-1.5).. controls (-1,-1)..  (1,-1) -- (4,-1);
    \node[gray] at (-1.3,1.6){$M$};

    \draw (-2,0.7)..controls (-1,0.8) and (-1,0.3).. (0.7,0.3) -- (3.7,0.3) node[pos = 0.5, sloped]{$<$} node[pos = 0.5, above]{$Y$} node{$\bullet$} node[above]{$Y,-$};
    \draw (-2,-1.2)..controls (-1,0) and (0,-0.6)..  (1,-0.6) -- (4,-0.6) node[pos = 0.5, sloped]{$>$} node[pos = 0.5, above]{$X$} node{$\bullet$} node[above]{$X,+$};

    \node[xscale = 8, rotate = -90, gray] at (2.5,1.3){$\{$};
    \node[gray] at (2.5,1.8){$\Sigma_+ \times (\frac{1}{2},1]$};
\end{tikzpicture} 
$ \cdot \quad $
\begin{tikzpicture}  [scale = 0.7, baseline = 0pt]  
    \fill[gray!10] (1,1) -- (8,1) arc(90:-90:1) -- (1,-1) arc (270:90:1);
    \draw[gray] (7,0) arc(180:540:1);
    \draw[gray] (7,0) arc(180:360:1 and 0.5);
    \draw[dashed, gray] (7,0) arc(180:0:1 and 0.5);
    
    \draw[gray] (0,0) arc(180:540:1);
    \draw[gray] (0,0) arc(180:360:1 and 0.5);
    \draw[dashed, gray] (0,0) arc(180:0:1 and 0.5);
    
    \draw[gray] (1,1) -- (8,1);
    \draw[gray] (1,-1) -- (8,-1);

    \draw (0.7,0.3) node{$\bullet$} node[above]{$Y,-$} -- (3.7,0.3) node[pos = 0.5, sloped]{$<$} node[pos = 0.5, above]{$Y$} ..controls (4.7, 0.3) and (4.7,-0.7) .. (6,0) node{$\bullet$} node[below]{$f$} -- (8,0)node[pos = 0.3, sloped]{$>$} node[pos = 0.3, above]{$Z$} node{$\bullet$} node[above]{$Z,+$};
    \draw[gray!10, line width = 5pt] (4,-0.6) ..controls (5,-0.6) and (5,0.5) .. (6,0.3);
    \draw (1,-0.6) node{$\bullet$} node[above]{$X,+$} -- (4,-0.6) node[pos = 0.5, sloped]{$>$} node[pos = 0.5, above]{$X$} ..controls (5,-0.6) and (5,0.5) .. (6,0);

    \node[gray] at (4,1.6){$\Sigma_+ \times [-1,1]$};
\end{tikzpicture} \\
$:=$
\begin{tikzpicture}  [scale = 0.7, baseline = 0pt]      
    \fill[gray!10] (-2,1.5).. controls (-1,1).. (1,1) -- (4,1) arc (90:-90:1) --  (1,-1).. controls (-1,-1).. (-2,-1.5)--(-2,1.5);
    \fill[gray!10] (1,1) -- (8,1) arc(90:-90:1) -- (1,-1) arc (270:90:1);
    \draw[gray] (7,0) arc(180:540:1);
    \draw[gray] (7,0) arc(180:360:1 and 0.5);
    \draw[dashed, gray] (7,0) arc(180:0:1 and 0.5);
    
    \draw[gray!30] (0,0) arc(180:540:1);
    \draw[gray!30] (0,0) arc(180:360:1 and 0.5);
    \draw[dashed, gray!30] (0,0) arc(180:0:1 and 0.5);
    \draw[gray!30] (3,0) arc(180:540:1);
    \draw[gray!30] (3,0) arc(180:360:1 and 0.5);
    \draw[dashed, gray!30] (3,0) arc(180:0:1 and 0.5);
    \draw[gray](-2,1.5).. controls (-1,1)..  (1,1) -- (8,1);
    \draw[gray] (-2,-1.5).. controls (-1,-1)..  (1,-1) -- (8,-1);
    \node[gray] at (-1.3,1.6){$M$};

    \draw  (-2,0.7)..controls (-1,0.8) and (-1,0.3).. (0.7,0.3) -- (3.7,0.3) node[pos = 0.5, sloped]{$<$} node[pos = 0.5, above]{$Y$} ..controls (4.7, 0.3) and (4.7,-0.7) .. (6,0) node{$\bullet$} node[below]{$f$} -- (8,0)node[pos = 0.3, sloped]{$>$} node[pos = 0.3, above]{$Z$} node{$\bullet$} node[above]{$Z,+$};
    \draw[gray!10, line width = 5pt] (4,-0.6) ..controls (5,-0.6) and (5,0.5) .. (6,0.3);
    \draw (-2,-1.2)..controls (-1,0) and (0,-0.6)..  (1,-0.6) -- (4,-0.6) node[pos = 0.5, sloped]{$>$} node[pos = 0.5, above]{$X$} ..controls (5,-0.6) and (5,0.5) .. (6,0);
\end{tikzpicture}
\caption{The action of the skein category on the skein module. At the top left is a skein $S \in \Sk_\II(M;X\sqcup Y)$. At the top right is a morphism $T_+ \in \Hom_{\SkCat_\II(\Sigma_+)}(X\sqcup Y, Z)$. At the bottom is the skein $S\cdot T_+ \in \Sk_\II(M;Z)$.}
    \label{fig:ActSkCatSkMod}
\end{figure}
Note that $\SkFun_\II(M)$ is a morphism from $\SkCat_\II(\Sigma_+)$ to $\SkCat_\II(\Sigma_-)$ in $\Bimod$. This contravariance also appears in \cite{WalkerNotes}, where skeins are treated as the dual theory to a TQFT. This is only a nuisance and not a deep issue, since $\operatorname{Cob} \simeq \operatorname{Cob}^{op}$ via orientation reversal. However, instead of using orientation reversal here and there, we will keep this contravariance. We will denote $\Bimod^{hop}$ the opposite bicategory in the horizontal direction (i.e. for 1-morphisms).

The following result has been shown in \cite[Thm. 2.21]{BHskcat} or \cite[Thm. 3.1]{RunkelSchweigertThamExciAdmSkeins} in a 1-categorical setting, see also \cite[Thm. 4.4.2]{WalkerNotes}. The definition of a symmetric monoidal 2-functor is recalled in \cite[Def. A.5 and 2.5]{SPPhD}.
\begin{theorem}\label{thm:SkCategTQFT}
Given a tensor ideal $\II$ in a ribbon category $\AA$, there exists a categorified TQFT 
    $$\SkFun_\II:\Cob_{2+1+\varepsilon} \to \Bimod^{hop}$$
    with $\SkFun_\II(\Sigma) = \SkCat_\II(\Sigma),\ \SkFun_\II(M) = \SkFun_\II(M)$ and $\SkFun_\II(f) = f_*$.
\end{theorem}
\begin{proof} 
We must exhibit for any composable pair $\Sigma_1 \overset{M_{12}} \to \Sigma_2 \overset{M_{23}}\to \Sigma_3$ an isomorphism 
$$\SkFun(M_{23}) \circ \SkFun(M_{12}) \Tilde{\to}\SkFun(M_{23}\circ M_{12})$$
compatible with 2-morphisms in $M$ and $M'$.

This is the data for any objects $X_1, X_3$ of $\SkCat(\Sigma_1), \SkCat(\Sigma_3)$ of an isomorphism
    \begin{equation*}\label{eq:skfun-composes}
            \int^{X_2 \in \SkCat(\Sigma_2)} \Sk_\II(M_{12};X_1,X_2) \otimes \Sk_\II(M_{23};X_2,X_3) \tilde\to \Sk_\II(M_{12}\underset{\Sigma_2}{\cup}M_{23};X_1,X_3)\ .
    \end{equation*}
    This map is given by gluing of skeins. It descends to a morphism on the coend as the coend relations are induced by an isotopy $\varphi$ of $M_{12}\underset{\Sigma_2}{\cup}M_{23}$ which pushes a collar of $\Sigma_2$ in $M_{12}$ into a collar of $\Sigma_2$ in $M_{23}$, see Figure \ref{fig:CoendIsotopy}. 
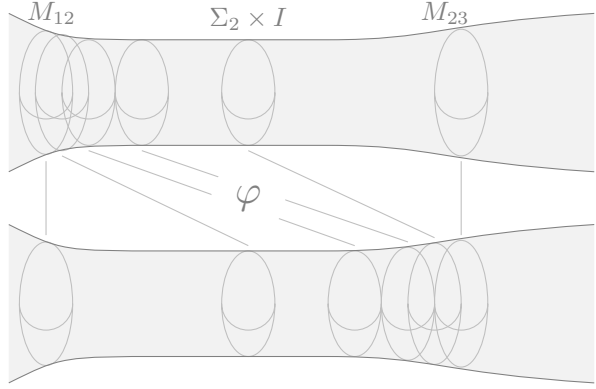
\begin{figure}
    \centering
\begin{tikzpicture}  [scale = 0.7, baseline = 0pt]      
    \fill[gray!10] (-2,1.5).. controls (-1,1).. (1,1) -- (4,1) arc (90:-90:1) --  (1,-1).. controls (-1,-1).. (-2,-1.5)--(-2,1.5);
    \fill[gray!10] (1,1) -- (4,1)..controls (6,1) and (6,1.3) .. (9,1.5) -- (9,-1.5) ..controls (6,-1.3) and (6,-1).. (4,-1) -- (1,-1) arc (270:90:1);
    \draw[gray!50] (-1.8,0) arc(180:540:0.5 and 1.17) arc(180:360:0.5 and 0.5);
    \draw[gray!50] (-1.5,0) arc(180:540:0.5 and 1.1) arc(180:360:0.5 and 0.5);
    \draw[gray!50] (-1,0) arc(180:540:0.5 and 1) arc(180:360:0.5 and 0.5);
    \draw[gray!50] (0,0) arc(180:540:0.5 and 1) arc(180:360:0.5 and 0.5);
    \draw[gray!50] (2,0) arc(180:540:0.5 and 1) arc(180:360:0.5 and 0.5);
    \draw[gray!50] (6,0) arc(180:540:0.5 and 1.2) arc(180:360:0.5 and 0.5);
    \draw[gray](-2,1.5).. controls (-1,1)..  (1,1) -- (4,1)..controls (6,1) and (6,1.3) .. (9,1.5);
    \draw[gray] (-2,-1.5).. controls (-1,-1)..  (1,-1) -- (4,-1)..controls (6,-1) and (6,-1.3) .. (9,-1.5);
    \node[gray] at (-1.2,1.5){$M_{12}$};
    \node[gray] at (2.5,1.4){${\Sigma_2\times I}$};
    \node[gray] at (6.2,1.5){$ M_{23}$};
\begin{scope}[yshift = -4cm]
    \fill[gray!10] (-2,1.5).. controls (-1,1).. (1,1) -- (4,1) arc (90:-90:1) --  (1,-1).. controls (-1,-1).. (-2,-1.5)--(-2,1.5);
    \fill[gray!10] (1,1) -- (4,1)..controls (6,1) and (6,1.3) .. (9,1.5) -- (9,-1.5) ..controls (6,-1.3) and (6,-1).. (4,-1) -- (1,-1) arc (270:90:1);
    \draw[gray!50] (-1.8,0) arc(180:540:0.5 and 1.17) arc(180:360:0.5 and 0.5);
    \draw[gray!50] (2,0) arc(180:540:0.5 and 1) arc(180:360:0.5 and 0.5);
    \draw[gray!50] (4,0) arc(180:540:0.5 and 1) arc(180:360:0.5 and 0.5);
    \draw[gray!50] (5,0) arc(180:540:0.5 and 1.07) arc(180:360:0.5 and 0.5);
    \draw[gray!50] (5.5,0) arc(180:540:0.5 and 1.15) arc(180:360:0.5 and 0.5);
    \draw[gray!50] (6,0) arc(180:540:0.5 and 1.2) arc(180:360:0.5 and 0.5);
    \draw[gray](-2,1.5).. controls (-1,1)..  (1,1) -- (4,1)..controls (6,1) and (6,1.3) .. (9,1.5);
    \draw[gray] (-2,-1.5).. controls (-1,-1)..  (1,-1) -- (4,-1)..controls (6,-1) and (6,-1.3) .. (9,-1.5);        
\end{scope}
\draw[gray!50, thin](-1.3,-1.3) -- (-1.3,-2.7);
\draw[gray!50, thin](-1,-1.2) -- (2.5,-2.9);
\draw[gray!50, thin](-0.5,-1.1) -- (4.5,-2.9);
\draw[gray!50, thin](0.5,-1.1) -- (5.5,-2.83);
\draw[gray!50, thin](2.5,-1.1) -- (6,-2.75);
\draw[gray!50, thin](6.5,-1.3) -- (6.5,-2.7);
    \node[gray, circle, fill=white, scale = 1.5] at (2.5,-2){$\varphi$};
\end{tikzpicture}  
\caption{The isotopy $\varphi$ of $M_{12}\underset{\Sigma_2}{\cup}M_{23}$ realizing the coend relations. It intertwines the two maps $M_{12}\underset{\Sigma_2\times I}{\cup}\Sigma_2\times [-1,1]\underset{\Sigma_2\times I}{\cup}M_{23} \to M_{12}\underset{\Sigma_2\times I}{\cup}M_{23}$ induced by the unitor diffeomorphisms for $M_{12}$ and $M_{23}$.}
    \label{fig:CoendIsotopy}
\end{figure}
    This map is shown to be an isomorphism in \cite[Thm. 2.21]{BHskcat}, \cite[Thm. 3.1]{RunkelSchweigertThamExciAdmSkeins}. Let us recall the proof for the reader's convenience. It happens in two steps: 

    First, we apply \cite[Lem. 2.11]{BHskcat} to reduce the RHS to the space of skeins in $M_{12}\underset{\Sigma_2}{\cup}M_{23}$ which intersect $\Sigma_2$ at least once on every connected components, modulo isotopies preserving this property. This way we need not worry about admissibility conditions below.
    
    Second, we follow the arguments of \cite[Thm. 4.4.2]{WalkerNotes}. The gluing of skeins is surjective as any skein in generic position in $M_{12}\underset{\Sigma_2}{\cup}M_{23}$ will intersect $\Sigma_2$ transversely, and can then be isotoped to be completely vertical in a collar of $\Sigma_2$. For injectivity, we need to prove that two skeins that are isotopic in $M_{12}\underset{\Sigma_2}{\cup}M_{23}$ are related by isotopies supported away from the collar of $\Sigma_2$ and coend relations. Any isotopy can be decomposed into isotopies supported in small balls, and up to conjugating by the isotopy $\varphi$ described above (i.e. using coend relations), these balls can be pushed to be away from the collar of $\Sigma_2$.
    
    This isomorphism is natural with respect to diffeomorphisms in $M$ and $M'$ preserving collars as $(f \cup f')(\oT \cup \oT') = f(\oT) \cup f'(\oT')$.

    By definition, we also have $$\SkFun_\II(\id_\Sigma) = \Hom_{\SkCat_\II(\Sigma)}(-,-) = \id_{\SkCat_\II(\Sigma)}$$
    and these isomorphisms are readily checked to be compatible with associators and unitors. This proves that $\SkFun_\II$ is a 2-functor.

    We now turn to symmetric monoidality. We have an isomorphism of vector spaces 
    $$\Sk_\II(M;X,Y)\otimes \Sk_\II(M';X',Y') \tilde\to \Sk_\II(M\sqcup M'; X\sqcup X', Y\sqcup Y')$$
    simply given by disjoint union of ribbon graphs. It induces an isomorphism of categories $$\SkCat_\II(\Sigma)\otimes \SkCat_\II(\Sigma') \tilde\to \SkCat_\II(\Sigma\sqcup \Sigma')$$ which is disjoint union on objects, and a natural isomorphism $$\SkFun_\II(M)\otimes \SkFun_\II(M') \tilde\Rightarrow \SkFun_\II(M\sqcup M')$$ where we implicitly used the isomorphism of categories above to match the source and target. 

    The skein category of the empty surface has only one object, the empty collection of points, with endomorphisms scalars times the empty ribbon graph, its identity, which is indeed the monoidal unit in $\Bimod$.

    All the coherence modifications of \cite[Def. 2.5]{SPPhD} are identities.
\end{proof}

\section{Extended non-semisimple Crane--Yetter}\label{Sec_ExtendedCY}
In this section we will extend the categorified TQFT 
    $$\SkFun_\II:\Cob_{2+1+\varepsilon} \to \Bimod^{hop}$$
into a possibly non-compact once-extended 4-TQFT
    $$\SS_\II:\Cob_{2+1+1} \to \Bimod^{hop}$$
under some additional conditions and structure on the category $\AA$. 

We will construct this TQFT by specifying its values on the standard attachments of 0--4-handles, check that they satisfy handle cancellation and invariance under reversal of the attaching spheres and use the main result of \cite{HaiounHandle}.

This TQFT is a once-extended version of \cite{CGHP}, and the values on handle attachments are constructed there. We will recall the definitions for the readers convenience, but refer to \cite{CGHP} for details.

\subsection{Hypothesis and structure on the input category}\label{Subs_HypothCat}
In this section, $\AA$ is a \textbf{finite ribbon tensor category} in the sense of \cite{EGNO} over an algebraically closed field $\Bbbk$ and $\II \subseteq \AA$ is the tensor ideal of projective objects. We denote $P_\unit \in \II$ a projective cover of the unit, equipped with its projection $\varepsilon_\unit: P_\unit \to \unit$. For any projective object $P$, let $s_P: P\to P\otimes P_\unit$ be a section of $\id_P \otimes \varepsilon_\unit: P\otimes P_\unit \to P$. By definition this means that $\id_P = (\id_P \otimes \varepsilon_\unit)\circ s_P$ which at the levels of skeins means that we can introduce a $P_\unit$-colored strand ending with a $\varepsilon_\unit$ whenever there is a projective-colored strand.

We further assume that $\AA$ is \textbf{unimodular} which is equivalent to asking that $P_\unit$ is self dual. This implies that $\AA$ has a non-degenerate modified trace
$$\mt_P: \End_\AA(P) \to \Bbbk\ , \quad P\in\II$$
which is unique up to scalar by \cite[Cor. 5.6]{GKPmtrace}. We assume that a choice of modified trace has been made. It fixes the choice of a morphism $\eta_\unit: \unit\to P_\unit$ such that $\mt_{P_\unit}(\eta_\unit\circ\varepsilon_\unit)=1$. The modified trace being non-degenerate implies that it induces a non-degenerate pairing
$$\mt_P(-\circ-): \Hom_\AA(\unit,P)\otimes \Hom_\AA(P,\unit)\to \Bbbk$$
We denote $$\Omega_P = \sum_i x^i \otimes x_i \in \Hom_\AA(P,\unit) \otimes \Hom_\AA(\unit, P)$$ the associated copairing, i.e. $(x_i)_i$ and $(x^i)_i$ are dual basis with respect to $\mt_P$. We also denote 
$$\Lambda_P := \sum_i x_i\circ x^i \in \End_\AA(P)\ .$$
Let $G \in \II$ be a projective generator of $\AA$, e.g. take $G$ to be the direct sum of all irreducible projectives. By \cite{CGPVchromatic}, or \cite[Thm. 1.10]{CGHP}, there exist a chromatic morphism based at any $P \in \II$ \cite[Sec. 1.3]{CGHP}, i.e. a morphism
$$\chr_P : G\otimes P\to G\otimes P$$
satisfying 
\begin{equation*}
    \begin{tikzpicture}[scale = 0.8, baseline = 15pt]
        \draw (0,-0.2)--++(0,2.4) node[pos = 0.2,left]{$V$} node[pos = 0.8,left]{$V$} node[pos = 0.2,sloped]{$>$} node[pos = 0.8,sloped]{$>$};
        \draw (0.5,1) arc(180:540:1) node[pos = 0.2,below]{$G$} node[pos = 0.8,above]{$G$} node[pos = 0.2,sloped]{$>$} node[pos = 0.8,sloped]{$<$};
        \draw (3,-0.2)--++(0,2.4) node[pos = 0.2,left]{$P$} node[pos = 0.8,left]{$P$} node[pos = 0.2,sloped]{$>$} node[pos = 0.8,sloped]{$>$};
        \node[rectangle, draw=black, fill=white] at (0.25,1) {$\Lambda_{V\otimes G^*}$};
        \node[rectangle, draw=black, fill=white] at (2.75,1) {$\ \chr_P\ $};
    \end{tikzpicture}\quad \ \  = \quad
    \begin{tikzpicture}[scale = 0.8, baseline = 15pt]
        \draw (0,-0.2)--++(0,2.4) node[pos = 0.5,left]{$V$} node[pos = 0.5,sloped]{$>$};
        \draw (1,-0.2)--++(0,2.4) node[pos = 0.5,left]{$P$}node[pos = 0.5,sloped]{$>$};
    \end{tikzpicture} 
\end{equation*} 
for any $V\in \AA$.

We further assume that $\AA$ is \textbf{chromatic non-degenerate} in the sense that the morphism 
\begin{equation*} \Delta^0_{P_\unit} :=
    \begin{tikzpicture}[baseline = 10pt]
        \draw (0,0) ..controls (0,0.5) and (1,0.5).. (1,1);

        \draw (1,-1) -- (1,0)node[pos = 0.5,left]{$P_\unit$}node[pos = 0.5,sloped]{$>$};
        \draw[white, line width = 5pt] (1,0) ..controls (1,0.5) and (0,0.5).. (0,1);
        \draw (1,0) ..controls (1,0.5) and (0,0.5).. (0,1);
        \draw (0,1) ..controls (0,1.5) and (1,1.5).. (1,2);

        \draw[white, line width = 5pt] (1,1) ..controls (1,1.5) and (-1,1.5).. (-1,0.5);
        \draw (1,1) ..controls (1,1.5) and (-1,1.5).. (-1,0.5)node[pos = 0.6,above]{$G$}node[pos = 0.6,sloped]{$<$};
        \draw (-1,0.5) ..controls (-1,-1) and (0,-1).. (0,0);
        \node[rectangle, draw = black, fill=white] at (0.5,0) {$\ \ \ \chr_{P_\unit}\ \ \ $};
    \end{tikzpicture}
\end{equation*}
is non-zero. This implies the existence of a gluing morphism \cite[Def 1.5]{CGHP}
$$\gm: P_\unit \to P_\unit$$
satisfying 
$$\gm \circ \Delta^0_{P_\unit} = \Lambda_{P_\unit}$$

We will sometimes assume that $\AA$ is even \textbf{chromatic compact} in the sense that $\varepsilon_\unit \circ \gm$ is non-zero. This implies that there exists a non-zero global dimension 
$$\zeta \in \Bbbk^\times$$
such that $\zeta^{-1} \id_{P_\unit}$ is a gluing morphism.

\begin{example}
    If $\AA$ is semisimple and $(S_i)_{i=1,\dots,n}$ are its simples, with $S_0 = \unit$, then:
    
    -- $P_\unit = \unit$, 
    
    -- any scalar times the usual categorical trace $\lambda \operatorname{tr}$ is a modified trace, 
    
    -- the copairing is $\Omega_{S_i} = \lambda^{-1} \delta_{i,0} \id\otimes \id$,
    
    -- a projective generator is given by $G=\oplus_i S_i$, 
    
    -- $\chr := \lambda\cdot \oplus_i \operatorname{qdim}(S_i) \id_{S_i}$ is a chromatic morphism based at $P_\unit$, 
    
    -- $\Delta^0_{P_\unit} = \lambda\sum_i \operatorname{qdim}(S_i)^2 \id_\unit$ which is non-zero when $\operatorname{char}\Bbbk=0$ or when $\AA$ is separable
    
    -- in this case, $\gm = \frac{1}{\lambda^2\sum_i \operatorname{qdim}(S_i)^2}\id_\unit$ is a gluing morphism, and

    -- the global dimension is $\zeta = \lambda^2\sum_i \operatorname{qdim}(S_i)^2$.\\
    Note that there exists precisely two values of $\lambda$ for which $\zeta = 1$.
\end{example}

\subsection{The construction}
We give an operation on admissible skein module associated to each 4-dimensional handle attachment, and then check that they define an extended TQFT.
\paragraph{The 4-handle}
We define a natural transformation $$Z_4: \SkFun_\II(S^3) \Rightarrow \SkFun_\II(\emptyset)$$
In this case, the incoming and outgoing boundary of $S^3$ are both the empty surface, and $\SkFun_\II(S^3): \SkCat(\emptyset)\otimes \SkCat_\II(\emptyset)^{op}\to \Vect$ is just the data of one vector space $\Sk_\II(S^3) = \SkFun_\II(S^3,\emptyset,\emptyset)$. The natural transformation $Z_4$ is just the data of a linear map, which we will still denote $Z_4$ by abuse,
$$Z_4: \Sk_\II(S^3)\to \Bbbk$$
We take this map to be the invariant of $\II$-colored admissible ribbon graphs induced by the modified trace \cite{GeerPatureauTuraevModifiedqdim, GeerPatureauTraceProj} as in \cite{CGHP}. If $T \in \Sk_\II(S^3)$ is the closure of a 1-1-tangle $T_{cut}$ whose endpoints are both colored by a projective $P$, we set 
$$Z_4(T) := \mt_P(\operatorname{RT}(T_{cut}))$$
where $\operatorname{RT}$ denote the usual Reshetikhin--Turaev evaluation functor \cite{TuraevBook}. See Figure \ref{fig:4handle}. This is well-defined by the work of Geer--Patureau-Mirand et al. and there is no naturality to check in this case. 
\begin{figure}
    \centering
\begin{tikzpicture}[baseline = 0pt]
\draw[gray] (-1,0.2) arc(180:540:1.5);
\fill[gray!10] (-1,0.2) arc(180:540:1.5);
\node[color = gray] at (1.2,1.2){$S^3$};
\node[rectangle, draw, fill=white] (T) at (0,0) {$T_{cut}$};
\draw (T) ..controls (0,1) and (1,1).. (1,0) .. controls (1,-1) and (0,-1)..(T) node[pos = 0.1, right]{$P$} node [pos = 0.1, sloped]{$<$};
\end{tikzpicture} $\quad \mapsto \quad \mt_P(T_{cut})$
\caption{The 4-handle}
    \label{fig:4handle}
\end{figure}
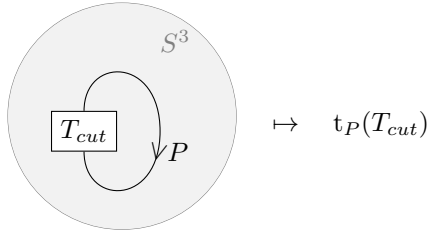

\paragraph{The 3-handle}
We define a natural transformation $$Z_3: \SkFun_\II(S^2\times \D^1) \Rightarrow \SkFun_\II(\D^3 \times S^0)\ .$$
Both 3-cobordisms $S^2\times \D^1$ and $\D^3 \times S^0$ have incoming boundary $S^2 \times S^0$, and outgoing boundary $\emptyset$. Let $X \in \SkCat_\II(S^2\times S^0)$ be an admissible $\II$-labelling, which we may write $X = X_+\sqcup X_-$ as $S^2 \times S^0 = S^2\times \{+\}\sqcup S^2\times \{-\}$, and $T \in \Sk_\II(S^2\times\D^1;X_+\sqcup X_-)$ an $\II$-colored ribbon graph.
We want to ``cut" $T$ in two pieces. 

An object $P \in \II$ induces an $\II$-labeling in $S^2$ with a single point colored by $P$, which by abuse we will still call $P \in \SkCat_\II(S^2)$. As $S^2$ is connected, any object of $\SkCat_\II(S^2)$ is actually isomorphic to an object of this form. If $X$ is any configuration of points, then $P$ is the tensor product of all its colors, with duals for negatively oriented points. Similarly, a morphism $f \in \Hom_\AA(P,P')$ induces a morphism in $\SkCat_\II(S^2)$, and morphisms of this form generate all morphisms.

As in the proof of Theorem \ref{thm:SkCategTQFT}, we have an equivalence 
$$\Sk_\II(S^2\times \D^1, X_+\sqcup X_-) \simeq \int^{P \in \SkCat_\II(S^2)} \Sk_\II(S^2 \times [0,1]; X_+,P) \otimes \Sk_\II(S^2 \times [-1,0]; P\sqcup X_-)$$
Let us denote $T_+\otimes T_-$ two skeins in the RHS that glue to $T$. 

Given a morphism $f:P \to \unit$, we get a skein which by abuse we will still denote $f \in \Sk_\II(\D^3, P)$ which has a single vertex colored by $f$ at $0 \in \D^3$ linked by a straight line to $P \in \SkCat_\II(S^2)$. Similarly, for $\Omega \in \Hom_\AA(P,\unit)\otimes \Hom_\AA(\unit, P)$ we get a skein $\Omega \in \Sk_\II(\D^3\times S^0)$.

We set 
$$Z_3(T) := (T_+\sqcup T_-)\cdot \Omega_P$$
where $(T_+\sqcup T_-)\cdot -$ is the action of morphisms in $\SkCat_\II(S^2 \times S^0)$ on $\Sk_\II(D^3\times S^0)$. See Figure \ref{fig:3handle}. It is well-defined, i.e. preserves the coend relation in the coend above, by naturality of $\Omega_P$ \cite[Lem. 1.1]{CGHP}. It is natural as for any morphism $S$ in $\SkCat_\II(S^2\times S^0)$ we have $Z_3(S\cdot T) = S \cdot (T_+\sqcup T_-)\cdot \Omega_P$.

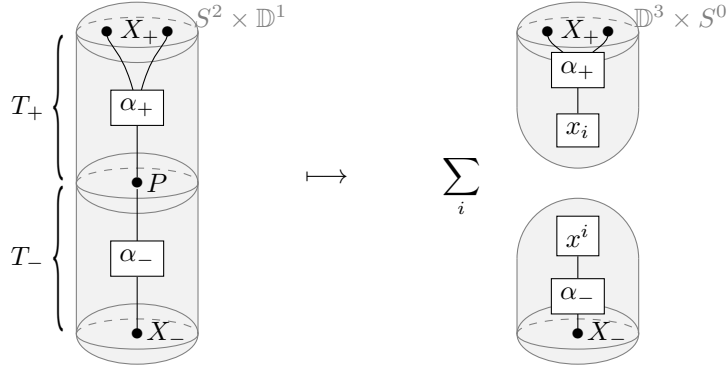
\begin{figure}
    \centering
\begin{tikzpicture}[baseline = 0pt,xscale = 0.8]
\fill[gray!10] (0,-2) rectangle (2,2);
    \fill[gray!10] (0,2) arc(180:540:1 and 0.4);
    \draw[gray] (0,2) arc(180:540:1 and 0.4);
    \draw[gray] (0,2) arc(180:360:1 and 0.2);
    \draw[dashed, gray] (0,2) arc(180:0:1 and 0.2);
    
    \fill[gray!10] (0,-2) arc(180:540:1 and 0.4);
    \draw[gray] (0,-2) arc(180:540:1 and 0.4);
    \draw[gray] (0,-2) arc(180:360:1 and 0.2);
    \draw[dashed, gray] (0,-2) arc(180:0:1 and 0.2);
    
    \draw[gray] (0,0) arc(180:540:1 and 0.4);
    \draw[gray] (0,0) arc(180:360:1 and 0.2);
    \draw[dashed, gray] (0,0) arc(180:0:1 and 0.2);

    \draw[gray] (0,-2)--++(0,4);
    \draw[gray] (2,-2)--++(0,4);
    \node[gray] at (2.7,2.2) {$S^2\times \D^1$};

\node[inner sep = 0pt, outer sep = 0pt] (P) at (1,0) {$\bullet$};
\draw (0.5,2) node{$\bullet$} node[right = 0.05cm]{$X_+$} ..controls (0.5, 1.8) and (0.8,1.8).. (1, 1) -- (P) node[right]{$P$};
\draw (1.5,2) node{$\bullet$} ..controls (1.5, 1.8) and (1.2,1.8).. (1, 1);
\node[rectangle, draw, fill=white] at (1,1){$\alpha_+$};

\draw (1,-2) node{$\bullet$} node[right = 0cm]{$X_-$} -- (1, -1) -- (P);
\node[rectangle, draw, fill=white] at (1,-1){$\alpha_-$};

\node[yscale = 5.5] at (-0.3,1) {$\{$};
\node[yscale = 5.5] at (-0.3,-1) {$\{$};
\node at (-0.8,1) {$T_+$};
\node at (-0.8,-1) {$T_-$};
\end{tikzpicture}
$\longmapsto$ \hspace{1cm} $\displaystyle \sum_i \quad$
\begin{tikzpicture}[baseline = 0pt,xscale = 0.8]
\fill[gray!10] (0,1) arc(180:360:1 and 0.8) -- (2,2) -- (0,2) -- (0,1);
\fill[gray!10] (0,-1) arc(180:0:1 and 0.8) -- (2,-2) -- (0,-2) -- (0,-1);
    \fill[gray!10] (0,2) arc(180:540:1 and 0.4);
    \draw[gray] (0,2) arc(180:540:1 and 0.4);
    \draw[gray] (0,2) arc(180:360:1 and 0.2);
    \draw[dashed, gray] (0,2) arc(180:0:1 and 0.2);
    
    \fill[gray!10] (0,-2) arc(180:540:1 and 0.4);
    \draw[gray] (0,-2) arc(180:540:1 and 0.4);
    \draw[gray] (0,-2) arc(180:360:1 and 0.2);
    \draw[dashed, gray] (0,-2) arc(180:0:1 and 0.2);
    
    \draw[gray] (0,1) arc(180:360:1 and 0.8);
    \draw[gray] (0,-1) arc(180:0:1 and 0.8);

    \draw[gray] (0,1)--++(0,1);
    \draw[gray] (0,-2)--++(0,1);
    \draw[gray] (2,1)--++(0,1);
    \draw[gray] (2,-2)--++(0,1);
    \node[gray] at (2.7,2.2) {$\D^3\times S^0$};

\node[rectangle, draw, fill=white] (Up) at (1,0.7) {$x_i$};
\node[rectangle, draw, fill=white] (Do) at (1,-0.7) {$x^i$};
\draw (0.5,2) node{$\bullet$} node[right = 0.05cm]{$X_+$} ..controls (0.5, 1.8) and (0.8,1.8).. (1, 1.5) -- (Up);
\draw (1.5,2) node{$\bullet$} ..controls (1.5, 1.8) and (1.2,1.8).. (1, 1.5);
\node[rectangle, draw, fill=white] at (1,1.5){$\alpha_+$};

\draw (1,-2) node{$\bullet$} node[right = 0cm]{$X_-$} -- (1, -1.5) -- (Do);
\node[rectangle, draw, fill=white] at (1,-1.5){$\alpha_-$};
\end{tikzpicture}
\caption{The 3-handle. Here $\Omega_P = \sum_i x^i \otimes x_i$ is the copairing.}
    \label{fig:3handle}
\end{figure}


\paragraph{The 2-handle}
We define a natural transformation $$Z_2: \SkFun_\II(S^1\times \D^2) \Rightarrow \SkFun_\II(\D^2 \times S^1)\ .$$
Both 3-cobordisms $S^1\times \D^2$ and $\D^2 \times S^1$ have incoming boundary $S^1 \times S^1$, and outgoing boundary $\emptyset$. Let $X \in \SkCat_\II(S^1\times S^1)$ be an admissible $\II$-labelling and $T \in \Sk_\II(S^1\times\D^2;X)$.

We may isotope $T$ so that it does not intersect the core $S^1\times\{0\}$, so we can think of the pushed $T$ as a skein $T_{pushed} \in \Sk_\II(S^1\times S^1\times I; X, \emptyset)$. Note that this skein is not canonically defined and depends on how we pushed $T$. We will incautiously write $T=T_{pushed}\cdot \emptyset$, even though $\emptyset$ is not an admissible skein in $S^1\times \D^2$, and in particular $Z_2(\emptyset)$ need not be defined.


We set
$$Z_2(T) := T_{pushed} \cdot (\{0\}\times S^1)_{red}$$
where the cocore $\{0\}\times S^1$ is colored in red. The color "red" is not an object of our category, and this does not a priori define a skein. It is interpreted as a skein using the red-to-blue operation of \cite{CGHP}, i.e. using a chromatic morphism:
\begin{equation}\label{eq:redToBlue}
    \begin{tikzpicture}[baseline = 5pt]
        \draw[red] (0,0) --++(0,1);
        \draw[red, dashed] (0,1)..controls (0,2) and (-1,2).. (-1,1) node[pos = 0.4, sloped, above=-2pt]{\footnotesize red}--(-1,0)..controls (-1,-1) and (0,-1)..(0,0);
        \draw (0.5,-1)--++(0,3) node[pos = 0.3, right]{$P$} node[pos = 0.3, sloped]{$>$};
    \end{tikzpicture} \ := \quad
        \begin{tikzpicture}[baseline = 5pt]
        \draw (0,0) --++(0,1);
        \draw[dashed] (0,1)..controls (0,2) and (-1,2).. (-1,1) node[pos = 0.3, right]{$G$}node[pos = 0.3, sloped]{$<$}--(-1,0)..controls (-1,-1) and (0,-1)..(0,0);
        \draw (0.5,-1)--++(0,3) node[pos = 0.3, right]{$P$} node[pos = 0.3, sloped]{$>$};
        \node[rectangle, draw, fill=white] at (0.25,0.5){$\ \chr_P\ $};
    \end{tikzpicture}
\end{equation}
Note that this is only defined in the presence of a projective object. In other words, $Z_2(T) = T_{pushed} \cdot (\{0\}\times S^1)_{red}$ can be turned into a skein, but $(\{0\}\times S^1)_{red}$ alone cannot. 

This operation does not depend on how we isotoped $T$ to be disjoint from the core by \cite[Lem. 2.4]{CGHP} nor on how we applied the red-to-blue operation or on the choice of a chromatic morphism by \cite[Lem. 2.3]{CGHP}. Both of these operations can be chosen to leave $T$ unchanged near the boundary, which shows naturality of $Z_2$.

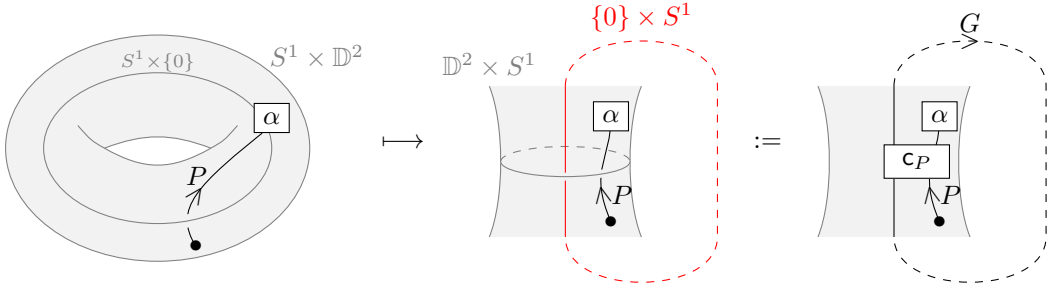
\begin{figure}
    \centering
\begin{tikzpicture}[baseline = 0pt]
    \fill[gray!10] (0,0) arc(180:540:2 and 1.5);
    \draw (2.5,-1.3) node{$\bullet$}.. controls (2.2,-0.8) and (2.5, -0.5).. (3.5,0.3) node[pos = 0.5, sloped]{$>$}node[pos = 0.5, above]{$P$};
    \draw[gray] (0,0) arc(180:540:2 and 1.5) node[pos = 0.65, right=5pt]{$S^1\times \D^2$};
    \draw[line width = 5pt, gray!10] (0.5,0) arc(180:540:1.5 and 1);
    \draw[gray] (0.5,0) arc(180:540:1.5 and 1) node[pos = 0.75, above=-3pt]{$\scriptstyle S^1\times \{0\}$};
    \node[rectangle, draw, fill=white] at (3.5,0.4){$\alpha$};
    \begin{scope}[xshift = 0.6cm,xscale = 0.7]
    \fill[white] (1,0).. controls (1.75, 0.2) and (2.25, 0.2).. (3,0).. controls (2.25, -0.3) and (1.75, -0.3)..(1,0);
    \draw[gray] (0.5, 0.3) .. controls (0.75,0.1) ..(1,0).. controls (1.75, -0.3) and (2.25, -0.3).. (3,0).. controls (3.25,0.1).. (3.5,0.3);
    \draw[gray] (1,0).. controls (1.75, 0.2) and (2.25, 0.2).. (3,0);
    \end{scope}
\end{tikzpicture} $\longmapsto$
\begin{tikzpicture}[baseline = 5pt]
\fill[gray!10] (-1,-1)..controls (-0.8,-0.5) and (-0.8,0.5).. (-1,1) -- (1,1) ..controls (0.8,0.5) and (0.8,-0.5).. (1,-1) -- (-1,-1);
    \draw[gray] (-1,-1)..controls (-0.8,-0.5) and (-0.8,0.5).. (-1,1) node[above]{$\D^2\times S^1$};
    \draw[gray] (1,-1)..controls (0.8,-0.5) and (0.8,0.5).. (1,1);

    \draw[dashed, gray] (-0.85,0) arc(180:0:0.85 and 0.2);
    
    \draw[red] (0,-1)--(0,1);
    \draw[red, dashed] (0,1).. controls (0,1.8) and (2,1.8).. (2,1) node[pos = 0.5,above]{$\{0\}\times S^1$} --(2,-1)..controls (2,-1.8) and (0,-1.8)..(0,-1);

    \draw (0.6,-0.8) node{$\bullet$}.. controls (0.3,-0.1) and (0.6, 0).. (0.6,0.6) node[pos = 0.2, sloped]{$<$}node[pos = 0.2, right]{$P$};
    \node[rectangle, draw, fill=white] at (0.6,0.6){$\alpha$};

    \draw[line width = 3pt, gray!10] (-0.8,0) arc(180:360:0.8 and 0.2);
    \draw[gray] (-0.85,0) arc(180:360:0.85 and 0.2);
\end{tikzpicture}
\quad := \quad
\begin{tikzpicture}[baseline = 5pt]
\fill[gray!10] (-1,-1)..controls (-0.8,-0.5) and (-0.8,0.5).. (-1,1) -- (1,1) ..controls (0.8,0.5) and (0.8,-0.5).. (1,-1) -- (-1,-1);
    \draw[gray] (-1,-1)..controls (-0.8,-0.5) and (-0.8,0.5).. (-1,1);
    \draw[gray] (1,-1)..controls (0.8,-0.5) and (0.8,0.5).. (1,1);
    
    \draw (0,-1)--(0,1);
    \draw[dashed] (0,1).. controls (0,1.8) and (2,1.8).. (2,1) node[pos = 0.5,above]{$G$}node[pos = 0.5]{$>$} --(2,-1)..controls (2,-1.8) and (0,-1.8)..(0,-1);

    \draw (0.6,-0.8) node{$\bullet$}.. controls (0.3,-0.1) and (0.6, 0).. (0.6,0.6) node[pos = 0.2, sloped]{$<$}node[pos = 0.2, right]{$P$};
    \node[rectangle, draw, fill=white] at (0.6,0.6){$\alpha$};
    \node[rectangle, draw, fill=white] at (0.3,0){$\ \chr_P\ $};
\end{tikzpicture}
\begin{tikzpicture}
    
\end{tikzpicture}
\caption{The 2-handle. For depiction purposes we have represented $S^1\times \D^2$ embedded in $S^3$, and $\D^2 \times S^1$ as its complement. Note however that the orientations don't quite match, as they both should be on the same side of $S^1 \times S^1$.}
    \label{fig:2handle}
\end{figure}

\paragraph{The 1-handle}
We define a natural transformation $$Z_1: \SkFun_\II(S^0\times \D^3) \Rightarrow \SkFun_\II(\D^1 \times S^2)\ .$$
Both 3-cobordisms $S^0\times \D^3$ and $\D^1 \times S^2$ have incoming boundary $S^0 \times S^2$, and outgoing boundary $\emptyset$. Let $X= X_+\sqcup X_- \in \SkCat_\II(S^0\times S^2)$ be an admissible $\II$-labelling and $T = T_+\sqcup T_- \in \Sk_\II(S^0\times\D^3;X_+\sqcup X_-)$.

As we have seen with the 3-handle, the copairing $\Omega_{P_\unit} = \varepsilon_\unit \otimes \eta_\unit$ can be seen as a skein in $\Sk_\II(S^0\times \D^3; P_\unit\sqcup -P_\unit)$. We may isotope $T$ to be disjoint from $S^0 \times \frac{1}{2}\D^3$ and introduce two $P_\unit$-strands in $T_+$ and $T_-$ ending by a vertex $v_\pm$ colored by $\varepsilon_\unit$ or $\eta_\unit = \varepsilon_\unit^*$ which we may isotope to be at $0\in\frac{1}{2}\D^3$. In other words, we can write $T = \widetilde T \cdot \Omega_{P_\unit}$. 

We may think of a gluing morphism $\gm: P_\unit\to P_\unit$ as a skein $\gm \in \Sk_\II(\D^1 \times S^2; P_\unit \sqcup -P_\unit)$.

We set
$$Z_1(T) = \widetilde T \cdot \gm$$
See Figure \ref{fig:1handle}.
This is well-defined and does not depend on the choice of a gluing morphism or on how we introduced $P_\unit$-strands by \cite[Proposition 5.1]{CGHP}. Again naturality is clear as we may leave $T$ unchanged near the boundary.

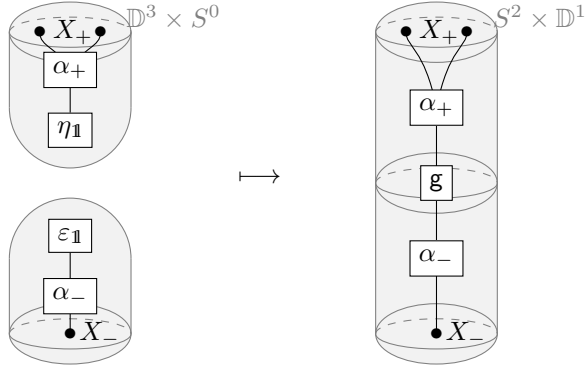
\begin{figure}
    \centering
\begin{tikzpicture}[baseline = 0pt,xscale = 0.8]
\fill[gray!10] (0,1) arc(180:360:1 and 0.8) -- (2,2) -- (0,2) -- (0,1);
\fill[gray!10] (0,-1) arc(180:0:1 and 0.8) -- (2,-2) -- (0,-2) -- (0,-1);
    \fill[gray!10] (0,2) arc(180:540:1 and 0.4);
    \draw[gray] (0,2) arc(180:540:1 and 0.4);
    \draw[gray] (0,2) arc(180:360:1 and 0.2);
    \draw[dashed, gray] (0,2) arc(180:0:1 and 0.2);
    
    \fill[gray!10] (0,-2) arc(180:540:1 and 0.4);
    \draw[gray] (0,-2) arc(180:540:1 and 0.4);
    \draw[gray] (0,-2) arc(180:360:1 and 0.2);
    \draw[dashed, gray] (0,-2) arc(180:0:1 and 0.2);
    
    \draw[gray] (0,1) arc(180:360:1 and 0.8);
    \draw[gray] (0,-1) arc(180:0:1 and 0.8);

    \draw[gray] (0,1)--++(0,1);
    \draw[gray] (0,-2)--++(0,1);
    \draw[gray] (2,1)--++(0,1);
    \draw[gray] (2,-2)--++(0,1);
    \node[gray] at (2.7,2.2) {$\D^3\times S^0$};

\node[rectangle, draw, fill=white] (Up) at (1,0.7) {$\eta_\unit$};
\node[rectangle, draw, fill=white] (Do) at (1,-0.7) {$\varepsilon_\unit$};
\draw (0.5,2) node{$\bullet$} node[right = 0.05cm]{$X_+$} ..controls (0.5, 1.8) and (0.8,1.8).. (1, 1.5) -- (Up);
\draw (1.5,2) node{$\bullet$} ..controls (1.5, 1.8) and (1.2,1.8).. (1, 1.5);
\node[rectangle, draw, fill=white] at (1,1.5){$\alpha_+$};

\draw (1,-2) node{$\bullet$} node[right = 0cm]{$X_-$} -- (1, -1.5) -- (Do);
\node[rectangle, draw, fill=white] at (1,-1.5){$\alpha_-$};
\end{tikzpicture}
$\longmapsto$ \hspace{1cm}
\begin{tikzpicture}[baseline = 0pt,xscale = 0.8]
\fill[gray!10] (0,-2) rectangle (2,2);
    \fill[gray!10] (0,2) arc(180:540:1 and 0.4);
    \draw[gray] (0,2) arc(180:540:1 and 0.4);
    \draw[gray] (0,2) arc(180:360:1 and 0.2);
    \draw[dashed, gray] (0,2) arc(180:0:1 and 0.2);
    
    \fill[gray!10] (0,-2) arc(180:540:1 and 0.4);
    \draw[gray] (0,-2) arc(180:540:1 and 0.4);
    \draw[gray] (0,-2) arc(180:360:1 and 0.2);
    \draw[dashed, gray] (0,-2) arc(180:0:1 and 0.2);
    
    \draw[gray] (0,0) arc(180:540:1 and 0.4);
    \draw[gray] (0,0) arc(180:360:1 and 0.2);
    \draw[dashed, gray] (0,0) arc(180:0:1 and 0.2);

    \draw[gray] (0,-2)--++(0,4);
    \draw[gray] (2,-2)--++(0,4);
    \node[gray] at (2.7,2.2) {$S^2\times \D^1$};

\node[inner sep = 0pt, outer sep = 0pt] (P) at (1,0) {};
\draw (0.5,2) node{$\bullet$} node[right = 0.05cm]{$X_+$} ..controls (0.5, 1.8) and (0.8,1.8).. (1, 1) -- (P);
\draw (1.5,2) node{$\bullet$} ..controls (1.5, 1.8) and (1.2,1.8).. (1, 1);
\node[rectangle, draw, fill=white] at (1,1){$\alpha_+$};

\draw (1,-2) node{$\bullet$} node[right = 0cm]{$X_-$} -- (1, -1) -- (P)  node[rectangle, draw, fill=white]{$\gm$};
\node[rectangle, draw, fill=white] at (1,-1){$\alpha_-$};
\end{tikzpicture}
\caption{The 1-handle.}
    \label{fig:1handle}
\end{figure}

\paragraph{The 0-handle} In the case where $\AA$ is chromatic compact, i.e. there exists $\zeta \in \Bbbk^\times$ the global dimension satisfying $\varepsilon_\unit \circ \gm =\zeta^{-1} \varepsilon_\unit$, we set 
$$Z_0: \SkFun_\II(\emptyset) \Rightarrow \SkFun_\II(S^3)$$
by mapping $1 \in \Bbbk$ to $\zeta \Gamma_0 \in \Sk_\II(S^3)$ where $\Gamma_0$ is the admissible skein in $S^3$ with a single strand colored by $P_\unit$ and two coupons $\varepsilon_\unit$ and $\eta_\unit$. See Figure \ref{fig:0handle}.

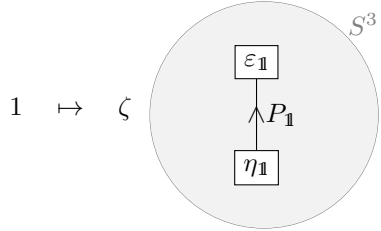
\begin{figure}
    \centering
    $1 \quad \mapsto \quad \zeta\ $
\begin{tikzpicture}[baseline = 0pt]
\draw[gray] (-1.4,0) arc(180:540:1.5);
\fill[gray!10] (-1.4,0) arc(180:540:1.5);
\node[color = gray] at (1.4,1.2){$S^3$};
\draw (0,-0.7)node[rectangle, draw, fill=white]{$\eta_\unit$} --(0,0.7)node[pos = 0.5, right]{$P_\unit$} node [pos = 0.5, sloped]{$>$} node[rectangle, draw, fill=white]{$\varepsilon_\unit$};
\end{tikzpicture} 
\caption{The 0-handle}
    \label{fig:0handle}
\end{figure}

\begin{theorem}\label{thm:ExtendedCGHP}
    Let $\AA$ be a chromatic non-degenerate category, $\II = \operatorname{Proj}(\AA)$ and $\mt$ a modified trace. Then there exists a unique non-compact (2+1+1)-TQFT $$\SS_\AA: \Cob_{2+1+1}^{nc}\to \Bimod^{hop}$$
    extending $\SkFun_\II$ and given by $Z_4,\dots, Z_1$ on the standard handle attachments of index $4,\dots, 1$.

    If $\AA$ is moreover chromatic compact, then $\SS_\AA$ extends in a unique way to a fully defined $(2+1+1)$-TQFT $$\SS_\AA: \Cob_{2+1+1}\to \Bimod^{hop}$$ given by $Z_0$ on the 0-handle. 
\end{theorem}
\begin{proof}
    We apply \cite[Thm. 3.8]{HaiounHandle}. The fact that our assignment on handles satisfies handle cancellations and invariance under reversal of the attaching spheres is checked in \cite[Thm 5.4]{CGHP}. The cancellation of handles actually follows from the definitions, of the copairing for the 3-4-cancellation, of the chromatic morphism for the 2-3-cancellation, of the gluing morphism for the 1-2-cancellation and of the global dimension for the 0-1-cancellation. Invariance under reversal of the attaching sphere is checked in \cite[Lem. 1.1 (1)]{CGHP} for the 3-handle, in \cite[Lem. 2.3]{CGHP} for the 2-handle and in \cite[Prop. 5.1]{CGHP} for the 1-handle.
\end{proof}

\subsection{Properties}
For the applications in Section \ref{Sec_WRT} we will be interested in the special case when $\AA$ is modular in the sense of \cite{Lyub3mfldProjMCG, DGGPR}. In this case, the TQFT constructed above is very simple in dimension 4, and almost all of its data is located "at the boundary". 

\begin{proposition}\label{Prop_ZdepOnSignature}
    Suppose $\AA$ is a possibly non-semisimple modular tensor category. Then:
    
    $\AA$ is chromatic compact and the TQFT $\SS_\AA$ constructed above is invertible for any choice of non-degenerate modified trace $\mt$. 

    The red-to-blue operation used in our construction agrees with the red-to-blue operation of \cite{DGGPR}.

    There are exactly two choices of modified trace for which $\SS_\AA(S^4)=1$. If we use one of these, the natural transformations $\SS_\AA(W): \Sk_\II(M_-)\to \Sk_\II(M_+)$, for fixed $M_-, M_+$, depend only on the signature of $W$. Moreover, if $M$ is a closed 3-manifold, $T \subseteq M$ is admissible and $W:M\to \emptyset$ is a bounding 4-manifold, we have $$\SS_\AA(W)(T) = \DGGPR_\AA(M, T, \sigma(W))\ .$$ 
\end{proposition}
\begin{proof}
    The fact that modular implies chromatic compact is immediate from \cite[Def. 1.7]{CGHP}. The invertibility statement is \cite[Thm. 5.8]{CGHP} for the (3+1)-part, and follows from \cite{SPtorusInvertibility} for the whole theory. The fact that the red-to-blue operations coincide is explained in \cite[Thm. 1.10, Eq. (8)]{CGHP}.
    
    Choose a modified trace $\mt$. By \cite[Prop 5.7]{CGHP}, the (3+1)-TQFT obtained from $\kappa \mt$ differs from the one obtained from $\mt$ by an Euler characteristic term. As $\chi(S^4)=2$, there are exactly two choices for $\kappa$ such that $\SS_{\AA, \kappa \mt}(S^4)=1$, namely $\kappa = \pm\DD^{-1}$ for $\DD$ a square root of the global dimension $\zeta = \SS_{\AA,\mt}(S^4)$. 

    We now assume we have chosen one of the two modified traces above, so $\zeta = \Delta_+\Delta_- = 1$.
    As every cobordism act by isomorphisms, we can pre-compose and side compose $W$ by bounding manifolds without loosing information and it is equivalent to check the dependence on $W$ for $W$ closed. As the signature is additive when gluing on closed boundary components, the signature of the closed up manifold is the sum of signature of the initial manifold and the manifold we closed it up with.

As $\gm = \id$, the action of a disjoint union of two 4-handles is the same as that of first using a 3-handle to connect the two balls and then using only one 4-handle. This allows us to further reduce to the case where $W$ is connected.
    
    As observed in \cite[Def. 1.7]{CGHP}, or \cite[Lem. 4.4]{DGGPR}, as $\AA$ is modular we have
    $$\Delta^0_P = \Lambda_P,\quad P \in \II$$ 
    where the endomorphism $\Delta^0_P$ is obtained by encircling a $P$ strand by a red circle. Reading the handles backwards, and using Akbulut's dotted circle convention, this is saying that one may replace a dotted circle (the RHS) by a plain circle (the LHS). As $W$ is connected we may assume it has only one 0-handle, and applying this observation to every 1-handle we may further reduce to simply connected $W$.

The scalar $\SS_\II(W)$ is multiplicative under connected sum as $\SS_\II(S^4)=1$. As $\SS_\II(\mathbb C P_2)\SS_\II(\overline{\mathbb C P}_2) \neq 0$, it is stable under $\mathbb C P_2$-stabilization. Finally, two simply connected closed 4-manifolds are $\mathbb C P_2$-stably diffeomorphic if and only if they have same signature and Euler characteristic. Taking connected sum with $\mathbb C P_2\# \overline{\mathbb C P}_2$ adds 2 to the Euler characteristic, and two 4-manifolds with same signature have Euler characteristic differing by an even number. As $\SS_\II(\mathbb C P_2\# \overline{\mathbb C P}_2)= \Delta_+\Delta_- =1$, the invariant $\SS_\II(W)$ does not depend on the Euler characteristic, and depends only on the signature of $W$ as claimed. 

Finally, if $M$ is closed and $T \subseteq M$ is admissible by \cite[Thm. 4.4 and 5.9]{CGHP} (if $W$ is obtained by 2- and 4-handles there, and for any $W$ by the arguments above) we have $$\SS_\II(W)(T) = \SS_\II(\mathbb CP_2)^{\sigma(W)}\operatorname{L}'(M,T) = \DGGPR_\AA(M, T, \sigma(W))$$
where $\operatorname{L}'$ is the renormalized Lyubashenko invariant introduced in \cite{DGGPR} and $\DGGPR_\AA$ is the TQFT introduced there, which gives invariants of 3-manifolds equipped with an admissible ribbon graph and a ``signature defect" integer.
    \end{proof}

\section{The regular boundary condition to skein theory}\label{Sec_BoundaryCond}
In this section we will define a possibly non-compact boundary condition 
$$\RR: \Triv \Rightarrow \SkFun_\II$$
given by the empty skein.

This boundary condition will be non-compact when $\II \neq \AA$, which is quite different from the non-compact cases of Section \ref{Sec_ExtendedCY}. In the case of interest for Section \ref{Sec_WRT}, we will assume that $\AA$ is modular and $\II=\operatorname{Proj}(\AA)$, in which case $\SS_\II$ is defined on all cobordisms, but $\RR$ will still be non-compact when $\AA$ is non-semisimple.

\subsection{Definition of a boundary condition}
Recall that $\Sk_\II$ is contravariant in the direction of 1-morphism. This is only an annoyance as one can always take opposite orientation on cobordisms to make it covariant, and when we say a boundary condition to $\Sk_\II$ we mean a boundary condition to this covariant functor. However, to avoid confusion or having to introduce a different notation, let us recall what explicit data a boundary condition to a contravariant functor represents.
\begin{definition}\label{def:BondaryCond}
    A boundary condition         \begin{equation*}
            \begin{tikzcd}
            \Cob_{2+1+\varepsilon}^{hop} \ar[dd, bend right=50pt,"\mathlarger\Triv"' pos = 0.45, ""{name=left, right}]\ar[dd, bend left=50pt,"\mathlarger{\ZZ^\varepsilon}"pos = 0.45, ""{name=right, left}] \\ \ \\\CC          
            \ar[from=left,to=right,Rightarrow,"\mathlarger{\RR}"]
        \end{tikzcd}
        \end{equation*} to a categorified TQFT $\ZZ^\varepsilon: \Cob_{2+1+\varepsilon}^{hop}\to \CC$ contravariant in the direction of 1-morphism is the data of
    \begin{description}
        \item[For every object $\Sigma$:] a 1-morphism $\RR(\Sigma): \unit_\CC \to \ZZ^\varepsilon(\Sigma)$
        \item[For every cobordism $M:\Sigma'\to \Sigma$:] a 2-morphism \begin{equation*}
            \begin{tikzcd}[column sep = 50pt]
                \unit_\CC \ar[r, "\RR(\Sigma)"] \ar[d, equal] & \ZZ^\varepsilon(\Sigma) \ar[d, "\ZZ^\varepsilon(M)"]\\
                \unit_\CC \ar[r, "\RR (\Sigma')"] & \ZZ^\varepsilon(\Sigma')
                \arrow[from = 2-1, to = 1-2, "\RR(M)", Rightarrow,sloped, shorten = 10]
            \end{tikzcd}
        \end{equation*}
        \item[Symmetric monoidal structure:] 2-isomorphisms $\RR(\Sigma)\otimes \RR(\Sigma') \Rightarrow \RR(\Sigma\sqcup\Sigma')$ and $\RR(\emptyset)\Rightarrow\unit_\CC$ with appropriate symmetric monoidal structure of $\ZZ^\varepsilon$ inserted to make the source and target of the 1-morphisms match. 
    \end{description}
    Such that 
    \begin{description}
        \item[For every diffeomorphism $f:M_-\to M_+$:] the following equality holds: \begin{equation*}
            \begin{tikzcd}[column sep = 50pt]
                \unit_\CC \ar[r, "\RR \Sigma"] \ar[d, equal] & \ZZ^\varepsilon(\Sigma) \ar[d, "\ZZ^\varepsilon(M_-)"{name = Mm, right}] \ar[r, equal] & \ZZ^\varepsilon(\Sigma)\ar[d, "\ZZ^\varepsilon(M_+)"{name = Mp, right}]\\
                \unit_\CC \ar[r, "\RR \Sigma'"] & \ZZ^\varepsilon(\Sigma') \ar[r, equal] & \ZZ^\varepsilon(\Sigma')
                \arrow[from = 2-1, to = 1-2, "\RR(M_-)", Rightarrow,sloped, shorten = 10]
                \arrow[from = Mm, to = Mp, "\ZZ^\varepsilon(f)", Rightarrow, shorten = 10]
            \end{tikzcd}     
            =
            \begin{tikzcd}[column sep = 50pt]
                \unit_\CC \ar[r, "\RR \Sigma"] \ar[d, equal] & \ZZ^\varepsilon(\Sigma) \ar[d, "\ZZ^\varepsilon(M_+)"]\\
                \unit_\CC \ar[r, "\RR \Sigma'"] & \ZZ^\varepsilon(\Sigma')
                \arrow[from = 2-1, to = 1-2, "\RR(M_+)", Rightarrow,sloped, shorten = 10]
            \end{tikzcd}
        \end{equation*}
        \item[For every composeable 1-morphisms $\Sigma_3\overset{M_{32}}{\to} \Sigma_2 \overset{M_{21}}{\to}\Sigma_1$:] the following equality holds:\begin{equation*}
            \begin{tikzcd}[column sep = 50pt]
                \unit_\CC \ar[r, "\RR \Sigma_1"] \ar[dd, equal] & \ZZ^\varepsilon(\Sigma_1) \ar[dd, "\ZZ^\varepsilon(M_{21}\circ M_{32})"]\\
                & \\
                \unit_\CC \ar[r, "\RR \Sigma_3"] & \ZZ^\varepsilon(\Sigma_3)
                \arrow[from = 3-1, to = 1-2, "\RR(M_{21}\circ M_{32})", Rightarrow, sloped, shorten = 10]
            \end{tikzcd}
            =
            \begin{tikzcd}[column sep = 50pt]
                \unit_\CC \ar[r, "\RR \Sigma_1"] \ar[d, equal] & \ZZ^\varepsilon(\Sigma_1) \ar[d, "\ZZ^\varepsilon(M_{21})"]\\
                \unit_\CC \ar[r, "\RR \Sigma_2"] \ar[d, equal] & \ZZ^\varepsilon(\Sigma_2) \ar[d, "\ZZ^\varepsilon(M_{32})"]\\
                \unit_\CC \ar[r, "\RR \Sigma_3"] & \ZZ^\varepsilon(\Sigma_3)
                \arrow[from = 2-1, to = 1-2, "\RR(M_{21})",sloped, Rightarrow, shorten = 10]
                \arrow[from = 3-1, to = 2-2, "\RR(M_{32})",sloped, Rightarrow, shorten = 10]
            \end{tikzcd}
        \end{equation*}
        \item[Coherence of the symmetric monoidal structure:] see \cite[Def. 2.7]{SPPhD}. 
    \end{description}
    A non-compact boundary condition is one that only gives values to 3-cobordisms with non-empty incoming boundary $\Sigma'$ in every connected component.
\end{definition}

\subsection{The empty skein as a boundary condition}
Let $\II \subseteq\AA$ be a tensor ideal in a ribbon category.
\begin{definition}
    Let $\Sigma$ be a surface. We define the 1-morphism $\RR_\II(\Sigma)$ from $\unit$ to $\SkCat_\II(\Sigma)$ in $\Bimod$ by
    $$\begin{array}{rcl}
        \RR_\II(\Sigma): \SkCat_\II(\Sigma)^{op}&\to&\Vect\\
        X &\mapsto&  \Sk_\II(\Sigma\times [0,1]; X,\emptyset)
    \end{array}$$
    It is called the distinguished object in \cite{BHskcat, BBJ}. See Figure \ref{fig:DistObj}.

    If $\II=\AA$, then this presheaf is actually representable, and represented by any number of points all colored by the monoidal unit, which we will call the empty collection of points. 
\end{definition}
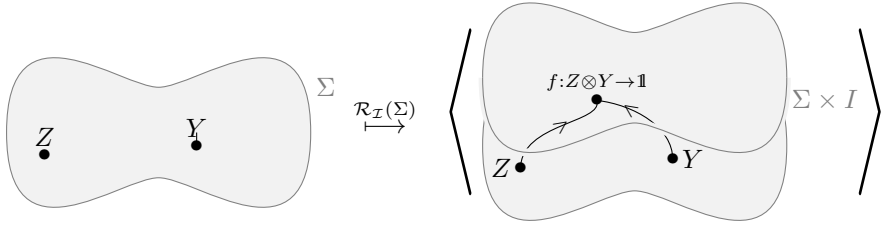
\begin{figure}
    \centering
\begin{tikzpicture}[yscale = 0.6, baseline = 0pt]
    \fill[gray!10] (0,0) ..controls (0,3) and (1.5,1) .. (2,1) .. controls (2.5,1) and (4,3)..(4,0)..controls (4,-3) and (2.5,-1).. (2,-1)..controls (1.5,-1) and (0,-3).. (0,0);
    \draw[gray] (0,0) ..controls (0,3) and (1.5,1) .. (2,1) .. controls (2.5,1) and (4,3)..(4,0)..controls (4,-3) and (2.5,-1).. (2,-1)..controls (1.5,-1) and (0,-3).. (0,0);
    \draw (0.5,-0.5) -- ++(0,0) node{$\bullet$} node[above]{$Z$};
    \draw (2.5,0) -- ++(0,-0.3) node{$\bullet$} node[above]{$Y$};
    \node[gray] at (4.2,1){$\Sigma$};
\end{tikzpicture} $\overset{\RR_\II(\Sigma)}\longmapsto$ 
\begin{tikzpicture}[yscale = 0.6, baseline = 5pt]
    \fill[gray!10] (0,0) ..controls (0,3) and (1.5,1) .. (2,1) .. controls (2.5,1) and (4,3)..(4,0)..controls (4,-3) and (2.5,-1).. (2,-1)..controls (1.5,-1) and (0,-3).. (0,0);
    \draw[gray] (0,0) ..controls (0,3) and (1.5,1) .. (2,1) .. controls (2.5,1) and (4,3)..(4,0)..controls (4,-3) and (2.5,-1).. (2,-1)..controls (1.5,-1) and (0,-3).. (0,0);
    \begin{scope}[yshift = 1.5cm]
    \fill[gray!10] (0,0) ..controls (0,3) and (1.5,1) .. (2,1) .. controls (2.5,1) and (4,3)..(4,0)..controls (4,-3) and (2.5,-1).. (2,-1)..controls (1.5,-1) and (0,-3).. (0,0);
    \draw[gray] (0,0) ..controls (0,3) and (1.5,1) .. (2,1) .. controls (2.5,1) and (4,3)..(4,0)..controls (4,-3) and (2.5,-1).. (2,-1)..controls (1.5,-1) and (0,-3).. (0,0);
    \end{scope}
    \draw (0.5,-0.5) node{$\bullet$} node[left]{$Z$} .. controls (0.5, 0.5) and (1.6,0.5).. (1.5,1) node[midway, sloped]{$\scriptstyle >$} node{$\bullet$} node[above]{$\scriptstyle f:Z\otimes Y \to \unit$} .. controls (2,0.9) and (2.5,0.4).. (2.5,-0.3)node[pos = 0.3, sloped]{$\scriptstyle <$} node{$\bullet$} node[right]{$Y$};
    \node[gray] at (4.5,1){$\Sigma\times I$};
\begin{scope}[yshift = 1.5cm]
    \draw[gray!10, line width = 3pt] (4,0)..controls (4,-3) and (2.5,-1).. (2,-1)..controls (1.5,-1) and (0,-3).. (0,0);
    \draw[gray] (4,0)..controls (4,-3) and (2.5,-1).. (2,-1)..controls (1.5,-1) and (0,-3).. (0,0);
    \end{scope}
    \node[yscale = 3.5, xscale = 2] at (-0.3,0.75){$\Big\langle$};
    \node[yscale = 3.5, xscale = 2] at (5.1,0.75){$\Big\rangle$};
\end{tikzpicture}
\caption{The ``empty" object as a presheaf on $\SkCat_\II(\Sigma)$. The configuration of colored points $X\in \SkCat_\II(\Sigma)$ at the left is mapped to the skein module $\Sk_\II(\Sigma\times I; X, \emptyset)$ at the right.}
    \label{fig:DistObj}
\end{figure}
\begin{definition}
    Let $M:\Sigma' \to \Sigma$ be a cobordism where $\pi_0(\Sigma')\to \pi_0(M)$ is surjective. We define a natural transformation $$\begin{tikzcd}[column sep = 50pt]
                \unit_\CC \ar[r, "\RR_\II \Sigma"] \ar[d, equal] & \SkFun_\II(\Sigma) \ar[d, "\SkFun_\II(M)"]\\
                \unit_\CC \ar[r, "\RR_\II \Sigma'"] & \SkFun_\II(\Sigma')
                \arrow[from = 2-1, to = 1-2, "\RR_\II(M)", Rightarrow,sloped, shorten = 10]
            \end{tikzcd}$$
            whose component on $X \in \SkCat_\II(\Sigma')$ is 
    $$\begin{array}{rcl}
         (\RR_\II(M))_X: \Sk_\II(\Sigma'\times [0,1]; X,\emptyset) &\to&  \Sk_\II(M;X,\emptyset)\\
         T &\mapsto& i_*T
    \end{array}$$
        where $i:\Sigma'\times [0,1]\inj M$ is the collar of the boundary, and we have used the equivalence 
        \begin{equation}\label{eq:BondM}
            (\SkFun_\II(M)\circ \RR_\II(\Sigma))(X):= \int^{Y \in \SkCat_\II(\Sigma)}\Sk_\II(\Sigma\times [0,1]; Y,\emptyset) \otimes \Sk_\II(M;X,Y) \overset{glue}{\simeq} \Sk_\II(M;X,\emptyset)
        \end{equation}
as in Theorem \ref{thm:SkCategTQFT}. The transported skein $i_*T$ is indeed admissible as $i$ is surjective on connected components. We think of this as ``extending $T$ by the empty skein in $M$". See Figure \ref{fig:BoundaryM}.

Note that if $\II=\AA$, then this is defined for every $M$, without the restriction that $\pi_0(\Sigma')\to \pi_0(M)$ is surjective.
\end{definition}

\begin{figure}
    \centering
\begin{tikzpicture}[xscale = 0.75,yscale = 0.45, baseline = 35pt]
    \fill[gray!10] (0,0) ..controls (0,3) and (1.5,1) .. (2,1) .. controls (2.5,1) and (4,3)..(4,0)..controls (4,-3) and (2.5,-1).. (2,-1)..controls (1.5,-1) and (0,-3).. (0,0);
    \draw[gray] (0,0) ..controls (0,3) and (1.5,1) .. (2,1) .. controls (2.5,1) and (4,3)..(4,0)..controls (4,-3) and (2.5,-1).. (2,-1)..controls (1.5,-1) and (0,-3).. (0,0);
    \begin{scope}[yshift = 1.5cm]
    \fill[gray!10] (0,0) ..controls (0,3) and (1.5,1) .. (2,1) .. controls (2.5,1) and (4,3)..(4,0)..controls (4,-3) and (2.5,-1).. (2,-1)..controls (1.5,-1) and (0,-3).. (0,0);
    \draw[gray] (0,0) ..controls (0,3) and (1.5,1) .. (2,1) .. controls (2.5,1) and (4,3)..(4,0)..controls (4,-3) and (2.5,-1).. (2,-1)..controls (1.5,-1) and (0,-3).. (0,0);
    \end{scope}
    \draw (0.5,-0.5) node{$\bullet$} node[left]{$Z$} .. controls (0.5, 0.5) and (1.6,0.5).. (1.5,1) node[midway, sloped]{$\scriptstyle >$} node{$\bullet$} node[above]{$\scriptstyle f:Z\otimes Y \to \unit$} .. controls (2,0.9) and (2.5,0.4).. (2.5,-0.3)node[pos = 0.3, sloped]{$\scriptstyle <$} node{$\bullet$} node[below right]{$Y$};
    \node[gray] at (2,3.2){$\Sigma'\times I$};
\begin{scope}[yshift = 1.5cm]
    \draw[gray!10, line width = 3pt] (4,0)..controls (4,-3) and (2.5,-1).. (2,-1)..controls (1.5,-1) and (0,-3).. (0,0);
    \draw[gray] (4,0)..controls (4,-3) and (2.5,-1).. (2,-1)..controls (1.5,-1) and (0,-3).. (0,0);
\end{scope}
\end{tikzpicture}
$
\overset{\RR_\II(M)}\longmapsto
$
\begin{tikzpicture}[xscale = 0.75,yscale = 0.45, baseline = 35pt]
    \fill[gray!10] (0,0) ..controls (0,3) and (1.5,1) .. (2,1) .. controls (2.5,1) and (4,3)..(4,0)..controls (4,-3) and (2.5,-1).. (2,-1)..controls (1.5,-1) and (0,-3).. (0,0);
    \draw[gray] (0,0) ..controls (0,3) and (1.5,1) .. (2,1) .. controls (2.5,1) and (4,3)..(4,0)..controls (4,-3) and (2.5,-1).. (2,-1)..controls (1.5,-1) and (0,-3).. (0,0);

    \fill[gray!10] (0,0) -- (0,1.5)..controls (0,2.5) and (-0.2,3).. (-0.2,3.5).. controls (-0.2,4.5) and (1,5).. (1,6) -- (1,7.5) arc (180:0:1) --(3,6) ..controls (3,5) and (4.2,4.5)..(4.2,3.5)..controls (4.2,3) and (4,2.5)..(4,1.5) -- (4,0);
    \draw[gray] (0,0) -- (0,1.5)..controls (0,2.5) and (-0.2,3).. (-0.2,3.5).. controls (-0.2,4.5) and (1,5).. (1,6) -- (1,7.5);
    \draw[gray] (1,6) arc(180:360:1);
    \draw[gray, dashed] (1,6) arc(180:0:1);
    \draw[gray] (1,7.5) arc(180:540:1);
    \draw[gray] (3,7.5)--(3,6) node[midway, right]{$\Sigma$}..controls (3,5) and (4.2,4.5)..(4.2,3.5) node[pos = 0.7,above=5pt]{$M$}..controls (4.2,3) and (4,2.5)..(4,1.5) -- (4,0) node[midway, right]{$\Sigma'$};
    
    \draw (0.5,-0.5) node{$\bullet$} node[left]{$Z$} .. controls (0.5, 0.5) and (1.6,0.5).. (1.5,1) node[midway, sloped]{$\scriptstyle >$} node{$\bullet$} node[above]{$\scriptstyle f:Z\otimes Y \to \unit$} .. controls (2,0.9) and (2.5,0.4).. (2.5,-0.3)node[pos = 0.3, sloped]{$\scriptstyle <$} node{$\bullet$} node[below right]{$Y$};

    \draw[gray, dashed] (0,0) ..controls (0,3) and (1.5,1) .. (2,1) .. controls (2.5,1) and (4,3)..(4,0);
\begin{scope}[yshift = 1.5cm]
    \draw[gray, dashed] (0,0) ..controls (0,3) and (1.5,1) .. (2,1) .. controls (2.5,1) and (4,3)..(4,0);
    \draw[gray!10, line width = 3pt] (4,0)..controls (4,-3) and (2.5,-1).. (2,-1)..controls (1.5,-1) and (0,-3).. (0,0);
    \draw[gray] (4,0)..controls (4,-3) and (2.5,-1).. (2,-1)..controls (1.5,-1) and (0,-3).. (0,0);
\end{scope}
\end{tikzpicture}
$
\overset{\eqref{eq:BondM}}{=}
$
\begin{tikzpicture}[xscale = 0.75,yscale = 0.45, baseline = 35pt]
    \fill[gray!10] (0,0) ..controls (0,3) and (1.5,1) .. (2,1) .. controls (2.5,1) and (4,3)..(4,0)..controls (4,-3) and (2.5,-1).. (2,-1)..controls (1.5,-1) and (0,-3).. (0,0);
    \draw[gray] (0,0) ..controls (0,3) and (1.5,1) .. (2,1) .. controls (2.5,1) and (4,3)..(4,0)..controls (4,-3) and (2.5,-1).. (2,-1)..controls (1.5,-1) and (0,-3).. (0,0);

    \fill[gray!10] (0,0) -- (0,1.5)..controls (0,2.5) and (-0.2,3).. (-0.2,3.5).. controls (-0.2,4.5) and (1,5).. (1,6) -- (1,7.5) arc (180:0:1) --(3,6) ..controls (3,5) and (4.2,4.5)..(4.2,3.5)..controls (4.2,3) and (4,2.5)..(4,1.5) -- (4,0);
    \draw[gray] (0,0) -- (0,1.5)..controls (0,2.5) and (-0.2,3).. (-0.2,3.5).. controls (-0.2,4.5) and (1,5).. (1,6) -- (1,7.5);
    \draw[gray, dashed] (1,6) arc(180:0:1);
    \draw[gray] (1,7.5) arc(180:540:1);
    \draw[gray] (3,7.5)--(3,6) ..controls (3,5) and (4.2,4.5)..(4.2,3.5) ..controls (4.2,3) and (4,2.5)..(4,1.5) -- (4,0);
    
    \draw (0.5,-0.5) node{$\bullet$} node[left]{$Z$} .. controls (0.5, 0.5) and (1.6,0.5).. (1.5,1) node[midway, sloped]{$\scriptstyle >$} node{$\bullet$} node[above]{$\scriptstyle f$} .. controls (2,0.9) and (2,2).. (2,6)node[pos = 0.3, sloped]{$\scriptstyle <$} node{$\bullet$}..controls (2, 7) and (2.3,7)..(2.3,6) node{$\bullet$} ..controls (2.3,4) and (2.5,1).. (2.5,-0.3) node[pos = 0.3, sloped]{$\scriptstyle <$} node{$\bullet$} node[below right]{$Y$};

    \draw[gray!10, line width = 5pt] (1.1,6) arc(180:360:0.9 and 1);
    \draw[gray] (1,6) arc(180:360:1);
    \draw[gray, dashed] (0,0) ..controls (0,3) and (1.5,1) .. (2,1) .. controls (2.5,1) and (4,3)..(4,0);
\begin{scope}[yshift = 1.5cm]
    \draw[gray, dashed] (0,0) ..controls (0,3) and (1.5,1) .. (2,1) .. controls (2.5,1) and (4,3)..(4,0);
    \draw[gray!10, line width = 3pt] (4,0)..controls (4,-3) and (2.5,-1).. (2,-1)..controls (1.5,-1) and (0,-3).. (0,0);
    \draw[gray] (4,0)..controls (4,-3) and (2.5,-1).. (2,-1)..controls (1.5,-1) and (0,-3).. (0,0);
\end{scope}
\node[yscale = 1.7] at (4.5,6.8){$\Big\}$};
\node[yscale = 2.4] at (4.5,2.7){$\Bigg\}$};
\node at (5.3,6.8){$\RR_\II(\Sigma)$};
\node at (5.6,2.7){$\SkFun_\II(M)$};
\end{tikzpicture}
\caption{The ``empty skein" as a boundary condition to $\SkFun_\II$. At the left is a skein $T \in \Sk_\II(\Sigma'\times [0,1]; X,\emptyset)$. In the middle is the same skein $i_*T \in \Sk_\II(M;X,\emptyset)$ pushed in $M$, which we think of as $T$ extended by the empty skein. At the right is the image of this skein under the equivalence $(\SkFun_\II(M)\circ \RR_\II(\Sigma))(X) \overset{\eqref{eq:BondM}}{\simeq} \Sk_\II(M;X,\emptyset)$.}
    \label{fig:BoundaryM}
\end{figure}
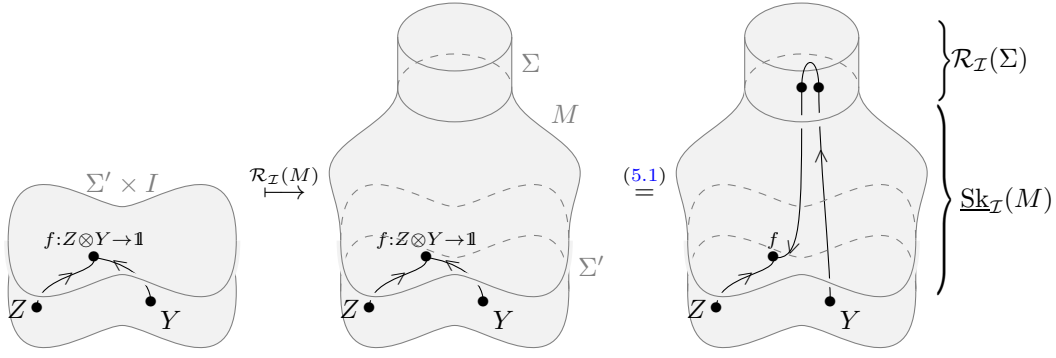

\begin{theorem}\label{thm:BondCond}
    There exists a (non compact if $\II\neq \AA$) boundary condition $$\RR_\II:\Triv \Rightarrow \SkFun_\II$$ with $\RR_\II(\Sigma)$ and $\RR_\II(M)$ as defined above. 
\end{theorem}
\begin{proof}
    The symmetric monoidal structure is the identity, which simplifies a lot the verification. We simply have to check naturality with respect to diffeomorphisms and composition.

    Let $f: M \to M'$ be a diffeomorphism preserving the collars, then indeed $f_*(i_*T) = i_*T$ as $i_*T$ is concentrated near the collars. 

    Let $\Sigma_3\overset{M_{32}}{\to} \Sigma_2 \overset{M_{21}}{\to}\Sigma_1$ be composeable, and take $T \in \Sk_\II(\Sigma_3\times [0,1]; X,\emptyset)$. Then $\RR(M_{32})(T)$ is $(i_3)_*T$, which we may isotope to meet $\Sigma_2\times I \subseteq M_{23}$ and write as $T'\cup T''$ for some $T' \in \Sk_\II(\Sigma_2\times [0,1]; Y,\emptyset), T''\in \Sk_\II(M;X,Y)$. Then $$(\RR_\II(M_{21})\circ_h \id_{M_{32}})\circ \RR_\II(M_{32})(T) = (i_2)_*T'\cup T'' \in \Sk_\II(M_{21}\circ M_{32};X,\emptyset)$$ is equal to $i_* \circ (i_3)_* (T)$, where $i:M_{32}\inj M_{21}\circ M_{32}$ is the canonical inclusion, and hence to $\RR_\II(M_{21}\circ M_{32})(T)$.
\end{proof}

\section{Non-semisimple WRT at the boundary of Crane--Yetter} \label{Sec_WRT}

In this section we will explain how to obtain Witten--Reshetikhin--Turaev 3-TQFTs, and their non-semisimple generalizations \cite{DGGPR}, from the once-extended 4-TQFT $\SS_\II$ defined in Section \ref{Sec_ExtendedCY} and its boundary condition $\RR_\II$ defined in Section \ref{Sec_BoundaryCond}.

We assume in this section that $\AA$ is a modular category in the sense of \cite{Lyub3mfldProjMCG, DGGPR} (which include semisimple modular categories in the sense of \cite{TuraevBook} but allow non-semisimple examples), and that $\II \subseteq \AA$ is the ideal of projective objects, so $\II = \AA$ if and only if $\AA$ is semisimple. 

\subsection{Anomalous theories}
We begin by explaining how to obtain an anomalous theory out of the data $\SS,\RR$. 
\begin{definition}
    The \textbf{category of filled $(n+1)$-cobordisms} $\cob_{n+1}^{filled}$ has:
    \begin{description}
        \item[Objects: ] Closed $n$-manifolds $\Sigma$ equipped with a bounding $(n+1)$-manifold $H: \emptyset\to \Sigma$
        \item[Morphisms: ] Cobordisms $M: \Sigma \to \Sigma'$ equipped with a bounding $(n+2)$-manifold 
        \begin{tikzcd}
            \emptyset \ar[d, equal]\ar[r, "H"] & \Sigma\ar[d, "M"']  \\
            \emptyset\ar[r, "H'"] & \Sigma'
            \arrow[from = 2-1, to = 1-2, "W", Leftarrow, shorten = 5]
        \end{tikzcd} 
    \end{description}
\end{definition}
\begin{definition}\label{def:anomalous}
    Let $\SS:\Cob_{2+1+1}^{hop}\to \CC$ be a once-extended TQFT contravariant in the direction of 1-morphisms and $\RR:\Triv \to \SS^\varepsilon$ a (resp. non-compact) boundary condition to $\SS$. 

    The \textbf{anomalous theory}\footnote{This terminology is maybe only appropriate when $\SS$ is an invertible theory, which will be the case in our example.} $\Aa_{\SS,\RR}$ associated to $\SS,\RR$ is the symmetric monoidal functor
    $$ \begin{array}{rcl}
         \Aa_{\SS,\RR} : \cob_{n+1}^{filled}& \to& \Omega\CC:= \End_\CC(\unit_\CC)\\
         \left[\emptyset\overset{H}{\to} \Sigma\right] &\mapsto & \hskip4pt\left[ \unit_\CC \overset{\RR (\Sigma)}{\longrightarrow} \SS(\Sigma) \overset{\SS(H)}{\longrightarrow} \unit_\CC \right]\\
         \left[ \begin{tikzcd}
            \emptyset \ar[d, equal]\ar[r, "H"] & \Sigma\ar[d, "M"']  \\
            \emptyset\ar[r, "H'"] & \Sigma'
            \arrow[from = 2-1, to = 1-2, "W", Leftarrow, shorten = 5]
        \end{tikzcd} \right] &\mapsto & \hskip-4pt
    \left[ \begin{tikzcd}[column sep = 45] \unit_\CC \ar[r, "\RR(\Sigma)"] \ar[d, equal] & \SS(\Sigma) \ar[r, "\SS(H)"]\ar[d, "\SS(M)"', leftarrow] & \emptyset \ar[d, equal] \\
            \unit_\CC \ar[r, "\RR(\Sigma')"] & \SS(\Sigma')\ar[r, "\SS(H')", near start] & \emptyset
            \arrow[from = 1-1, to = 2-2, "\RR(M)", Rightarrow,sloped, shorten = 5]
            \arrow[from = 1-2, to = 2-3, "\SS(W)", Rightarrow,sloped, shorten = 5]
        \end{tikzcd} \right]
    \end{array}$$
    If $\RR$ is non-compact, then $\Aa_{\SS,\RR}$ is also non-compact, i.e. it is only defined on the category $\cob_{2+1}^{filled, nc}$ of cobordisms with non-empty incoming boundary in every connected component.
\end{definition}

The category of filled cobordisms has been suggested by Walker as the right source category for WRT theories. It is more standard to consider a smaller category where we have forgotten part of the data of the filling.
\begin{definition}
    The category $\widetilde \cob_{2+1}$ has objects surfaces equipped with a Lagrangian $L\subseteq H_1(\Sigma)$ and morphisms 3-cobordisms equipped with an integer $n\in \mathbb Z$. Composition is given by composing the cobordisms, adding the integers and adding a Maslov index of the three Lagrangians involved as defined in \cite[Section 2]{WalkerOnWitten}. The category $\widetilde \cob_{2+1}^{nc}$ is the subcategory where 3-cobordisms must have incoming boundary in every connected component.

    The projection $$\pi: \cob_{2+1}^{filled}\to \widetilde \cob_{2+1}$$ takes a pair $(\Sigma,H)$ to $(\Sigma, \operatorname{Ker}(i_*:H_1(\Sigma\to H_1(H)))$ and a pair $(M,W)$ to $(M, \sigma(W))$. Composition in $\widetilde \cob_{2+1}$ is defined precisely to make this assignment preserve composition, using Wall's non-additivity theorem.
\end{definition}

\subsection{Description of the skein-theoretic anomalous theory}\label{ssec:DescrAnomalous}
Let us describe the anomalous theory $\Aa_\II$ associated to the once-extended 4-TQFT $\SS_\II$ defined in Section \ref{Sec_ExtendedCY} and its boundary condition $\RR_\II$ defined in Section \ref{Sec_BoundaryCond}, for modular inputs.


Let $\AA$ be a modular tensor category, $\II = \operatorname{Proj}(\AA)$, and choose $\mt$ a modified trace on $\II$ with global dimension $\zeta=1$ (equivalently, choose any modified trace $\mt$, pick a square root $\DD$ of its global dimension and consider $\DD^{-1}\mt$). The ribbon category $\AA$ is in particular chromatic compact, Theorem \ref{thm:ExtendedCGHP} constructs a once-extended 4-TQFT $\SS_\II$, and Theorem \ref{thm:BondCond} constructs a boundary condition $\RR_\II$ to it, which is non-compact when $\II\neq \AA$.

We denote $$\Aa_\II: \cob_{2+1}^{filled, nc} \to \Omega \Bimod \simeq \Vect$$ the anomalous theory associated to $\SS_\II, \RR_\II$, which is non-compact when $\II\neq \AA$.

\paragraph{State spaces of surfaces:} Let $\emptyset\overset{H}{\to} \Sigma$ be a filled surface. By definition $$\Aa_\II(\emptyset\overset{H}{\to} \Sigma) := \SkFun_\II(H)\circ \RR_\II(\Sigma) := \int^{X\in \SkCat_\II(\Sigma)}\Sk_\II(\Sigma\times [0,1]; X,\emptyset) \otimes \Sk_\II(H;X)$$
and, by Theorem \ref{thm:SkCategTQFT}, gluing gives an isomorphism
$$\Aa_\II(\emptyset\overset{H}{\to} \Sigma) \simeq \Sk_\II(H;\emptyset)\ .$$
See Figure \ref{fig:StateSpaceAnom}.
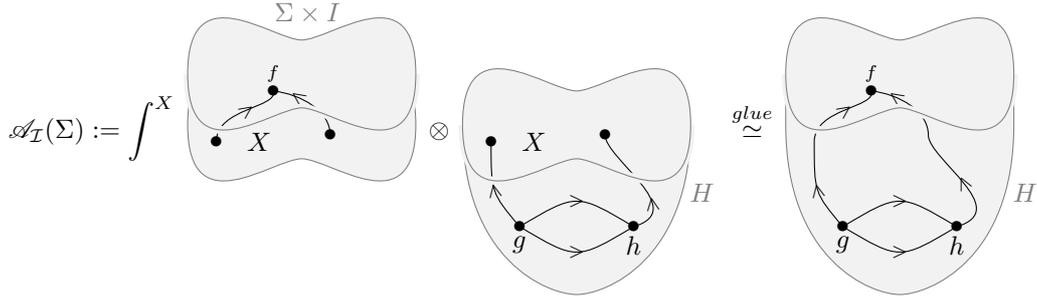
\begin{figure}
    \centering
$\displaystyle \Aa_\II(\Sigma) := \int^X$
\begin{tikzpicture}[xscale = 0.75,yscale = 0.45, baseline = -5pt]
    \fill[gray!10] (0,0) ..controls (0,3) and (1.5,1) .. (2,1) .. controls (2.5,1) and (4,3)..(4,0)..controls (4,-3) and (2.5,-1).. (2,-1)..controls (1.5,-1) and (0,-3).. (0,0);
    \draw[gray] (0,0) ..controls (0,3) and (1.5,1) .. (2,1) .. controls (2.5,1) and (4,3)..(4,0)..controls (4,-3) and (2.5,-1).. (2,-1)..controls (1.5,-1) and (0,-3).. (0,0);
    \begin{scope}[yshift = 1.5cm]
    \fill[gray!10] (0,0) ..controls (0,3) and (1.5,1) .. (2,1) .. controls (2.5,1) and (4,3)..(4,0)..controls (4,-3) and (2.5,-1).. (2,-1)..controls (1.5,-1) and (0,-3).. (0,0);
    \draw[gray] (0,0) ..controls (0,3) and (1.5,1) .. (2,1) .. controls (2.5,1) and (4,3)..(4,0)..controls (4,-3) and (2.5,-1).. (2,-1)..controls (1.5,-1) and (0,-3).. (0,0);
    \end{scope}
    \draw (0.5,-0.5) node{$\bullet$} node[right=8pt]{$X$} .. controls (0.5, 0.5) and (1.6,0.5).. (1.5,1) node[midway, sloped]{$\scriptstyle >$} node{$\bullet$} node[above]{$\scriptstyle f$} .. controls (2,0.9) and (2.5,0.4).. (2.5,-0.3)node[pos = 0.3, sloped]{$\scriptstyle <$} node{$\bullet$};
    \node[gray] at (2.1,3.3){$\Sigma\times I$};
\begin{scope}[yshift = 1.5cm]
    \draw[gray!10, line width = 3pt] (4,0)..controls (4,-3) and (2.5,-1).. (2,-1)..controls (1.5,-1) and (0,-3).. (0,0);
    \draw[gray] (4,0)..controls (4,-3) and (2.5,-1).. (2,-1)..controls (1.5,-1) and (0,-3).. (0,0);
\end{scope}
\end{tikzpicture}  
$
\otimes
$
\begin{tikzpicture}[xscale = 0.75,yscale = 0.45, baseline = -5pt]
    \fill[gray!10] (0,0) ..controls (0,3) and (1.5,1) .. (2,1) .. controls (2.5,1) and (4,3)..(4,0)..controls (4,-3) and (2.5,-1).. (2,-1)..controls (1.5,-1) and (0,-3).. (0,0);
    \fill[gray!10] (4,0)..controls (4,-4) and (2.5,-5).. (2,-5)..controls (1.5,-5) and (0,-4).. (0,0)--cycle;
    \draw[gray] (0,0) ..controls (0,3) and (1.5,1) .. (2,1) .. controls (2.5,1) and (4,3)..(4,0)..controls (4,-3) and (2.5,-1).. (2,-1)..controls (1.5,-1) and (0,-3).. (0,0);
    \draw[gray] (4,0)..controls (4,-4) and (2.5,-5).. (2,-5)..controls (1.5,-5) and (0,-4).. (0,0);
    \draw (0.5,-0.5) node{$\bullet$} node[right=8pt]{$X$}..controls (0.5,-2) .. (1,-3)node[midway, sloped]{$\scriptstyle <$} node{$\bullet$} node[below]{$g$} ..controls (2,-2).. (3,-3)node[midway, sloped]{$\scriptstyle >$} node{$\bullet$} node[below]{$h$} .. controls (4,-2.5) and (2.5,-1).. (2.5,-0.3)node[pos = 0.3, sloped]{$\scriptstyle <$} node{$\bullet$};
    \draw (1,-3)  ..controls (2,-4).. (3,-3)node[midway, sloped]{$\scriptstyle >$};
    \node[gray] at (4.2,-2){$H$};
    \draw[gray!10, line width = 3pt] (4,0)..controls (4,-3) and (2.5,-1).. (2,-1)..controls (1.5,-1) and (0,-3).. (0,0);
    \draw[gray] (4,0)..controls (4,-3) and (2.5,-1).. (2,-1)..controls (1.5,-1) and (0,-3).. (0,0);
\end{tikzpicture}  
$
\overset{glue}{\simeq}
$
\begin{tikzpicture}[xscale = 0.75,yscale = 0.45, baseline = -5pt]
    \fill[gray!10] (0,0) ..controls (0,3) and (1.5,1) .. (2,1) .. controls (2.5,1) and (4,3)..(4,0)..controls (4,-3) and (2.5,-1).. (2,-1)..controls (1.5,-1) and (0,-3).. (0,0);
    \fill[gray!10] (4,0)..controls (4,-4) and (2.5,-5).. (2,-5)..controls (1.5,-5) and (0,-4).. (0,0)--cycle;
    \draw[gray] (4,1.5)--(4,0)..controls (4,-4) and (2.5,-5).. (2,-5)..controls (1.5,-5) and (0,-4).. (0,0)--(0,1.5);
    \begin{scope}[yshift = 1.5cm]
    \fill[gray!10] (0,0) ..controls (0,3) and (1.5,1) .. (2,1) .. controls (2.5,1) and (4,3)..(4,0)..controls (4,-3) and (2.5,-1).. (2,-1)..controls (1.5,-1) and (0,-3).. (0,0);
    \draw[gray] (0,0) ..controls (0,3) and (1.5,1) .. (2,1) .. controls (2.5,1) and (4,3)..(4,0)..controls (4,-3) and (2.5,-1).. (2,-1)..controls (1.5,-1) and (0,-3).. (0,0);
    \end{scope}
    \draw (0.5,-0.5).. controls (0.5, 0.5) and (1.6,0.5).. (1.5,1) node[midway, sloped]{$\scriptstyle >$} node{$\bullet$} node[above]{$\scriptstyle f$} .. controls (2,0.9) and (2.5,0.4).. (2.5,-0.3)node[pos = 0.3, sloped]{$\scriptstyle <$};
    \draw (0.5,-0.5)..controls (0.5,-2) .. (1,-3)node[midway, sloped]{$\scriptstyle <$} node{$\bullet$} node[below]{$g$} ..controls (2,-2).. (3,-3)node[midway, sloped]{$\scriptstyle >$} node{$\bullet$} node[below]{$h$} .. controls (4,-2.5) and (2.5,-1).. (2.5,-0.3)node[midway, sloped]{$\scriptstyle <$};
    \draw (1,-3)  ..controls (2,-4).. (3,-3)node[midway, sloped]{$\scriptstyle >$};
    \node[gray] at (4.2,-2){$H$};
\begin{scope}[yshift = 1.5cm]
    \draw[gray!10, line width = 3pt] (4,0)..controls (4,-3) and (2.5,-1).. (2,-1)..controls (1.5,-1) and (0,-3).. (0,0);
    \draw[gray] (4,0)..controls (4,-3) and (2.5,-1).. (2,-1)..controls (1.5,-1) and (0,-3).. (0,0);
\end{scope}
\end{tikzpicture}    \caption{The state spaces $\Aa_\II(\Sigma)$ of the anomalous theory on a filled surface $\emptyset\overset{H}{\to} \Sigma$ is the admissible skein module of $H$ without boundary points.}
    \label{fig:StateSpaceAnom}
\end{figure}

\paragraph{Correlation functions of 3-cobordisms:} Let $M:\Sigma \to \Sigma'$ be a 3-cobordism and $W:H\underset{\Sigma}{\cup}M \Rightarrow H'$ a filling. If $\II\neq\AA$, suppose that $\pi_0(\Sigma)\twoheadrightarrow  \pi_0(M)$. Denote $i_\Sigma:\Sigma\times I \to M$ the collar and $i:H \subseteq H\underset{\Sigma}{\cup}M$ the inclusion. By definition,
$$    \Aa_\II(M,W) := (\SS_\II(W) \circ_h  \id_{\RR_\II(\Sigma')}) \circ (\id_{\SkFun_\II(H)}\circ_h \RR_\II(M)) \ .
$$
Let $T \in \Sk_\II(H;\emptyset)$ be an admissible skein, which one can write as $T_H \otimes T_\Sigma \in \SkFun_\II(H) \circ \RR_\II(\Sigma) \simeq  \Sk_\II(H)$. Then $$(\id_{\SkFun_\II(H)}\circ_h\RR_\II(M))(T) := i_{\Sigma,*}T_\Sigma \otimes T_H \overset{glue}\simeq i_*T \in \Sk_\II(H\underset{\Sigma}{\cup}M;\emptyset)$$
and $$\Aa_\II(M,W)(T) = \SS_\II(W)(i_*T)\ .$$
See Figure \ref{fig:AaofM}.

\begin{figure}
    \centering
\begin{tikzpicture}[xscale = 0.75,yscale = 0.45, baseline = 15pt]
    \fill[gray!10] (0,0) ..controls (0,3) and (1.5,1) .. (2,1) .. controls (2.5,1) and (4,3)..(4,0)..controls (4,-3) and (2.5,-1).. (2,-1)..controls (1.5,-1) and (0,-3).. (0,0);
    \fill[gray!10] (4,0)..controls (4,-4) and (2.5,-5).. (2,-5)..controls (1.5,-5) and (0,-4).. (0,0)--cycle;
    \draw[gray] (4,1.5)--(4,0)..controls (4,-4) and (2.5,-5).. (2,-5)..controls (1.5,-5) and (0,-4).. (0,0)--(0,1.5);
    \begin{scope}[yshift = 1.5cm]
    \fill[gray!10] (0,0) ..controls (0,3) and (1.5,1) .. (2,1) .. controls (2.5,1) and (4,3)..(4,0)..controls (4,-3) and (2.5,-1).. (2,-1)..controls (1.5,-1) and (0,-3).. (0,0);
    \draw[gray] (0,0) ..controls (0,3) and (1.5,1) .. (2,1) .. controls (2.5,1) and (4,3)..(4,0) node[right]{$\Sigma$}..controls (4,-3) and (2.5,-1).. (2,-1)..controls (1.5,-1) and (0,-3).. (0,0);
    \end{scope}
    \draw (0.5,-0.5).. controls (0.5, 0.5) and (1.6,0.5).. (1.5,1) node{$\bullet$} node[above]{$\scriptstyle f$} .. controls (2,0.9) and (2.5,0.4).. (2.5,-0.3);
    \draw (0.5,-0.5)..controls (0.5,-2) .. (1,-3)node[midway, sloped]{$\scriptstyle <$} node{$\bullet$} node[below]{$g$} ..controls (2,-2).. (3,-3)node[midway, sloped]{$\scriptstyle >$} node{$\bullet$} node[below]{$h$} .. controls (4,-2.5) and (2.5,-1).. (2.5,-0.3)node[midway, sloped]{$\scriptstyle <$};
    \draw (1,-3)  ..controls (2,-4).. (3,-3)node[midway, sloped]{$\scriptstyle >$};
    \node[gray] at (4.2,-2){$H$};
\begin{scope}[yshift = 1.5cm]
    \draw[gray!10, line width = 3pt] (4,0)..controls (4,-3) and (2.5,-1).. (2,-1)..controls (1.5,-1) and (0,-3).. (0,0);
    \draw[gray] (4,0)..controls (4,-3) and (2.5,-1).. (2,-1)..controls (1.5,-1) and (0,-3).. (0,0);
\end{scope}
\end{tikzpicture}
$
\overset{\RR}{\longmapsto}
$
\begin{tikzpicture}[xscale = 0.75,yscale = 0.45, baseline = 15pt]
    \fill[gray!10] (0,0) ..controls (0,3) and (1.5,1) .. (2,1) .. controls (2.5,1) and (4,3)..(4,0)..controls (4,-3) and (2.5,-1).. (2,-1)..controls (1.5,-1) and (0,-3).. (0,0);
    \fill[gray!10] (0,0) -- (0,1.5)..controls (0,2.5) and (-0.2,3).. (-0.2,3.5).. controls (-0.2,4.5) and (1,5).. (1,6) -- (1,7.5) arc (180:0:1) --(3,6) ..controls (3,5) and (4.2,4.5)..(4.2,3.5)..controls (4.2,3) and (4,2.5)..(4,1.5) -- (4,0);
    \draw[gray] (0,0) -- (0,1.5)..controls (0,2.5) and (-0.2,3).. (-0.2,3.5).. controls (-0.2,4.5) and (1,5).. (1,6) -- (1,7.5);
    \draw[gray] (1,7.5) arc(180:540:1);

    \fill[gray!10] (0,0) ..controls (0,3) and (1.5,1) .. (2,1) .. controls (2.5,1) and (4,3)..(4,0)..controls (4,-3) and (2.5,-1).. (2,-1)..controls (1.5,-1) and (0,-3).. (0,0);
    \fill[gray!10] (4,0)..controls (4,-4) and (2.5,-5).. (2,-5)..controls (1.5,-5) and (0,-4).. (0,0)--cycle;
    \draw[gray] (4,1.5)--(4,0)..controls (4,-4) and (2.5,-5).. (2,-5)..controls (1.5,-5) and (0,-4).. (0,0)--(0,1.5);
    \begin{scope}[yshift = 1.5cm]
    \fill[gray!10] (0,0) ..controls (0,3) and (1.5,1) .. (2,1) .. controls (2.5,1) and (4,3)..(4,0)..controls (4,-3) and (2.5,-1).. (2,-1)..controls (1.5,-1) and (0,-3).. (0,0);
    \end{scope}
    \draw (0.5,-0.5).. controls (0.5, 0.5) and (1.6,0.5).. (1.5,1) node{$\bullet$} node[above]{$\scriptstyle f$} .. controls (2,0.9) and (2.5,0.4).. (2.5,-0.3);
    \draw (0.5,-0.5)..controls (0.5,-2) .. (1,-3)node[midway, sloped]{$\scriptstyle <$} node{$\bullet$} node[below]{$g$} ..controls (2,-2).. (3,-3)node[midway, sloped]{$\scriptstyle >$} node{$\bullet$} node[below]{$h$} .. controls (4,-2.5) and (2.5,-1).. (2.5,-0.3)node[midway, sloped]{$\scriptstyle <$};
    \draw (1,-3)  ..controls (2,-4).. (3,-3)node[midway, sloped]{$\scriptstyle >$};
    \node[gray] at (4.2,-2){$H$};
    
    \draw[gray] (3,7.5)--(3,6) node[pos =0, right]{$\Sigma'$}..controls (3,5) and (4.2,4.5)..(4.2,3.5) node[pos = 0.7,above=5pt]{$M$}..controls (4.2,3) and (4,2.5)..(4,1.5) -- (4,0);
\begin{scope}[yshift = 1.5cm]
    \draw[gray!10, line width = 3pt] (4,0)..controls (4,-3) and (2.5,-1).. (2,-1)..controls (1.5,-1) and (0,-3).. (0,0);
    \draw[gray] (4,0)..controls (4,-3) and (2.5,-1).. (2,-1)..controls (1.5,-1) and (0,-3).. (0,0);
    \draw[gray, dashed] (0,0) ..controls (0,3) and (1.5,1) .. (2,1) .. controls (2.5,1) and (4,3)..(4,0);
\end{scope}
\end{tikzpicture}
$
\overset{\SS}{\longmapsto}
$
\begin{tikzpicture}[xscale = 0.75,yscale = 0.45, baseline = 15pt]
    \fill[gray!10] (0,0) ..controls (0,3) and (1.5,1) .. (2,1) .. controls (2.5,1) and (4,3)..(4,0)..controls (4,-3) and (2.5,-1).. (2,-1)..controls (1.5,-1) and (0,-3).. (0,0);
    \fill[gray!10] (0,0) -- (0,1.5)..controls (0,2.5) and (-0.2,3).. (-0.2,3.5).. controls (-0.2,4.5) and (1,5).. (1,6) -- (1,7.5) arc (180:0:1) --(3,6)..controls (3,5) and (4,3.5)..(4,1.5) -- (4,0);
    \draw[gray] (0,0) -- (0,1.5)..controls (0,2.5) and (-0.2,3).. (-0.2,3.5).. controls (-0.2,4.5) and (1,5).. (1,6) -- (1,7.5);
    \draw[gray] (1,7.5) arc(180:540:1);
    \fill[gray!10] (0,0) ..controls (0,3) and (1.5,1) .. (2,1) .. controls (2.5,1) and (4,3)..(4,0)..controls (4,-3) and (2.5,-1).. (2,-1)..controls (1.5,-1) and (0,-3).. (0,0);
    \fill[gray!10] (4,0)..controls (4,-4) and (2.5,-5).. (2,-5)..controls (1.5,-5) and (0,-4).. (0,0)--cycle;
    \draw[gray] (4,1.5)--(4,0)..controls (4,-4) and (2.5,-5).. (2,-5)..controls (1.5,-5) and (0,-4).. (0,0)--(0,1.5);
    \begin{scope}[yshift = 1.5cm]
    \fill[gray!10] (0,0) ..controls (0,3) and (1.5,1) .. (2,1) .. controls (2.5,1) and (4,3)..(4,0)..controls (4,-3) and (2.5,-1).. (2,-1)..controls (1.5,-1) and (0,-3).. (0,0);
    \end{scope}
    \draw (0.5,-0.5).. controls (0.5, 0.5) and (1.6,0.5).. (1.5,1) node{$\bullet$} node[above]{$\scriptstyle f$} .. controls (2,0.9) and (2.5,0.4).. (2.5,-0.3);
    \draw (0.5,-0.5)..controls (0.5,-2) .. (1,-3)node[midway, sloped]{$\scriptstyle <$} node{$\bullet$} node[below]{$g$} ..controls (2,-2).. (3,-3)node[midway, sloped]{$\scriptstyle >$} node{$\bullet$} node[below]{$h$} .. controls (4,-2.5) and (2.5,-1).. (2.5,-0.3)node[midway, sloped]{$\scriptstyle <$};
    \draw (1,-3)  ..controls (2,-4).. (3,-3)node[midway, sloped]{$\scriptstyle >$};
    \node[gray] at (4.4,1){$H'$};
    \draw[gray] (3,7.5)--(3,6) node[pos=0, right]{$\Sigma'$}..controls (3,5) and (4,3.5)..(4,1.5) -- (4,0);
    \draw[gray!10, line width = 5pt] (2,4).. controls (2.5,4) and (3,3).. (3,1)..controls (3,0) and (2.5,-1)..(1.5,-1)..controls (1,-1) and (0.5,0).. (0.5,1) ..controls (0.5,2) and (1,4)..(2,4);
    \draw[red] (2,4).. controls (2.5,4) and (3,3).. (3,1)..controls (3,0) and (2.5,-1)..(1.5,-1)..controls (1,-1) and (0.5,0).. (0.5,1) ..controls (0.5,2) and (1,4)..(2,4);
    \draw[gray!10, line width = 5pt] (0.5,-0.5).. controls (0.5, 0.5) and (1.6,0.5).. (1.5,0.8);
    \draw (0.5,-0.5).. controls (0.5, 0.5) and (1.6,0.5).. (1.5,1);
\end{tikzpicture}
\caption{The anomalous theory on a filled 3-cobordism. Here we assumed $W$ is given by a single 2-handle, so the 3-manifold $H'$ is obtained from $H\underset{\Sigma}{\cup}M$ by a single surgery.}
    \label{fig:AaofM}
\end{figure}
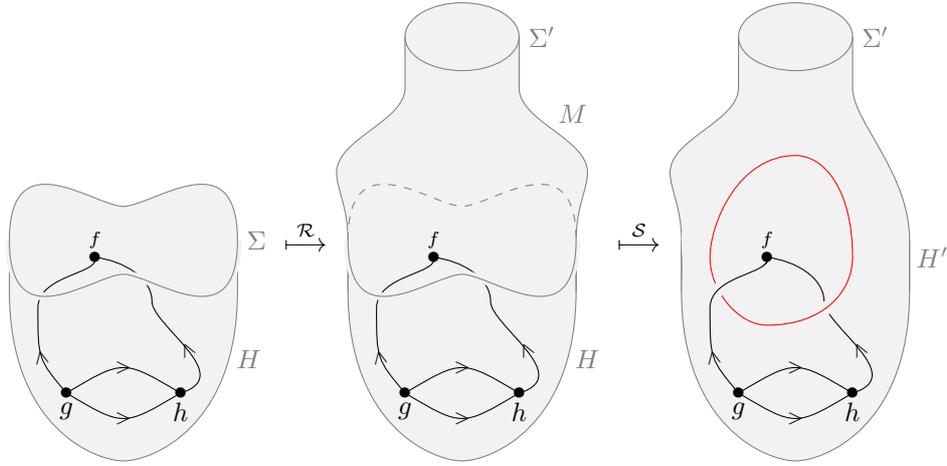

\subsection{Non-semisimple WRT at the boundary of Crane--Yetter}
We are now able to compare the anomalous theory $\Aa_\II$ described above and the Witten--Reshetikhin--Turaev and DGGPR TQFTs. We will consider the latter two as TQFTs in the usual sense, without decorations by ribbon graphs in the category of 3-cobordisms.
\begin{definition}
    Let $\AA$ be a semisimple modular tensor category and $\DD$ a chosen square root of its global dimension. The Witten--Reshetikhin--Turaev TQFT is the symmetric monoidal functor $$\WRT_\AA: \widetilde \cob_{2+1} \to \Vect$$ defined in \cite{TuraevBook, WalkerOnWitten}.
\end{definition}
\begin{definition}
    Let $\AA$ be a possibly non-semisimple modular tensor category, $\II = \operatorname{Proj}(\AA)$, $\mt$ a modified trace on $\II$ and $\DD$ a chosen square root of the global dimension\footnote{Called modularity parameter in \cite{DGGPR}} $\zeta$ (so $\DD^{-1}\mt$ is a modified trace with $\zeta = 1$). The DGGPR TQFT is the symmetric monoidal functor $$\DGGPR_\AA: \widetilde \cob_{2+1}^{nc} \to \Vect$$ defined in \cite{DGGPR}. It generalizes the WRT theories above. 
\end{definition}
\begin{theorem}\label{thm:WRT}
    Let $\AA$ be a semisimple modular tensor category and $\II=\AA$, then 
    \begin{equation*}
        \begin{tikzcd}
            \cob_{2+1}^{filled} \ar[rr, "\Aa_\AA"] && \Vect\\
            & \widetilde \cob_{2+1} &
            \arrow[from = 1-1, to = 2-2, "{\pi}"]
            \arrow[from = 2-2, to = 1-3, "\WRT_\AA"', sloped]
        \end{tikzcd}
    \end{equation*}
    commutes up to symmetric monoidal natural isomorphism.   
\end{theorem}
\begin{theorem}\label{thm:DGGPR}
    Let $\AA$ be a non-semisimple modular tensor category and $\II = \operatorname{Proj}(\AA)$, then 
    \begin{equation*}
        \begin{tikzcd}
            \cob_{2+1}^{filled, nc} \ar[rr, "\Aa_\II"] && \Vect\\
            & \widetilde \cob_{2+1}^{nc} &
            \arrow[from = 1-1, to = 2-2, "{\pi}"]
            \arrow[from = 2-2, to = 1-3, "\DGGPR_\AA"', sloped]
        \end{tikzcd}
    \end{equation*}
    commutes up to symmetric monoidal natural isomorphism.   
\end{theorem}
\begin{proof}
Given $\emptyset \overset{H}{\to}\Sigma$, we need to give a natural isomorphism 
$$\eta_{\Sigma,H}:\Aa_\II(\Sigma) \tilde\to \DGGPR_\AA(\Sigma)\ .$$
On the one hand, we have $$\Aa_\II(\Sigma) := \int^{X \in \SkCat_\II(\Sigma)} \Sk_\II(\Sigma\times [0,1]; X,\emptyset) \otimes \Sk_\II(H;X) \overset{glue}{\simeq} \Sk_\II(H;\emptyset)$$
is the admissible skein module of $H$ with empty boundary points. 

On the other hand, the state spaces of $\DGGPR$ are defined via the universal construction of \cite{BHMV95}, i.e. as a quotient of the vector space generated by all 3-manifolds $N$ bounding $\Sigma$ equipped with an admissible $\II$-colored ribbon graph $T$. Let us denote $[N,T] \in \DGGPR(\Sigma)$ the induced vector. The quotient asks the relation $\sum_i [N_i, T_i]=0$ if for every $N', T_{N'}$ of boundary $-\Sigma$, the invariant of closed 3-manifold $\sum_i \DGGPR(N_i\underset{\Sigma}{\cup}N', T_i\sqcup T', n_i) = 0$, where $n_i$ is a Maslov index computed in the composition of $N'$ and $N_i$.

The map $\eta_{\Sigma,H}$ is the canonical map to the quotient. It is shown to be well-defined in \cite[Prop. 4.11]{DGGPR} (and is called $\pi$ there), and surjective when $H$ is connected. We defer to Lemma \ref{Lem_DGGPRStateSpaces} the proof that it is an isomorphism. 

The symmetric monoidal structure of $\DGGPR$ is given by taking disjoint union on these generators \cite[Prop. 4.8]{DGGPR}, hence $\eta_{\Sigma,H}$ is symmetric monoidal. 

We are left with the core of the proof: checking that $\eta_{\Sigma,H}$ is natural. Let $M: \Sigma\to \Sigma'$ be a 3-cobordism between filled surfaces equipped with a bounding 4-manifold $W : H\underset{\Sigma}{\cup}M \Rightarrow H'$. 

The action of $M,W$ on a vector $[H,T]\in \DGGPR(\Sigma)$ is given by $$\DGGPR(M,\sigma(W))([H,T]) = \SS_\II(\mathbb C P_2)^{\sigma(W)} [H\underset{\Sigma}{\cup}M,T]\ .$$
This is now a skein living in the gluing of $M$ and $H$, and not in $H'$ as we would hope. We would like to relate it to a element of the form $[H', T']$. This is asking: what is the skein $T' \subseteq H'$ so that the invariants $$\SS_\II(\mathbb C P_2)^{\sigma(W)+n-n'}\DGGPR(H\underset{\Sigma}{\cup}M\underset{\Sigma'}{\cup}N', T \cup T_{N'}) \overset ? = \DGGPR(H'\underset{\Sigma'}{\cup}N', T'\cup T_{N'})$$ match for any $N', T_{N'}$, where $n$ and $n'$ are Maslov indices.

By Proposition \ref{Prop_ZdepOnSignature}, each of these scalars is given by bounding the 3-manifold by a 4-manifold, and evaluation it under the TQFT $\SS_\II$. Let $W': H'\underset{\Sigma'}{\cup}N' \to \emptyset$ be a bounding 4-manifold. Then $W' \circ (W \circ_h \id_{N'})$ is a bounding manifold for $H\underset{\Sigma}{\cup}M\underset{\Sigma'}{\cup}N'$. By construction, the skein $\SS_\II(W)(T) \subseteq H'$ satisfies $\SS_\II(W')(\SS_\II(W)(T)\cup T_{N'}) = \SS_\II(W' \circ (W \circ_h \id_{N'}))(T \cup T_{N'})$, i.e. precisely $$\SS_\II(\mathbb C P_2)^{\sigma(W)}[H\underset{\Sigma}{\cup}M,T] = [H', \SS_\II(W)(T) ]$$ 
as the integer $\sigma(W) +n-n'$ computes the difference of signature between $W'$ and $W'\circ (W \circ_h \id_{N'})$. This concludes naturality.
\end{proof}

\begin{lemma}\label{Lem_DGGPRStateSpaces}
    The natural transformation $\eta$ 
    is a natural isomorphism.
\end{lemma}
\begin{proof}
    Let us first suppose that $\Sigma$ is connected and $H: \emptyset \to \Sigma$ is a handle body. It is shown in \cite[Prop. 4.11]{DGGPR} that $\eta_{\Sigma, H}$ is surjective. We will prove that it is an isomorphism by showing that $\Sk_\II(H)$ and $\DGGPR_\AA(\Sigma)$ have the same dimension.
    
    We will use the coend description of state spaces of \cite[Sec. 4.1]{DGGPR}. Remember that the coend ${\mathcal L}$ is defined as the colimit $${\mathcal L} = \int^{X\in\AA} X\otimes X^* = \big(\oplus_{X\in\AA} X\otimes X^*\big)/\langle (f,\id)\sim (\id, f^*), f:X\to Y\rangle$$
    We only consider projectives in our case, but by
    \cite[Proposition 5.1.7]{KerlerLyubBook} this does not change this colimit and ${\mathcal L} \simeq \int^{P\in\II} P\otimes P^*$. Note that by \cite[Corollary 5.1.8]{KerlerLyubBook}, the infinite nature of this colimit is unnecessary, and we could allow only $P=G$ the projective generator. We will still denote it $\int^{P\in\II}$, but it will be useful to remember that everything is finite.\\
    It is shown in \cite[Proposition 4.17 and Lemma 4.1 at $V=\unit$]{DGGPR} that 
    $$\DGGPR_\AA(\Sigma) \simeq \big(\Hom_\AA({\mathcal L}^{\otimes g},\unit)\big)^*$$
    where $g$ is the genus of $\Sigma$. Using the definition of the colimit, the vector space $\Hom_\AA({\mathcal L}^{\otimes g},\unit)$ is obtained as a limit: the subspace of the product $\Pi\Hom_\AA(P_1\otimes P_1^*\otimes\dots\otimes P_g\otimes P_g^*,\unit)$ of the collections that satisfy the $(f,\id)\sim (\id, f^*)$ relations. The dual of this limit is then (using the fact everything is finite) the colimit $$\DGGPR_\AA(\Sigma)\hskip-3pt \simeq\hskip-3pt \big( \underset{(P_i)_i \in \II^g}{\bigoplus}\hskip-8pt\Hom_\AA(P_1\otimes P_1^* \otimes \cdots \otimes P_g \otimes P_g^*, \unit)^*\big) /\langle (f,\id)\sim (\id, f^*), f:P_i\to P_i'\rangle$$
    
On the other hand, by Theorem \ref{thm:SkCategTQFT}, and writing the handlebody $H_g$ as a ball $B^3$ with $g$ pairs of disks glued together, the admissible skein module $\Sk_\II(H_g)$ is obtained as the coend $$\Sk_\II(H_g) \simeq \int^{P_1,\dots, P_g \in \II}\hskip-5pt \Sk_\II(B^3; P_1,P_1^*,\dots,P_g,P_g^*) \simeq \int^{P_1,\dots, P_g \in \II}\hskip-5pt\Hom_\AA(P_1\otimes P_1^* \otimes \cdots \otimes P_g \otimes P_g^*, \unit)\ .$$ 
A morphism $\varphi$ in the RHS maps to the skein \quad
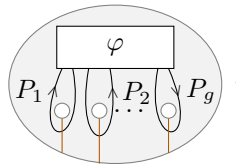
    \begin{tikzpicture}[baseline = 0pt, scale = 0.7]
        \fill[gray!10] (-2,0) arc(180:540:2 and 1.5);
        \draw[gray] (-2,0) arc(180:540:2 and 1.5);
        \node at (0.3,-0.5) {$\dots$};
        \fill[white] (-1.1, 0.3) rectangle (1.1, 1.1);
        \draw (-1.1, 0.3) rectangle (1.1, 1.1);
        \node at (0,0.7) {$\varphi$};
        \draw (-1,0.3) .. controls (-1.7,-1.3) and (-0.5,-1.3).. (-0.8,0.3) node[pos = 0.1, left]{$P_1$} node[pos = 0.1, sloped]{\tiny $>$};
        \draw (-0.5,0.3) .. controls (-0.8,-1.4) and (0.2,-1.4).. (-0.1,0.3)  node[pos = 0.9, right]{$P_2$}  node[pos = 0.9, sloped]{\tiny $<$};
        \draw (1,0.3) .. controls (1.7,-1.3) and (0.5,-1.3).. (0.8,0.3) node[pos = 0.1, right]{$P_g$} node[pos = 0.1, sloped]{\tiny $>$};
        \draw[orange!70!black] (-1,-0.5)--++(0,-0.8) node[pos = 0.7, right, text = orange!70!black]{};
        \draw[orange!70!black] (-0.3,-0.5)--++(0,-1) node[pos = 0.7, right, text = orange!70!black]{};
        \draw[orange!70!black] (1,-0.5)--++(0,-0.8) node[pos = 0.7, left, text = orange!70!black]{};
        \node[circle, draw=gray, fill=white, inner sep = 2pt] at (-1,-0.5) {};
        \node[circle, draw=gray, fill=white, inner sep = 2pt] at (-0.3,-0.5) {};
        \node[circle, draw=gray, fill=white, inner sep = 2pt] at (1,-0.5) {};
    \end{tikzpicture}\ .
    
    This is almost the same as the formula for $\DGGPR_\AA(\Sigma)$ above, though there are duals. We have an isomorphism $\Hom_\AA(P_1\otimes P_1^* \otimes \cdots \otimes P_g \otimes P_g^*, \unit)^* \simeq \Hom_\AA(P_1\otimes P_1^* \otimes \cdots \otimes P_g \otimes P_g^*, \unit)$ given by the modified trace paring, and noticing that by design $P_1\otimes P_1^* \otimes \cdots \otimes P_g \otimes P_g^*$ is self-dual. These isomorphisms preserve the $(f,\id)\sim (\id, f^*)$ relations, and induce an isomorphism on the quotient. Hence $\Sk_\II(H)$ and $\DGGPR_\AA(\Sigma)$ have the same dimension, and $\eta_{\Sigma,H}$ is an isomorphism when $H$ is a handle body.

    If $H$ is a disjoint union of handle bodies, then $\eta_{\Sigma,H}$ is still an isomorphism by monoidality.

    Now, consider a general bounding 3-manifold $M:\emptyset \to \Sigma$. Denote $H:\emptyset \to \Sigma$ a disjoint union of handle bodies. Any two 3-manifold with same boundary are related by a 4-cobordisms $W: M \Rightarrow H$. It can be thought of as a morphism $(\id_{\Sigma}, W)$ in $\cob_{2+1}^{filled}$ where the 3-cobordism part is the identity. It induces a map $\SS_\II(W)\circ_h \id_{\RR(\Sigma)}: \Sk_\II(M,\emptyset)\to \Sk_\II(H, \emptyset)$ which is an isomorphism because $\SS_\II$ is invertible, and an isomorphism $\SS_\II(\mathbb CP_2)^{\sigma(W)}\id: \DGGPR_\AA(\Sigma)\to \DGGPR_\AA(\Sigma)$. Naturality of $\eta$ implies that $\eta_{\Sigma,M}$ and $\eta_{\Sigma,H}$ are related by these isomorphisms, hence $\eta_{\Sigma,M}$ is an isomorphism.
\end{proof}

\section{WRT as a projective TQFT}\label{Sec_projectiveTQFT}
As argued in \cite{FreedAnomaly, JacksonVanDykeProjectiveTQFT}, the physical content of the TQFT described above is really the boundary condition $\RR_\II$. Let us recall the main ideas at play here. 

The intuitive notion that seems to better represent the idea of projectivity of quantum physics would be a symmetric monoidal functor $\ZZ$ from $\cob_{n+1}$ to the 2-category $\Proj$ of \cite[Appendix A]{FreedAnomaly}. It has objects vector spaces, 1-morphisms linear maps and a unique 2-morphism from $f$ to $g$ if $f = \lambda g$ for some $\lambda\in \Bbbk^\times$. It is symmetric monoidal with tensor product of vector spaces. The presence of these 2-morphisms authorize the functor $\ZZ$ to only preserve composition up to scalar. 
Freed and Van Dyke propose however that a more appropriate target category is the following.
\begin{definition}
    The bicategory $\mathbb P\Vect$ is the sub-bicategory of the arrow category $(\Bimod^{\unit \to})^{hop}$ of \cite{JFS} where the target is invertible.

    More explicitly, it has objects pairs $f=(f^\#,f^t)$ where $f^t \in \Bimod^\times$ is an invertible object (under tensor product) and $f^\#:\unit \to f^t$ is a profunctor. It has 1-morphisms from $f$ to $g$ pairs $h= (h^\#,h^t)$ where $h^t:g^t \to f^t$ is an invertible profunctor, and $h^\#:  f^\# \Rightarrow h^t\circ g^\#$ is a linear natural transformation, so in a diagram:\begin{equation*}
        \begin{tikzcd}
            \unit \ar[r, "f^\#"] \ar[d, equal] & f^t \\
            \unit \ar[r, "g^\#"] & g^t \ar[u, "h^t"']
            \arrow[from = 1-1, to = 2-2, Rightarrow, "h^\#", shorten = 5pt]
        \end{tikzcd}
    \end{equation*} It has 2-morphisms from $h_1$ to $h_2$ natural isomorphisms $\eta: h_1^t \to h_2^t$ transporting $h_2^\#$ to $h_1^\#$. 
    
    It contains $\Vect \simeq (\Bimod^{\unit\to\unit})^{hop}$ as the sub-bicategory where the target part $(-)^t$ is trivial. We denote $\mathbb P:\Vect\to \mathbb P\Vect$ the induced functor, which we think of a quotient where we have identified the linear maps that differ by a scalar by invertible 2-morphisms.
    
    A \textbf{projective $(n+1)$-TQFT} is a symmetric monoidal 2-functor 
    $$\ZZ: \cob_{n+1}\to \mathbb P \Vect$$
    where we see $\cob_{n+1}$ as a 2-category with only identity 2-morphisms.
    
    The \textbf{anomaly} of a projective TQFT $\ZZ$ is the invertible categorified $(n+1)$-TQFT $$\alpha: \cob_{n+1}\overset{\ZZ}{\to} \mathbb P \Vect \overset{t}{\to} \Bimod^{hop}$$
    where $t:\Bimod^{\unit \to} \to \Bimod$ is the target functor. 

    Let $\widetilde\cob_{n+1}\overset{\pi}{\to} \cob_{n+1}$ be a symmetric monoidal functor. We say that $\widetilde\ZZ: \widetilde\cob_{n+1} \to \Vect$ is a \textbf{resolution of $\ZZ$} if \begin{equation*}
    \begin{tikzcd}
        \widetilde\cob_{n+1} \ar[d, "\pi"] \ar[r, "\widetilde\ZZ"] & \Vect \ar[d, "\mathbb P"] \\
        \cob_{n+1} \ar[r, "\ZZ"] & \mathbb P\Vect
    \end{tikzcd}
    \end{equation*}
commutes up to symmetric monoidal pseudo-natural isomorphism.

Note that by definition a projective TQFT $\ZZ$ with anomaly $\alpha$, i.e. a functor valued in the arrow category with prescribed source $\unit$ and target $\alpha$, is the same data as an oplax transformation $\Triv \Rightarrow \alpha$ to the categorified TQFT $\alpha$ contravariant on 1-morphisms from Definition \ref{def:BondaryCond}. Hence, a projective TQFT is a particular case of a boundary condition, specifically it is a boundary condition to an invertible theory $\alpha$.
\end{definition}
\begin{remark}
    There is an equivalence of symmetric monoidal bicategory $\Proj \to \mathbb P\Vect$ which maps a vector space $V$ to the pair $(V, \unit)$ where $V$ is seen as a profunctor $\unit \otimes \unit \to \Vect$, sends a linear map $f: V \to W$ to the pair $(f, \id_\unit)$ where $f$ is seen as a natural transformation $V \to W$, and sends a 2-morphism $\lambda:f \to \lambda f$ to the natural transformation $\lambda: \id_\unit \to \id_\unit$. 
    
    It is essentially surjective because any invertible object $f^t \in \Bimod^\times$ is isomorphic to the unit. It is locally essential surjective because any invertible vector space is isomorphic to $\Bbbk$. It is locally fully faithful because the 2-morphisms are only scalars on both sides.

    An inverse equivalence $\mathbb P\Vect \to \Proj^{op}$ amounts to a trivialization of the target part of $\mathbb P\Vect$, i.e. a trivialization of the anomaly. 
\end{remark}
\begin{remark}
    There should be a general definition of $\mathbb P n\hskip-2pt\Vect$ as a subcategory of the arrow category $(n+1)\hskip-2pt\Vect^{\unit \to}$ where the target is invertible, see \cite[Hypothesis P]{JacksonVanDykeProjectiveTQFT}. For $n\geq 3$, the canonical map $n\hskip-2pt\Vect \to \mathbb P n\hskip-2pt\Vect$ may no longer be essentially surjective, as not every invertible object of $n\hskip-2pt\Vect$ need be isomorphic to the unit. It is expected that Reshetikhin--Turaev theories that are not of Turaev--Viro type will land outside of the essential image of $3\hskip-2pt\Vect \to \mathbb P3\hskip-2pt\Vect$. This explains why, even though as a projective theory they are fully extended, they cannot be resolved all the way down to the point. 
\end{remark}
We observed in Theorems \ref{thm:WRT} and \ref{thm:DGGPR} that the WRT theory is obtained as a composition of a boundary condition and the invertible 4-TQFT on a bounding manifold. In the setting of projective TQFTs, we understand this second part as a trivialization of the anomaly of the projective theory given by the boundary condition. This can be reformulated as follows.
\begin{theorem}\label{thm:WRTprojective}
Let $\AA$ be a semisimple modular tensor category with a chosen square root of its global dimension. Then $\RR_\AA: \cob_{2+1} \to \mathbb P\Vect$ is a projective theory, and the WRT theory $\WRT_\AA: \widetilde\cob_{2+1}\to \Vect$ is a resolution of $\RR_\AA$. 

Let $\AA$ be a non-semisimple modular category with a chosen modified trace with global dimension equal to 1. Then $\RR_\II: \cob_{2+1}^{nc} \to \mathbb P\Vect$ is a non-compact projective theory, and the DGGPR theory $\DGGPR_\AA: \widetilde\cob_{2+1}^{nc}\to \Vect$ is a resolution of $\RR_\II$.
\end{theorem}
\begin{proof}
First, $\RR_\II$ is a projective theory as its target is given by the truncation of the once-extended 4-TQFT $\SS_\II$  to $\cob_{2+1}$ which is indeed invertible when $\AA$ is modular by Proposition \ref{Prop_ZdepOnSignature}.

    We need to prove that \begin{equation*}
    \begin{tikzcd}
        \widetilde\cob_{2+1}^{nc} \ar[d, "\pi"] \ar[r, "\DGGPR_\AA"] & \Vect \ar[d, "\mathbb P"] \\
        \cob_{2+1}^{nc} \ar[r, "\RR_\II"] & \mathbb P\Vect
    \end{tikzcd}
    \end{equation*} commutes up to a symmetric monoidal pseudo-natural isomorphism $\delta: \mathbb P \circ \DGGPR_\AA \Rightarrow \RR_\II\circ\pi$. 
    
    For $(\Sigma,L)\in \widetilde\cob_{2+1}$, choose arbitrarily a 3-manifold $H$ bounding $\Sigma$ and inducing the Lagrangian $L$. It gives an isomorphism $\SS_\II(H): \RR_\II(\Sigma)^t =\SS_\II(\Sigma)\to \unit$, and we set $\delta_{\Sigma,L}^t := \SS_\II(H)$. By Theorem \ref{thm:DGGPR}, we have an isomorphism $\eta_{\Sigma,H}: \DGGPR_\AA(\Sigma,L) \to \Aa_\II(\Sigma,H) = \delta^t_{\Sigma, L}\circ \RR_\II(\Sigma)^\#$, and we set $\delta_{\Sigma,L}^\# := \eta_{\Sigma,H}$. In a diagram, we have\begin{equation*}
        \begin{tikzcd}
            \unit \ar[rr, "\DGGPR_\AA(\Sigma)"] \ar[d, equal] & & \unit \\
            \unit \ar[rr, "\RR_\II(\Sigma)"'] & & \SS_\II(\Sigma) \ar[u, "\SS_\II(H)"']
            \arrow[from = 1-1, to = 2-3, Rightarrow, "\eta_{\Sigma,H}", shorten = 5pt]
        \end{tikzcd}
    \end{equation*}    
    For $(M,\sigma): (\Sigma,L)\to(\Sigma',L')$ a 1-morphism, choose arbitrarily a 4-manifold $W: H\underset{\Sigma}{\cup}M \Rightarrow H'$ with signature $\sigma$ and set
    $$\delta_{M,\sigma} := \SS_\II(W): \delta_{\Sigma',H'}^t\circ \RR_\II(M)^t\Rightarrow \delta_{\Sigma,H}^t \ .$$ This does define a 2-morphism in the arrow category by the naturality of $\eta$ proven in Theorem \ref{thm:DGGPR}.
\end{proof}
\begin{remark}
    The topological content of a projective theory is not quite the same as that of a usual TQFT. Indeed, as every map is considered up to multiplication by a scalar, there is no underlying 3-manifold invariant to a projective 3-TQFT. Actually, in view of Theorem \ref{thm:WRTprojective}, we notice something even more surprising: in the definition of $\RR_\II$ we never used the Kirby color (it is only used in the definition of $\SS_\II$ on the 4-dimensional 2-handle) which is the core of the topological content of WRT theories. The projective theory is blind to this part of the topological content, which we only access because we have a good geometric way of trivializing the anomaly, namely choosing a bounding 4-manifold with prescribed signature. 

    However, if the projective theory does not know about the invariants of 3-manifolds, it does know about the projective representations of the mapping class groups of surfaces. 
\end{remark}

\small
\bibliography{mybib.bib}
\bibliographystyle{alpha}
\end{document}